\documentclass[reqno]{amsart}
\usepackage{amsmath,amsfonts,amssymb,amsthm}
\usepackage[usenames,dvipsnames]{xcolor}
\usepackage{graphicx}
\usepackage[colorlinks=true]{hyperref} %para PDFLaTeX
\usepackage{url}
\hypersetup{linkcolor=Violet,citecolor=PineGreen}
\newtheorem{teor}{Theorem}[section]

\newtheorem{prop}[teor]{Proposition}
\newtheorem{coro}[teor]{Corollary}
\theoremstyle{definition}
\newtheorem{defi}[teor]{Definition}

\newtheorem{nota}[teor]{Remark}

\newtheorem{notas}[teor]{Remarks}
\numberwithin{equation}{section}
\newcommand{\N}{\mathbb{N}}

\newcommand{\R}{\mathbb{R}}
\renewcommand{\S}{\mathbb{S}}

\newcommand{\Y}{\mathbb{Y}}

\newcommand{\mA}{\mathcal{A}}
\newcommand{\mB}{\mathcal{B}}
\newcommand{\mC}{\mathcal{C}}
\newcommand{\mD}{\mathcal{D}}

\newcommand{\mI}{\mathcal{I}}
\newcommand{\mJ}{\mathcal{J}}
\newcommand{\mK}{\mathcal{K}}

\newcommand{\mM}{\mathcal{M}}
\newcommand{\mN}{\mathcal{N}}
\newcommand{\mO}{\mathcal{O}}

\newcommand{\mR}{\mathcal{R}}
\newcommand{\mS}{\mathcal{S}}
\newcommand{\mT}{\mathcal{T}}
\newcommand{\mU}{\mathcal{U}}
\newcommand{\mV}{\mathcal{V}}

\newcommand{\mX}{\mathcal{X}}

\newcommand{\mminv}{\mathfrak{M}_{\text{\rm{inv}}}}
\newcommand{\mmerg}{\mathfrak{M}_{\text{\rm{erg}}}}
\newcommand{\alfaa}{\alpha_\mA}
\newcommand{\betaa}{\beta_\mA}
\newcommand{\alfam}{\alpha_\mM}
\newcommand{\betam}{\beta_\mM}

\newcommand{\malfa}{\mM^\alpha}
\newcommand{\mbeta}{\mM^\beta}
\newcommand{\ep}{\varepsilon}
\newcommand{\lb}{\lambda}
\newcommand{\pu}{{\cdot}}

\newcommand{\W}{\Omega}

\newcommand{\w}{\omega}
\newcommand{\ww}{\wit\omega}
\newcommand{\bw}{\bar\omega}

\newcommand{\wt}{\w\pu t}
\newcommand{\ws}{\w\pu s}
\newcommand{\wwt}{\ww\pu t}

\newcommand{\wit}{\widetilde}

\newcommand{\n}[1]{\|#1\|}

\newcommand{\lsm}{\left[\begin{smallmatrix}}
\newcommand{\rsm}{\end{smallmatrix}\right]}

\newcommand{\dist}{\text{\rm dist}}

\DeclareMathOperator{\Supp}{Supp}

\begin{document}
\title[Stability and chaos in nonhomogeneous linear dissipative ODEs]
{Uniform stability and chaotic dynamics in nonhomogeneous linear dissipative
scalar ordinary differential equations}
\author[J.~Campos]{Juan Campos}
\author[C.~N\'{u}\~{n}ez]{Carmen N\'{u}\~{n}ez}
\author[R.~Obaya]{Rafael Obaya}
\address[J. Campos]
{Departamento de Matem\'{a}tica Aplicada, Universidad de
Granada, Avenida de la Fuente Nueva S/N, 18071 Granada, Spain}
\address[C.~N\'{u}\~{n}ez and R.~Obaya]
{Departamento de Matem\'{a}tica Aplicada, Universidad de
Va\-lladolid, Paseo del Cauce 59, 47011 Valladolid, Spain}
\address[]{}
\email[Juan Campos]{campos@ugr.es}
\email[Carmen N\'{u}\~{n}ez]{carmen.nunez@uva.es}
\email[Rafael Obaya]{rafoba@wmatem.eis.uva.es}
\thanks{Juan Campos is partly supported by the the MINECO-Feder (Spain)
project RTI2018-098850-B-I00, and by the Junta de Andaluc\'{\i}a (Spain)
under projects PY18-RT-2422 and B-FQM-580-UGR20.
Carmen N\'{u}\~{n}ez and Rafael Obaya are partly supported Ministerio de Ciencia e
Innovaci\'{o}n (Spain) under project PID2021-125446NB-100, and by the University
of Valladolid under project PIP-TCESC-2020.}
\date{}
\begin{abstract}
The paper analyzes the structure and the inner long-term dynamics
of the invariant compact sets for the skewproduct flow induced by a
family of time-dependent ordinary differential equations
of nonhomogeneous linear dissipative type. The main assumptions
are made on the dissipative term and on the homogeneous linear term of the equations.
The rich casuistic includes the uniform stability of the invariant compact sets,
as well as the presence of Li-Yorke chaos and Auslander-Yorke chaos inside the attractor.
\end{abstract}
\maketitle
%%%%%%%%%%%%%%%%%%%%%%%%%%%%%%%%%%%%%%%%%%%%%%%%%%%%%%%%%%%%%%%%%%%%%%%%%%%%%%%%%%%%%
%%%%%%%%%%%%%%%%%%%%%%%%%%%%%%%%%%%%%%%%%%%%%%%%%%%%%%%%%%%%%%%%%%%%%%%%%%%%%%%%%%%%%
%%%%%%%%%%%%%%%%%%%%%%%%%%%%%%%%%%%%%%%%%%%%%%%%%%%%%%%%%%%%%%%%%%%%%%%%%%%%%%%%%%%%%
%%%%%%%%%%%%%%%%%%%%%%%%%%%%%%%%%%%%%%%%%%%%%%%%%%%%%%%%%%%%%%%%%%%%%%%%%%%%%%%%%%%%%
\section{Introduction}
The mathematical literature collects many different notions of chaos,
all of which share a common target: each definition takes into account
different properties of the long-term behavior of the system under study,
which, combined, imply the unpredictability of the dynamics due to divergence
of initially nearby orbits. There are also many different approaches to the
concept of stability for dynamical systems, but in this case
the subjacent idea is clearer and more globally accepted: initially nearby
orbits remain close. Hence it seems correct to say that, at least to some extent,
chaos and stability are opposite terms.
\par
This work concerns the long-term dynamics of a quite precise
mathematical model for which both situations (chaos and uniform stability)
are possible. Our dynamical system is generated by the solutions
of the family of nonautomonous (in the sense of time-dependent)
scalar dissipative ordinary differential equations
\begin{equation}\label{1.familia}
 x'=a(\wt)\,x+b(\wt)+g(\wt,x)\,,\qquad \w\in\W\,,
\end{equation}
where $\W$ is a compact metric space, $\sigma\colon\R\times\W\to\W\,,
(t,\w)\mapsto \sigma(\wt)=:\wt$ defines a minimal flow on $\W$,
$a,b\colon\W\to\R$ are continuous functions, and $g\colon\W\times\R\to\R$
is a smooth dissipative term. The analysis is made under the assumptions
$\int_\W a(\w)\,dm\le 0$ for any $\sigma$-ergodic measure on $\W$ and
decreasing behavior of $g$ with respect to the state variable $x$.
\par
We will consider two cases. The first one occurs
when the dissipation is negligable as long as the state remains in $[r_1,r_2]$
(since $g$ vanishes at the set $\W\times[r_1,r_2]$ with $r_1<r_2$) and,
at the same time, the dissipation is active and dominant with respect to the linear
term outside that set of states. Since the restriction of the equation to
$\W\times[r_1,r_2]$ is linear and nonhomogeneous, we say that \eqref{1.familia}
provides a nonautonomous nonhomogeneous linear dissipative model.
This is the case more interesting for our analysis, since the casuistic is
richer. The second case, which we will call purely dissipative,
occurs when $g$ vanishes exactly at the points of $\W\times\{r\}$
(so that $r_1=r_2=r$), which in general makes simpler the structure
of the attractor.
\par
The family \eqref{1.familia} generates the skewproduct flow
\[
 \tau\colon\mU\subseteq\R\times\W\times\R\to\W\times\R\,,\quad
 (t,\w,x_0)\mapsto (\wt,x(t,\w,x_0))\,,
\]
where $\mI_{\w,x_0}\to\R\,,\;t\mapsto x(t,\w,x_0)$ is the maximal solution
of the equation \eqref{2.ecescalar} corresponding to $\w$
with $x(0,\w,x_0)=x_0$, and $\mU$ is the open set
$\bigcup_{(\w,x_0)\in\W\times\R}\mI_{\w,x_0}$.
\par
The analysis of a family of equations like \eqref{1.familia}, or,
more generally, of the type $x'=f(\wt,x)$, is a classical tool in the
analysis of a single nonautonomous differential equation
$x'=f_0(t,x)$. Under some regularity conditions on $f_0$ which the translated
functions $f_t(s,x):=f_0(t+s,x)$ inherit, the {\em hull of $f_0$},
given by the closure in the compact open topology on $C(\R^2,\R)$
of the set $\{f_t\,|\;t\in\R\}$, turns out to be a compact metric space
$\W$, and the time-translation $\R\times\W\to\W,\;(t,\w)\mapsto\wt:=\w_t$
defines a global continuous flow. By representing $f(\w,x):=\w(0,x)$ we
obtain a family $x'=f(\wt,x)$ (i.e., $x'=\w(t,x)$) which includes the initial
equation. The function $f_0$ is {\em time-recurrent} if the flow on its hull
$\W$ is minimal, as we assume in this paper. This is for instance the case if
$f_0$ is, roughly speaking, almost periodic in $t$ uniformly
in $x$; but a minimal hull may come from other types of functions.
By being a bit more careful in the hull construction, we obtain
a family of the type \eqref{1.familia} if the starting point is
$x'=a_0(t)\,x+b_0(t)+g_0(t,x)$.
This collective formulation allows us to use techniques of
topological dynamics and ergodic theory in the analysis of the
long-term behavior of the orbits of the flow $\tau$, which include
the graphs of the solutions of the initial equation.
In this paper, we choose the (more general) approach of
not to assume that $\W$ is the hull of an initial function.
\par
The dissipative character of $\tau$, due to the hypothesis on $g$,
implies the existence of a global
attractor $\mA$. Our main objective is the description of the structure
and internal dynamics of the compact invariant subsets $\mK\subseteq\mA$.
In some cases, the presence of chaos is precluded:
there appear uniformly exponentially stable sets
on which the dynamics reproduces that of $(\W,\sigma)$,
or sets $\mK$ which are uniformly (not exponentially) stable.
But, in other cases (in the linear dissipative case), we find
compact invariant subsets $\mK\subseteq\mA$
on which the dynamics is highly complex, with the possible occurrence of
different types of chaos. This phenomenon (which cannot occur if
the functions $a,\,b $ and $g$ of \eqref{1.familia} are
autonomous or time-periodic) shows that unpredictability
can be a natural and expected ingredient in the dynamics
of simple nonautonomous mathematical models, which in general are
better adapted to the real world than the autonomous ones.
\par
Many of the notions of chaos on an invariant
compact subset require a positive upper Lyapunov exponent for the corresponding
linearized system, in order to obtain an exponential rate of
divergence of the forward orbits starting nearby in the phase space.
This behavior is not possible under the assumptions we make
on \eqref{1.familia}, which we will precise in Section \ref{3.sec}.
But some of the notions of chaos do not require this condition. In this paper, we
describe some conditions on the function $a$ of the
linear part or the equations which imply, in one of the possible
dynamical situations, the presence of Li-Yorke
chaos and of Auslander-Yorke chaos in a \lq\lq large part"~of the attractor.
These notions of chaos are compatible with
null upper Lyapunov exponents of the compact invariant sets
on which the chaos appears. Roughly speaking, Li-Yorke chaos \cite{liyo}
appears on a compact invariant set when this set contains an uncountable subset
of points such that any pair of them is Li-Yorke chaotic; i.e.,
it gives rise to two orbits which approach each other and
separate from each other alternatively on infinitely many
intervals of time becoming indistinguishable.
The notion of Li-Yorke chaos was introduced in \cite{liyo} in 1975
for transformations, and it is easily adapted to (semi)flows.
The interested reader can find in \cite{bgkm}, \cite{akko}, \cite{koly},
and the many references therein, some dynamical properties
associated to Li-Yorke chaos and its relation with
other notions of chaotic dynamics.
Auslander-Yorke chaos \cite{auyo}
occurs on topologically transitive flows on compact metric
spaces when the flow is sensitive with respect to initial
conditions. The idea was trying to capture some representative properties of
the notion of turbulence of fluids given by Ruelle and Takens in \cite{ruta}.
The abstract formulation of \cite{auyo} makes this notion applicable to
a much more general dynamical framework.
Among the many works devoted to characterize this type of chaos and to analyze
its dynamical consequences, as well as to establish connections
and differences with other types of chaos, we mention \cite{glwe4}
(which is central to our approach in this paper), \cite{gljk},
\cite{mila}, and references therein.
\par
In the rest of this introduction, we describe the structure of the
paper, which is organized in two sections,
as well as the main dynamical properties which we prove.
\par
Section \ref{2.sec} is a long preliminary section,
divided in seven parts. Its length is due in part to the many different
concepts and already known properties
needed for the statements and proofs of our main results. First,
we recall basic and (more or less) standard notions
on topological dynamics, ergodic theory, skewproduct flows,
stability, dissipativity and global attractors, exponential dichotomy,
Lyapunov exponents, Sacker and Sell spectrum, hyperbolicity of minimal
subsets\ldots~And then we continue with the description of the
less known nonempty set $\mR_m$ of those maps $a\colon\W\to\R$ which
will allow us to detect the
presence of chaotic invariant subsets, and with the definitions
and basic properties of Li-Yorke chaos (also in measure in the case of a
skewprodut flow) and of Auslander-Yorke chaos. The subindex $m$ refers to
a $\sigma$-invariant measure on which the definition of the set $\mR_m$ depends.
\par
The structure of this preliminary section
is better described at its first paragraphs. We point out here that,
in addition to this large number of notions and already known properties,
Section \ref{2.sec} includes the detailed proofs of three new results,
fundamental to our purposes. The first one, Theorem \ref{2.teornovacio},
shows that the sets $\mR_m$ are nonempty if the flow on $\W$ is non periodic,
and contain functions
with null Sacker and Sell spectrum. The second one, Theorem \ref{2.teordensidad},
refers to some extra properties of the maps $a\in\mR_m$, which will allow
us to emphasize that the Li-Yorke chaos which we detect
is \lq\lq quite more chaotic"~than what the initial definition
requires. We will explain this better in due time.
The third result, Theorem \ref{2.teorAYsec3}, determines a series of
compact subsets which are appropriate to detect the presence of Auslander-Yorke
chaos, given by the supports of certain ergodic measures.
\par
The main results of the paper are stated and proved in Section \ref{3.sec},
which begins with the precise description of the conditions imposed
on the dissipative term $g$ of \eqref{1.familia}:
different degrees of smoothness, vanishing set given by $\W\times[r_1,r_2]$,
dissipativity character, and (strictly or not) decreasing behavior outside
$\W\times[r_1,r_2]$. The last condition is not needed in our first three results.
Theorem \ref{3.teorexisteA} establishes the existence of
a global attractor, which thanks to the minimality assumed of the base flow
takes the shape
\[
 \mA=\bigcup_{\w\in\W}\big(\{\w\}\times[\alfaa(\w),\betaa(\w)]\big)\,,
\]
for two semicontinuous functions $\alfaa,\betaa\colon\W\to\R$
with $\tau$-invariant graphs. Theorem \ref{3.teortapas}
analyzes the properties of two minimal sets, $\malfa$ and $\mbeta$,
associated to the covers of $\mA$ (a tool for our main results),
and Theorem \ref{3.teortodos-} shows
that the unique situation in which all the minimal sets have
negative Lyapunov exponent is that of existence of a unique minimal set,
which is given by the uniformly
exponentially stable graph of a continuous function
$\eta\colon\W\to\R$, and which coincides with the attractor.
\par
With the condition on the monotonicity of $g$ in force from now on,
we first analize the dynamical situation arising when $\int_\W a(\w)\,dm<0$
for every $\sigma$-ergodic measure $m$: Theorem \ref{3.teornegativo}
shows that the upper Lyapunov exponent of every minimal sets is
negative, so that the situation is that of the end of the previous
paragraph.
\par
The rest of the paper analyzes the situation occurring when
$\int_\W a(\w)\,d\wit m\le 0$ for every $\sigma$-ergodic measure $\wit m$
and there exists one, say $\wit m$, with $\int_\W a(\w)\,dm= 0$.
Two global dynamical possibilities arise. The first one, which can only
occur if $r_1<r_2$, corresponds to the existence of
infinitely many minimal sets. All of them are contained in $\W\times[r_1,r_2]$
and are given by the graphs of the functions $\eta_c=c\,\alfaa+(1-c)\,\betaa$
for $c\in[0,1]$, which are continuous; and the union of all these
minimal sets, which are uniformly stable, form the global attractor.
Theorem \ref{3.teorSup0-lamin} explores this situation.
The second possibility, richer in casuistic, arises when
$\malfa=\mbeta$ is the unique $\tau$-minimal set, which is not necessarily
a copy of the base, and which may or may not coincide with the global attractor.
In particular, the global attractor is a pinched set; that is, its
section over the base reduces to a singleton for at least one element of $\W$.
These properties and some of their dynamical consequences are described
in Theorem \ref{3.teorSup0-pinched}.
\par
When, in addition, the family is linear dissipative and $a\in\mR_m$,
Li-Yorke chaos and Auslander-Yorke chaos may appear, as we explain
in Theorems \ref{3.teorSup0-chaosLY} and \ref{3.teorSup0-chaosAY}.
More precisely, if under these conditions the unique minimal set
is contained in $\W\times[r_1,r_2]$ and at least of one of its covers
is at a positive distance from $\W\times(\R-[r_1,r_2])$,
then the attractor is \lq\lq strongly"~Li-Yorke chaotic, in the following
sense: there exists a subset $\W_{LY}\subset\W$ with full measure
$m$ such that, for every $\w\in\W_{LY}$, any two points of
$\{\w\}\times[\alfaa(\w),\betaa(\w)]$ form a Li-Yorke chaotic pair.
Moreover, making use of the above mentioned Theorem \ref{2.teordensidad},
we explore the internal dynamics in $\mA$ in order to confirm the
physical observability of the Li-Yorke chaos, and hence its potential
relevance in applications. More precisely, we will prove the positive density
in $\R$ of two sets of times for $m$-almost point of the base:
those at which the forward orbits associated
to every Li-Yorke chaotic pairs (sharing the base point)
are \lq\lq clearly separated", and
those at which these orbits are \lq\lq as close as desired".
\par
Finally, under the same hypotheses, we detect Auslander-Yorke chaos in
infinitely many invariant compact subsets $\mS_c\subset\mA$ for every $c\in[0,1]$
excepting, perhaps, a particular value $c_0$. These (also pinched) sets
are transitive: they admit a dense forward semiorbit.
Besides this, the union $\wit\mS$ of all these sets is a chain recurrent
set, supporting an invariant measure $\wit\mu$,
composed by sensitive points, and with a dense subset of
generic points. These properties can be understood as
a weak version of the classical notion of
chaos introduced by Devaney in \cite{deva}. In addition,
$\wit\mS$ fills an \lq\lq important part"~of
$\mA$, which shows that also this chaotic phenomenon has physical relevance.
%%%%%%%%%%%%%%%%%%%%%%%%%%%%%%%%%%%%%%%%%%%%%%%%%%%%%%%%%%%%
%%%%%%%%%%%%%%%%%%%%%%%%%%%%%%%%%%%%%%%%%%%%%%%%%%%%%%%%%%%%
%%%%%%%%%%%%%%%%%%%%%%%%%%%%%%%%%%%%%%%%%%%%%%%%%%%%%%%%%%%%
%%%%%%%%%%%%%%%%%%%%%%%%%%%%%%%%%%%%%%%%%%%%%%%%%%%%%%%%%%%%
\section{Preliminaries}\label{2.sec}
This long preliminary section is organized in seven parts. The
first four contain general results, required in Section \ref{3.sec}
for the description of the global dynamics for the equations
of the Introduction. The last three, less standard,
present concepts, known properties, and new results which will
be used to analyze the possible presence of chaotic behavior.
\par
The basic concepts and properties of
topological dynamics and measure theory, with special focus on
skewproduct flows defined from a family of scalar nonautonomous
ordinary differential equations, are summarized in the first
two subsections, where we will also fix some notation. Good
references for their contents are \cite{nest}, \cite{elli},
\cite{sase2,sase1}, \cite{walt}, \cite{mane}, \cite{shyi4},
and references therein.
\par
As explained in the Introduction, our main results are formulated
under different assumptions on the linear homogenous
component of the family of equations. In Subsection \ref{2.subseclin}
we summarize the required notions and properties concerning
exponential dichotomy, Sacker and Sell spectrum, and Lyapunov exponents,
which can be found in \cite{copp1} and \cite{jops}.
Subsection \ref{2.subsecescalar} recalls some particular properties
of minimal sets for a skewproduct flow in the scalar case, and includes,
for the reader's convenience, a proof of a classical result relating the
uniform exponential stability of these minimal sets with the
sign of their Lyapunov exponents.
\par
In Subsection \ref{2.subsecR} we introduce a set of continuous
functions which will provide us with an adequate framework to
detect the presence of the two types of chaos mentioned in the Introduction:
Li-Yorke chaos, described in Subsection \ref{2.subsecLY}, and Auslander-Yorke
chaos, described in Subsection \ref{2.subsecAY}. As we mentioned in
the introduction, besides basic concepts and known properties, Subsections
\ref{2.subsecR} and \ref{2.subsecAY} present some new results
which we will use in Section \ref{3.sec} but which are valid
for a setting more general than that there considered. The contents
of these results are explained in the corresponding subsections.
%%%%%%%%%%%%%%%%%%%%%%%%%%%%%%%%%%%%%%%%%%%%%%%%%%%%%%%%%%%%
%%%%%%%%%%%%%%%%%%%%%%%%%%%%%%%%%%%%%%%%%%%%%%%%%%%%%%%%%%%%
\subsection{Basic concepts on flows}\label{2.subsecbasic}
Let $\W$ be a complete metric space, and let $\dist_\W$ be
the distance on $\W$. A ({\em real and continuous})
{\em flow\/} on $\W$ is a continuous map
$\sigma\colon\R\!\times\!\W\to\W,\; (t,\w)\mapsto\sigma(t,\w)$
such that $\sigma_0=\text{Id}$ and $\sigma_{s+t}=\sigma_t\circ\sigma_s$
for each $s,t\in\R$, where $\sigma_t(\w):=\sigma(t,\w)$.
The flow is {\em local\/} if the map $\sigma$
is defined, continuous, and satisfies the previous properties on an open subset
of $\R\!\times\!\W$ containing $\{0\}\!\times\!\W$.
\par
Let $\mU\subseteq\R\!\times\!\W$ be the domain of the map $\sigma$. The set
$\{\sigma_t(\w)\,|\;(t,\w)\in\mU\}$ is
the $\sigma$-{\em orbit\/} (or simply {\em orbit})
of the point $\w\in\W$. This orbit
is {\em globally defined\/} if $(t,\w)\in\mU$ for all $t\in\R$.
Restricting the time to $t\ge 0$ or
$t\le 0$ provides the definition of {\em forward\/}
or {\em backward $\sigma$-semiorbit}.
A Borel subset $\mC\subseteq \W$ is {\em $\sigma$-invariant\/}
if it is composed by globally defined orbits; i.e., if
$\sigma_t(\mC):=\{\sigma(t,\w)\,|\;\w\in\mC\}$
is defined and agrees with $\mC$ for every $t\in\R$.
A $\sigma$-invariant subset $\mM\subseteq\W$ is {\em $\sigma$-minimal\/}
(or simply {\em minimal}) if it is compact
and does not contain properly any other compact $\sigma$-invariant set;
or, equivalently, if each one of the two semiorbits of anyone of
its elements is dense in it. The flow $(\W,\sigma)$ is
{\em minimal\/} if $\W$ itself is minimal.
If the semiorbit $\{\sigma_t(\w_0)\,|\;t\ge 0\}$ is globally defined and
relatively compact, then the {\em omega limit set\/} of $\w_0$, which
we represent by $\mO_\sigma(\w_0)$, is given
by the points $\w\in\W$ such that $\w=\lim_{n\to \infty}\sigma_{t_n}(\w_0)$
for some sequence $(t_n)\uparrow\infty$.
This set is nonempty, compact, connected and $\sigma$-invariant.
By taking sequences $(t_m)\downarrow-\infty$ we obtain the definition of
the {\em alpha limit set\/} of $\w_0$, with analogous properties.
\par
Assume now that $\sigma$ is globally defined.
The flow is {\em equicontinuous\/} if given $\ep>0$
there exists $\delta>0$ such that
$\sup_{t\in\R}\dist_\W(\sigma_t(\w_1),\sigma_t(\w_2))<\ep$
whenever $\dist_\W(\w_1,\w_2)<\delta$. If $\W$ is a compact metric space,
equicontinuity is equivalent to {\em almost periodicity} (as proved in \cite{elli}).
A flow $(\W,\sigma)$ defined on a compact metric space $\W$
is {\em chain recurrent\/} if given $\varepsilon>0$, $t_0>0$, and
points $\w,\,\ww\in\W$, there exist points $\w_0:=\w,\,\w_1,\ldots,\,
\w_m:=\ww$ of
$\W$ and real numbers $t_1>t_0,\ldots,t_{m-1}>t_0$ such that
$\dist_\W(\sigma_{t_i}(\w_i),\w_{i+1})<\varepsilon$ for $i=0,\ldots,m-1$.
It is easy to check that minimality implies chain recurrence, and that
if $(\W,\sigma)$ is chain recurrent then $\W$ is connected.
\par
Let $m$ be a Borel measure on $\W$; i.e., a
regular measure defined on the Borel sets. (Any measure appearing in this paper
is of this type.) The measure is {\em concentrated on $\mB\subseteq\W$} if
$m(\W-\mB)=0$.
Its ({\em topological\/}) {\em support\/}, $\Supp m$, is the complement of
the biggest open set with null measure. In particular, it is contained in
any closed set $\mC$ on which the measure is concentrated; and if $\W$ is compact
then $\Supp m$ is compact. The measure $m$ is {\em $\sigma$-invariant\/}
if $m(\sigma_t(\mB))=m(\mB)$ for every Borel subset
$\mB\subseteq\W$ and every $t\in\R$. In this case, $\Supp m$ is
$\tau$-invariant;
and if $\W$ is minimal, then $\Supp m=\W$. Suppose that $m$ is
{\em finite\/} and {\em normalized}; i.e., that $m(\W)=1$. Then it is
{\em $\sigma$-ergodic\/} if it is $\sigma$-invariant and,
in addition, $m(\mB)=0$ or $m(\mB)=1$ for every $\sigma$-invariant subset
$\mB\subseteq\W$.
The sets of normalized $\sigma$-invariant and $\sigma$-ergodic measures are
represented by $\mminv(\W,\sigma)$ and $\mmerg(\W,\sigma)$.
If $\W$ is compact, there exists at least an element in $\mmerg(\W,\sigma)$.
Any equicontinuous minimal flow $(\W,\sigma)$
is {\em uniquely ergodic}, that is, $\mminv(\W,\sigma)$
reduces to just one element: a $\sigma$-ergodic measure.
\par
Let $\W$ be a compact metric space.
A Borel set $\mB\subseteq\W$ has {\em full measure for a measure
$m\in\mminv(\W,\sigma)$} if $m(\mB)=1$, and it has {\em complete measure\/} if
$m(\mB)=1$ for any $m\in\mminv(\W,\sigma)$.
A point $\w_0\in\W$ is {\em $\sigma$-generic\/} if $\lim_{t\to\infty}(1/t)\int_0^t
f(\sigma_s(\w_0))\,ds$ exists for every $f\in C(\W,\R)$.
In this case, Riesz representation theorem provides a mesure
$m_{\w_0}\in\mminv(\W,\sigma)$ such that $\lim_{t\to\infty}(1/t)\int_0^t
f(\sigma_s(\w_0))\,ds=\int_\W f(\w)\,dm_{\w_0}$ for every $f\in C(\W,\R)$.
In addition, the sets $\wit\W$ of $\sigma$-generic points and the subset
$\wit\W_e$ of those for which
$m_{\w_0}$ is $\sigma$-ergodic are $\sigma$-invariant and of complete
measure. And given a measure $m\in\mminv(\W,\R)$ and a real function
$f\in L^1(\W,m)$, there exists a set $\W_f\subseteq\wit\W_e$
with $m(\W_f)=1$ such that
$f\in L^1(\W,m_{\w_0})$ for every $\w_0\in\W_f$ and
$\int_\W f(\w)\,dm=\int_{\W_f}\big(\int_\W f(\w)dm_{\w_0}\big)\,dm$.
%This is the so-called ergodic decomposition of a $\sigma$-invariant
%measure.
\par
Throughout the paper, $\mB_\W(\w_0,\delta):=\{\w\in\W\,|\;\dist_{\W}(\w_0,\w)\le\delta\}$.
%%%%%%%%%%%%%%%%%%%%%%%%%%%%%%%%%%%%%%%%%%%%%%%%%%%%%%
%%%%%%%%%%%%%%%%%%%%%%%%%%%%%%%%%%%%%%%%%%%%%%%%%%%%%%
\subsection{Scalar skewproduct flows associated to families of ODEs}
\label{2.subsecskew}
Let $(\W,\sigma)$ be a global flow on a compact metric space,
and consider the one-dimensional trivial bundle $\W\times\R$, which is provided with
the structure of a complete metric space by the distance
$\dist_{\W\times\R}\big((\w_1,x_1),(\w_2,x_2)\big):=
\dist_\W(\w_1,\w_2)+|x_1-x_2|$. The sets
$\W$ and $\R$ are the {\em base} and the {\em fiber} of the bundle.
The sections of a subset $\mC\subseteq\W\times\R$,
over the base elements are represented as $\mC_\w:=
\{x\in\R\,|\;(\w,x)\in\mC\}$.
\par
From now on, and throughout the whole paper,
we will represent
\[
 \wt:=\sigma_t(\w)=\sigma(t,\w)\,.
\]
\par
Let us consider the scalar family of equations
\begin{equation}\label{2.ecescalar}
 x'=f(\wt,x)
\end{equation}
for $\w\in\W$, where $f\colon\R\times\W\to\R$ is
assumed to be jointly continuous and locally Lipschitz with respect to
the state variable $x$.
We will use the notation \eqref{2.ecescalar}$_\w$
to refer to the equation of the family corresponding to the point $\w$, and proceed
in an analogous way with the rest of families of equations appearing
in the paper.
\par
The family \eqref{2.ecescalar} allows us to define the map
\begin{equation}\label{2.deftau}
 \tau\colon\mU\subseteq\R\times\W\times\R\to\W\times\R\,,\quad
 (t,\w,x_0)\mapsto (\wt,x(t,\w,x_0))\,,
\end{equation}
where $\mI_{\w,x_0}\to\R\,,\;t\mapsto x(t,\w,x_0)$ is the maximal solution
of \eqref{2.ecescalar}$_\w$ with initial datum $x(0,\w,x_0)=x_0$, and $\mU:=
\bigcup_{(\w,x_0)\in\W\times\R}\big(\mI_{\w,x_0}\times\{(\w,x_0)\}\big)$, an open set.
The uniqueness of solutions ensures that $x(s+t,\w,x_0)=x(s,\wt,x(t,\w,x_0))$
whenever the right-hand term is defined, and this property ensures
that $\tau$ defines a (local or global) flow on $\W\times\R$. The
properties assumed on $f$ also ensure that $x(t,\w,x_0)$ varies continuously
with respect to $\w$ and $x_0$, and hence $\tau$ is continuous on its domain.
If, in addition, $f$ is assumed to be $C^1$ with respect to $x_0$, so is the map
$(t,\w,x_0)\mapsto x(t,\w,x_0)$, as long as it is defined.
The uniqueness of solutions also guarantees that
$\tau$ is {\em fiber-monotone}; that is, if $x_1<x_2$ then
$x(t,\w,x_1)<x(t,\w,x_2)$ for any $t$ in the common interval of definition
of both solutions.
\par
The flow $(\W\times\R,\tau)$ is a type of {\em skewproduct flow
on $\W\!\times\!\R$ projecting onto $(\W,\sigma)$}.
The flow $(\W,\sigma)$ is the {\em base flow} of $(\W\times\R,\tau)$.
In the linear homogeneous case $f(\w,x)=a(\w)\,x$, the flow
$\tau$ is globally defined and {\em linear\/}; that is, the map
$\R\to\R,\,x_0\mapsto x(t,\w,x_0)$ is defined and
linear for all $(t,\w)\in\R\times\W$.
\par
A measurable map $\alpha\colon\W\to\R$ is a {\em $\tau$-equilibrium}
if $\alpha(\wt)=x(t,\w,\alpha(\w))$ for all $t\in\R$ and $\w\in\W$;
a {\em $\tau$-subequilibrium}
if $\alpha(\wt)\le x(t,\w,\alpha(\w))$ for all $\w\in\W$ and $t\ge 0$;
and a {\em $\tau$-superequilibrium}
if $\alpha(\wt)\ge x(t,\w,\alpha(\w))$ for all $\w\in\W$ and $t\ge 0$.
%A $\tau$-subequilibrium or a $\tau$-subequilibrium is {\em strong\/} if the
%corresponding inequality is strict for any all $\w\in\W$ and $t>0$.
There is a strong connection among sub or superequilibria and
upper or lower solutions of the differential equations, which we will explain
when required.
A set $\mK\subset\W\times\Y$ is a {\em copy of the base for $\tau$}
if it is the graph of a continuous equilibrium $\alpha$, in which
case we write $\mK=\{\alpha\}$.
\par
We say that a $\tau$-invariant compact set $\mK\subset\W\times\R$ projecting over
the whole base is {\em uniformly stable at $+\infty$} ({\em on the fiber})
if for any $\kappa>0$ there exists some $\delta>0$
such that, if $(\w,\bar x_0)\in\mK$ and $(\w,x_0)\in\W\times\R$ satisfy
$|\bar x_0-x_0|<\delta$, then
$x(t,\w,x_0)$ is defined for all $t\ge 0$, and in addition
$|x(t,\w,\bar x_0)-x(t,\w,x_0)|\leq
\kappa$ for $t\ge 0$.
Changing $t\ge 0$ by $t\le 0$ provides the definition of {\em uniformly
stable at $-\infty$} $\tau$-invariant compact set.
\par
A $\tau$-invariant compact set $\mK\subset\W\times\R$ projecting over
the whole base is {\em uniformly exponentially stable at
$+\infty$} ({\em on the fiber}) if there exist $\delta>0$, $\kappa\ge 1$ and $\gamma>0$
such that, if $(\w,\bar x_0)\in\mK$ and $(\w,x_0)\in\W\times\R$ satisfy
$|\bar x_0-x_0|<\delta$,
then $x(t,\w,x_0)$ is defined for all $t\ge 0$, and in addition
$|x(t,\w,\bar x_0)-x(t,\w,x_0)|\leq
\kappa\,e^{-\gamma\,t}\,|\bar x_0-x_0|$ for $t\ge 0$.
Changing $t\ge 0$ by $t\le 0$ provides the definition of
{\em uniformly exponentially stable at $-\infty$}
$\tau$-invariant compact set.
\begin{nota}\label{2.enfibra}
We want to insist in the fact that our definitions of (exponential or not)
uniform stability for skew-product semiflows are not the classical ones for flows,
since we do not consider possible variation on the base points: we
just refer to variation on the fiber.
For further purposes we also point out that, if $(\W,\sigma)$ is a
equicontinuous continuous flow on a compact metric space,
then the whole space is uniformly stable at $\pm\infty$
in the classical sense.
\end{nota}
\par
The {\em Hausdorff semidistance\/}
from $\mC_1$ to $\mC_2$, where $\mC_1, \mC_2\subset\W\times\R$, is
\[
 \dist(\mC_1,\mC_2):=\sup_{(\w_1,x_1)\in\mC_1}\left(\inf_{(\w_2,x_2)\in\mC_2}
 \big(\dist_{\W\times\R}((\w_1,x_1),(\w_2,x_2))\big)\right).
\]
A set $\mB\subset\W\!\times\!\Y$ is said {\em to attract a set
$\mC\subseteq\W$ under $\tau$} if
$\tau_t(\mC)$ is defined for all $t\ge 0$ and,
in addition, $\lim_{t\to\infty}\dist(\tau_t(\mC),\mB)=0$. The flow
$\tau$ is {\em bounded
dissipative\/} if there exists a bounded set $\mB$ attracting all
the bounded subsets of $\W\times\R$ under $\tau$.
And a set $\mA\subset\W\times\R$
is a {\em global attractor\/} for $\tau$ if it is compact, $\tau$-invariant, and it
attracts every bounded subset of $\W\times\R$ under $\tau$.
%Finally, a set $\mB$ is {\em absorbing under $\Pi$} if, for any
%bounded set $\mB_0$, there exists $t_0=t_0(\mB_0,\mB)$ such that
%$\Pi(t,\mB_0)\subseteq\mB$ for all $t\ge t_0$.
\par
Finally, a Borel measure $\nu$ on $\W\times\R$ {\em projects on\/}
a measure $m$ on $\W$, given by $m(\mB)=\nu(\mB\times\R)$ for any Borel
subset $\mB\subseteq\W$; and it is easy to check that $m$ is $\sigma$-invariant
if $\nu$ is $\tau$-invariant.
%%%%%%%%%%%%%%%%%%%%%%%%%%%%%%%%%%%%%%%%%%%%%%%%%%%%%%%%%%%%%%%%%%%%
%%%%%%%%%%%%%%%%%%%%%%%%%%%%%%%%%%%%%%%%%%%%%%%%%%%%%%%%%%%%%%%%%%%%
\subsection{Sacker and Sell spectrum of a family of linear scalar equations}
\label{2.subseclin}
Let $(\W,\sigma)$ be a minimal flow on a compact metric space, and
let us consider the family of linear differential equations
\begin{equation}\label{2.eclinear}
 x'=a(\wt)\,x
\end{equation}
for $\w\in\W$, where $a\colon\W\to\R$ is continuous.
\begin{defi}\label{2.defEDsubfibrados}
The family \eqref{2.eclinear} has {\em exponential dichotomy
over $\W$\/} if there exist $\kappa\ge 1$ and $\gamma>0$ such that either
\begin{equation}\label{2.masi}
 \exp\int_0^t a(\w{\cdot}l)\,dl\le \kappa\,e^{-\gamma t} \quad
 \text{whenever $\w\in\W$ and $t\ge 0$}
\end{equation}
or
\begin{equation}\label{2.menosi}
 \exp\int_0^t a(\w{\cdot}l)\,dl\le\kappa\,e^{\gamma t} \quad
 \text{whenever $\w\in\W$ and $t\le 0$}\,.
\end{equation}
\end{defi}
\begin{notas} \label{2.notaED}
1.~Since the base flow $(\W,\sigma)$ is minimal, the exponential
dichotomy of the family \eqref{2.eclinear} over $\W$ is equivalent
to the exponential dichotomy of any of its equations over $\R$:
see e.g.~Theorem 2 and Section 3 of \cite{sase2}.
\par
2.~The family \eqref{2.eclinear} has
exponential dichotomy over $\W$ if and only if no one of its equations
has a nontrivial bounded solution: see e.g.~Theorem 1.61 of \cite{jonnf}.
In other words, the property fails if and only if there exists $\ww\in\W$
such that $\sup_{t\in\R}\exp\big(\int_0^ta(\ww{\cdot}s)\,ds\big)<\infty$.
\end{notas}
\begin{defi}\label{2.defSSS}
The {\em Sacker and Sell spectrum\/} or {\em dynamical spectrum\/}
of the linear family \eqref{2.eclinear}
is the set $\Sigma_a$ of $\lambda\in\R$ such that the family
$x'=(a(\wt)-\lambda)\,x$
does not have exponential dichotomy over $\W$.
\end{defi}
Note that, in the autonomous case $a(\w)\equiv a\in\R$, the
set $\Sigma_a$ is given by $\{a\}$.
\begin{defi}\label{2.defiexp}
The {\em lower Lyapunov exponent of the family\/} \eqref{2.eclinear} {\em for $(\W,\sigma)$}
is
\[
 \gamma^i_\W:=\inf\left\{\int_\W a(\w)\,dm\,|\;m\in\mminv(\W,\sigma)\right\}\,,
\]
and the {\em upper Lyapunov exponent of the family\/} \eqref{2.eclinear} {\em for $(\W,\sigma)$} is
\[
 \gamma^s_\W:=\sup\left\{\int_\W a(\w)\,dm\,|\;m\in\mminv(\W,\sigma)\right\}\,.
\]
\end{defi}
For the reader's convenience, we include a proof of the next well known result.
\begin{teor}\label{2.teorespectro}
\begin{itemize}
\item[\rm(i)] There exist $m^i,m^s\in\mmerg(\W,\sigma)$ such that
\[
 \gamma^i_\W:=\int_\W a(\w)\,dm^i\quad\text{and}\quad\gamma^s_\W:=\int_\W a(\w)\,dm^s\,.
\]
\item[\rm(ii)] The Sacker and Sell spectrum of the linear family \eqref{2.eclinear}
is $\Sigma_a=[\,\gamma^i_\W\,,\,\gamma^s_\W\,]$, and it may be a singleton.
\end{itemize}
\end{teor}
\begin{proof}
The Sacker and Sell spectral theorem \cite[Theorem 2]{sase6}
states that, in this scalar case,
$\Sigma_a$ is given by a closed (perhaps degenerate) interval,
say $[\lambda_i,\lambda_s]$. Theorem 2.3 of \cite{jops} shows that this
interval contains $\int_\W a(\w)\,dm$ for all $m\in\mminv(\W,\sigma)$,
as well as the existence of $m^i,m^s\in\mmerg(\W,\sigma)$ such that
$\lb_i:=\int_\W a(\w)\,dm^i$ and $\lb_s=\int_\W a(\w)\,dm^s$. These
properties show the assertions.
\end{proof}
\begin{nota}\label{2.notaDE}
It is clear $0\in\Sigma_a$ if and only if the family \eqref{2.eclinear} does not
have exponential dichotomy over $\W$. In addition,
Theorem~\ref{2.teorespectro} ensures that:
\begin{itemize}
\item[-] $\Sigma_a\subset(-\infty,0)$ if and only if the upper Lyapunov exponent
of the family \eqref{2.eclinear} is negative; or, equivalently, if and only if
$\int_\W a(\w)\,dm<0$ for any $m\in\mminv(\W,\sigma)$.
\item[-] $\Sigma_a\subset(0,\infty)$ if and only if the lower Lyapunov exponent
of the family \eqref{2.eclinear} is positive; or, equivalently, if and only if
$\int_\W a(\w)\,dm>0$ for any $m\in\mminv(\W,\sigma)$.
\end{itemize}
\end{nota}
%%%%%%%%%%%%%%%%%%%%%%%%%%%%%%%%%%%%%%%%%%%%%%%%%%%
%%%%%%%%%%%%%%%%%%%%%%%%%%%%%%%%%%%%%%%%%%%%%%%%%%%
\subsection{The minimal subsets of a scalar skewproduct
flow induced by a family of scalar ODEs over a minimal base}\label{2.subsecescalar}
As in the previous section, $(\W,\sigma)$ is a minimal continuous flow on
a compact metric space, and this assumption on minimality is fundamental.
We will recall in this subsection some
properties of the minimal sets for the scalar skewproduct
flow $(\W\times\R,\tau)$ given by the expression \eqref{2.deftau}; that is,
given by the solutions of the family \eqref{2.ecescalar} over $\W$.
We will also define some types of sets which will be fundamental in the
dynamical description of Section \ref{3.sec}.
\par
It is very easy to deduce from the minimality of the base flow
that any copy of the base is $\tau$-minimal, and that
any compact $\tau$-invariant set $\mK\subset\W\times\R$ projects over the
whole base $\W$. If, for such a set $\mK$, there exists a point $\w\in\W$ such that
$\mK_\w$ is a singleton, then $\mK$ is a {\em pinched set}. A minimal
pinched set is an {\em almost automorphic extension
of the base}. It turns out that, for our scalar skewproduct flow,
any minimal set $\mM$ is an almost automorphic extension
of the base. To briefly explain this fact, we observe that
\begin{equation}\label{2.M}
 \mM\subseteq\bigcup_{\w\in\W}\big(\{\w\}\times[\alfam(\w),\betam(\w)]\big)
\end{equation}
where $\alfam(\w):=\inf\{x\in\R\,|\;(\w,x)\in\mM\}$ and
$\betam(\w):=\sup\{x\in\R\,|\;(\w,x)\in\mM\}$.
It is not hard to deduce from the compactness of $\mM$ that $\alfam$ and
$\betam$ are lower and upper semicontinuous; from its $\tau$-invariance
that they are $\tau$-equilibria; and from its minimality
that $\mM=\text{closure}_{\W\times\R}\{(\wt,\alfam(\wt))\,|\;t\in\R\}$
(resp.~$\mM=\text{closure}_{\W\times\R}\{(\wt,\betam(\wt))\,|\;t\in\R\}$)
for any $\w\in\W$, and hence that $\mM_\w=\{\alfam(\w)\}$
(resp.~$\mM_\w=\{\betam(\w)\}$) at any point $\w$ at which $\alfam$
(resp.~$\betam$) is continuous.
Therefore, $\alfam$ and $\betam$ have the same $(\sigma$-invariant and residual)
set $\W_\mM\subseteq\W$ of continuity points, which are exactly
the points at which both maps coincide; and
$\mM_\w$ reduces to a singleton if and only if $\w\in\W_\mM$.
The functions $\alfam$ and $\betam$ are hence continuous
if and only if $\alfam(\w)=\betam(\w)$ for all $\w\in\W$. In other words,
if and only if $\mM$ is a copy of the base: $\mM=\{\alfam\}=\{\betam\}$.
\par
Two different $\tau$-minimal sets $\mM$ and $\mN$ are {\em fiber-ordered},
in the following sense: if there exists $(\w_0,x_0)\in\mM$ and $(\w_0,y_0)\in\mN$
such that $x_0<y_0$, then $x<y$ whenever $(\w,x)\in\mM$ and $(\w,y)\in\mN$.
To prove this fact, we take a common element $\bw\in\W$ such that
$\mM_{\bw}=\{\bar x\}$ and $\mN_{\bw}=\{\bar y\}$ and assume without
restriction that $\bar x<\bar y$. Let us reason by contradiction assuming
the existence of $(\w,x)\in\mM$ and $(\w,y)\in\mN$ with $x>y$.
We look for $(t_n)$ such that $(\bw,\bar x)=\lim_{n\to\infty}\tau(t_n,\w,x)$
and a suitable subsequence $(t_k)$ such that there exists
$\lim_{k\to\infty}\tau(t_k,\w,y)$. Then, this limit is necessarily
$(\bw,\bar y)$, and the fiber-monotonicity of $\tau$ ensures that
$\bar x\ge \bar y$, which is the sought-for contradiction.
\par
Assume now that $f$ is $C^1$ with respect to its second argument.
Given a $\tau$-minimal set $\mM$, we can consider the {\em linearized flow} on
$\mM\times\R$, given by the solutions of the family of variational equations
\begin{equation}\label{2.ecvari}
 z'=f_x(\tau(t,\w,x_0))\,z
\end{equation}
for $(\w,x_0)\in\mM$, where $f_x:=\partial f/\partial x$.
A $\tau$-minimal set $\mM\subset\W\times\R$ is {\em hyperbolic} if
the family \eqref{2.ecvari} has exponential dichotomy over $\mM$.
This last definition is justified by the next result.
For the reader's convenience, we give a proof
of this well-known fact, concerning
hyperbolic minimal sets, which will be crucial in Section \ref{3.sec}. The functions
$\alfam$ and $\betam$ are those associated to $\mM$ by \eqref{2.M}.
The uniform exponential stability properties are defined in Subsection \ref{2.subsecskew}.
\begin{prop} \label{2.propcopia}
Assume that the functions $f,f_x\colon\W\times\R\to\R$ are jointly continuous,
let $(\W\times\R,\tau)$ be the flow induced by the family \eqref{2.ecescalar},
and let $\mM\subset\W\times\R$ be a $\tau$-minimal set.
Then,
\begin{itemize}
\item[\rm(i)] the family \eqref{2.ecvari} has exponential dichotomy over $\mM$ given by
condition \eqref{2.masi} if and only if $\mM$ is a uniformly exponentially stable at $+\infty$
copy of the base: $\mM=\{\alfam\}=\{\betam\}$. In addition, in this case, given $(\w,x_0)\notin\mM$,
there exists $\rho>0$ and $t_-<0$ such that $|x(t,\w,x_0)-\alfam(\wt)|>\rho$ for
$t\le t_-$ in the maximal interval of definition of $x(t,\w,x_0)$.
\item[\rm(ii)] The family \eqref{2.ecvari} has exponential dichotomy over $\mM$ given by
condition \eqref{2.menosi} if and only if $\mM$ is a uniformly exponentially stable at $-\infty$
copy of the base: $\mM=\{\alfam\}=\{\betam\}$. In addition, in this case, given $(\w,x_0)\notin\mM$,
there exists $\rho>0$ and $t_+>0$ such that $|x(t,\w,x_0)-\alfam(\wt)|>\rho$ for
$t\ge t_+$ in the maximal interval of definition of $x(t,\w,x_0)$.
\end{itemize}
\end{prop}
\begin{proof}
(i) Let us fix $(\w_1,x_1)\in\mM$, and assume that the family \eqref{2.ecvari} satisfies the
corresponding condition \eqref{2.masi}. The hypotheses on $f$ ensure that
$f(\w_1,x)-f(\w_1,x_1)=f_x(\w_1,x_1){\cdot}(x-x_1)+r(\w,x)$, with
$\lim_{x\to x_1}|r(\w,x)|/|x-x_1|=0$. Therefore, the change of variables $y=x-x(t,\w_1,x_1)$ takes
the equation \eqref{2.ecescalar}$_{\w_1}$ to
\begin{equation}\label{2.ectran}
 y'=f_x(\tau(t,\w_1,x_1))\,y+\wit r(\w_1{\cdot}t,y)\,,
\end{equation}
with $\lim_{x\to 0}|\wit r(\w_1,y)|/|y|=0$. Let $y(t,\w_1,y_0)$ represent the solution
of \eqref{2.ectran} with $y(0,\w_1,y_0)=y_0$, so that
$y(t,\w_1,y_0)=x(t,\w_1,y_0+x_1)-x(t,\w_1,x_1)$. Then, condition \eqref{2.masi} and
the First Approximation Theorem
(see \cite[Theorem~III.2.4]{hale} and its proof) provide $\delta>0$ such that
\begin{equation}\label{2.fa}
 |y(t,\w_1,y_0)|\le\kappa\,e^{(-\gamma/2)\,t}|y_0| \qquad
 \text{for any $\;t\ge 0\;$ if $\;|y_0|\le\delta$}\,.
\end{equation}
In addition, the constant $\delta$ can be chosen to satisfy \eqref{2.fa} for any
$\w_1\in\W$.
\par
Now we take any point $(\w_1,x_2)\in\mM$, and will check that $x_2=x_1$.
We can choose $\ww$ with $\mM_{\ww}=\{\wit x\}$. Then,
$\lim_{n\to\infty}(\w_1{\cdot}(-t_n),u(-t_n,\w_1,x_1))=
(\ww,\wit x)$ for a sequence $(t_n)\uparrow\infty$. We take a subsequence $(t_k)$
such that $\lim_{k\to\infty}(\w_1{\cdot}(-t_k),x(-t_k,\w_1,x_2))$ exists,
and observe that this limit must be $(\ww,\wit x)$, since it belongs to $\mM$.
Hence, $\lim_{k\to\infty}y(-t_k,\w_1,x_2-x_1)=
\lim_{k\to\infty}(x(-t_k,\w_1,x_2)-x(-t_k,\w_1,x_1))=\wit x-\wit x=0$.
For $k$ large enough to ensure that
$|y(-t_k,\w_1,x_2-x_1)|\le\delta$, \eqref{2.fa} yields
\[
 |x_2-x_1|=|y(t_k,\w_1{\cdot}(-t_k),y(-t_k,\w_1,x_2-x_1))|
 \le\kappa\,e^{(-\gamma/2)\,t_k}\,\delta\,.
\]
Taking limit as $k\to\infty$ allows us to ensure that $x_2=x_1$, as asserted.
\par
Let us write $\mM=\{\eta\}$ for a continuous function $\eta\colon\W\to\R$.
The continuous flow transformation $(\w,x)\mapsto(\w,x-\eta(\w))$ takes $\mM$ to the set
$\W\times\{0\}$, which is a copy of the base for the flow induced by the family of equations
$y'=f_x(\wt,\eta(\wt))\,y+\wit r(\wt,y)$ for $\w\in\W$.
It follows from \eqref{2.fa} that $\W\times\{0\}$ is uniformly
exponentially stable, which ensures the analogous property for $\mM$ and the initial flow $\tau$.
The \lq\lq only if"~part of the first assertion of (i) is proved.
\par
Conversely, let us assume that $\mM$ is an exponentially stable copy at $+\infty$ copy of the
base. Then, for all $(\w,x)\in\mM$,
$|(\partial x/\partial x_0)(t,\w,x_0)|=
\lim_{h\to 0}|x(t,\w,x_0+h)-x(t,\w,x_0)|/|h|\le \kappa \,e^{-\gamma t}$ for
certain constants $\kappa\ge 1$ and $\gamma>0$, and for all $t\ge 0$.
This implies that the family of equations \eqref{2.ecvari}, defined for $(\w,x_0)\in\mM$,
satisfies condition \eqref{2.masi}, and completes the proof of the equivalence stated in (i).
\par
Assume now that we are in the described situation, and let $\delta,\,\kappa$ and $\gamma$
be the constants associated to the uniformly exponentially character at $+\infty$ of
$\mM$. To prove the
last assertion in (i) we take $x_0\ne\alfam(\w)$ and assume for contradiction
the existence of $(t_n)\downarrow-\infty$
such that $\lim_{n \to\infty}|x(t_n,\w,x_0)-\alfam(\wt_n)|=0$. Thus, for large enough $n$,
$|x(t_n,\w,x_0)-\alfam(\wt_n)|\le\delta$.
But then $|x_0-\alfam(\w)|=|x(-t_n,\wt_n,x(t_n,\w,x_0))-\alfam((\wt_n){\cdot}(-t_n))|
\le\kappa\,e^{\gamma\,t_n}\delta$. The contradiction comes from the convergence to $0$
of the right-hand term. The proof of (i) is complete.
\smallskip\par
(ii) The proofs are analogous if the exponential dichotomy is given by \eqref{2.menosi}
or the uniform exponential stability occurs at $-\infty$.
\end{proof}
\begin{defi}
The {\em upper} and {\em lower Lyapunov exponents\/} of a $\tau$-minimal set
$\mM\subset\W\times\R$ are the upper and lower Lyapunov exponents
of the family of variational equations \eqref{2.ecvari} over $\mM$.
\end{defi}
As a consequence of this definition, Remark~\ref{2.notaED},
Theorem \ref{2.teorespectro}(ii), and Proposition \ref{2.propcopia}, we have:
\begin{coro}\label{2.coroDE}
Assume that the functions $f,f_x\colon\W\times\R\to\R$ are jointly continuous,
and let $(\W\times\R,\tau)$ be the flow induced by the family \eqref{2.ecescalar}.
If the upper Lyapunov exponent of the $\tau$-minimal set $\mM$
is negative, then $\mM$ is an exponentially
stable at $+\infty$ copy of the base. If its lower Lyapunov
exponent is positive, then $\mM$ is an exponentially
stable at $-\infty$ copy of the base. And, in both cases,
$\mM=\{\alfam\}=\{\betam\}$.
\end{coro}
%%%%%%%%%%%%%%%%%%%%%%%%%%%%%%%%%%%%%%%%%%%%%%%%%%%%%%%%%%%
%%%%%%%%%%%%%%%%%%%%%%%%%%%%%%%%%%%%%%%%%%%%%%%%%%%%%%%%%%%
\subsection{The set $\mR_m(\W)$}\label{2.subsecR}
We continue this section of preliminaries
by describing a set of continuous maps $a\colon\W\to\R$
which will play a crucial role in the description of
the occurrence of Li-Yorke chaos and Auslander-Yorke chaos
(defined in the next subsections) in one of the dynamical
situations which we will consider in Section \ref{3.sec}.
Most of these properties are (basically) already known; but, to our knowledge,
Theorem \ref{2.teordensidad} presents a new property.
The assumption of minimality of $(\W,\sigma)$ is in force.
\begin{defi}\label{2.defcontpri}
A continuous  function $a\colon\W\to\R$ {\em admits a continuous primitive}
if there exists a continuous function $h_a\colon\W\to\R$ such that
$h_a(\wt)-h_a(\w)=\int_0^t a(\ws)\,ds$ for all $\w\in\W$ and $t\in\R$.
\end{defi}
\begin{nota}\label{2.notacontpri}
Note that $\sup_{(t,\w)\in\R\!\times\!\W}\left|\int_0^t a(\ws)\,ds\right|<\infty$
if $a$ admits a continuous primitive, and that
Birkhoff's ergodic theorem ensures that $\int_{\W}a(\w)\,dm=0$ for any
$m\in\mminv(\W,\sigma)$.
It is well-known that if $(\W,\sigma)$ is minimal (as in our case)
then $a$ admits a continuous primitive if and only if there exists $\ww\in\W$ with
$\sup_{t\ge 0}\left|\int_0^t a(\ww\pu s)\,ds\right|<\infty$ or with
$\sup_{t\le 0}\left|\int_0^t a(\ww\pu s)\,ds\right|<\infty$:
a proof is given in \cite[Proposition A.1]{jnot}.
\end{nota}
\begin{defi}\label{2.defR}
Given $m\in\mminv(\W,\sigma)$, $\mR_m(\W)$ is the set of continuous functions
$a\colon\W\to\R$ satisfying $\int_\W a(\w)\,dm=0$ which do not
admit a continuous primitive and such that
$\sup_{t\le 0}\int_0^t a(\ws)\,ds<\infty$ for $m$-a.e.~$\w\in\W$.
\end{defi}
There are well known examples of quasi-periodic functions
$a_0\colon\R\to\R$ giving rise
to a hull $\W$ and a map $a$ in the set $\mR_m(\W)$
corresponding to the unique ergodic measure on $\W$.
For example, those described in \cite{john9}
and in \cite{orta}. Our next result shows that it is nonempty whenever
the flow is minimal and non periodic.
The $\sigma$-ergodic measure $m_{\w_0}$
associated to every $\sigma$-generic point in the set $\wit\W_e\subseteq\W$
(of complete measure) is defined in Subsection \ref{2.subsecbasic}.
\begin{teor}\label{2.teornovacio}
Assume that the flow $(\W,\sigma)$ is minimal
and non periodic. Then,
\begin{itemize}
\item[\rm(i)] $\mR_m(\W)$ is nonempty
for any $m\in\mminv(\W,\sigma)$, and it contains functions
$a$ with $\Sigma_a=\{0\}$.
\item[\rm(ii)] In fact, there exist functions $a$ which belong to
$\bigcap_{m\in\mminv(\W,\sigma)}\mR_m(\W)$, with $\Sigma_a=\{0\}$.
\item[\rm(iii)] If $a\in\mR_m(\W)$ for a measure $m\in\mminv(\W,\sigma)$, then
there exists at least a measure $\wit m\in\mmerg(\W,\sigma)$ such that
$a\in\mR_{\wit m}(\W)$.
More precisely, $a\in\mR_{m_{\w_0}}(\W)$ for $m$-almost all the measures
$m_{\w_0}\in\mmerg(\W,\sigma)$ with $\w_0\in\wit\W_e$.
\end{itemize}
\end{teor}
\begin{proof}
(i) Let us fix $m\in\mminv(\W,\R)$. We begin by proving an auxiliary result. Let us fix $\w_0\in\W$,
and let us take $\ep>0$. Then, there exists a continuous function $b_\ep\colon\W\to\R$
with $\n{b_\ep}_\W:=\sup_{\w\in\W}|b_\ep(\w)|\le\ep$ which admits a continuous primitive
$h_{b_\ep}\colon\W\to[0,1]$ with
$h_{b_\ep}(\w_0)=1$ and $\int_{\W}h_{b_\ep}(\w)\,dm\le\ep$. In fact, we will construct
$b_\ep$ and $h_{b_\ep}$. We take $T\ge 2/\ep$, and note that $m(\{\w_0{\cdot}t\,|\;t\in[0,T]\})=0$,
since the flow is non periodic: otherwise we wold obtain a $\sigma$-orbit with inifite measure,
impossible. The regularity of $m$ and Uryshon's Lemma provide a continuous function $c_\ep\colon\W\to[0,1]$
such that $c_\ep(\w_0{\cdot}t)=1$ for $t\in[0,T]$ and with $\int_{\W}c_\ep(\w)\,dm\le\ep$.
We define $b_\ep(\w):=(c_\ep(\w{\cdot}T)-c_\ep(\w))/T$ and $h_{b_\ep}(\w):=(1/T)\int_0^T c_\ep(\ws)\,ds$,
and check that $(h_{b_\ep})'(\w):=(d/dt)\,h_{b_\ep}(\wt)|_{t=0}$ coincides with $b_\ep(\w)$.
Clearly, $\n{b_\ep}_\W\le 2/T\le\ep$. In addition, $h_{b_\ep}\ge 0$, and
$h_{b_\ep}(\w_0)=(1/T)\int_0^T c(\w_0{\cdot}s)\,ds=1$. Finally, using the $\sigma$-invariance of
$m$, we get
\[
\begin{split}
 \int_\W h_{b_\ep}(\w)\,dm
 &=\frac{1}{T}\int_\W\int_0^T c_\ep(\ws)\,ds\,dm
 =\frac{1}{T}\int_0^T \int_\W c_\ep(\ws)\,dm\,ds\\
 &=\frac{1}{T}\int_0^T \int_\W c_\ep(\w)\,dm\,ds
 =\int_\W c_\ep(\w)\,dm\le\ep\,,
\end{split}\]
which completes the proof of our initial assertion.
\par
This property allows us to construct a sequence $(b_n)$ of
continuous functions with continuous primitives $(h_{b_n})$ such
that $\n{b_n}_\W\le 1/2^n$ (so that
$\sum_{n=1}^\infty \n{b_n}_\W\le 1$), $h_{b_n}(\w)\in[0,1]$
for all $\w\in\W$, $\int_{\W} h_{b_n}(\w)\,dm\le 1/2^n$
(so that $\sum_{n=1}^\infty\int_{\W} h_{b_n}(\w)\,dm\le 1$),
and with $h_{b_n}(\w_0)=1$ for a previously fixed $\w_0\in\W$ and all $n\in\N$.
Let us call $\wit h(\w):=\sum_{n=1}^\infty h_{b_n}(\w)\in[0,\infty]$. Lebesgue's
monotone convergence theorem shows that $\int_\W \wit h(\w)\,dm=\sum_{n=1}^\infty\int_{\W}
h_{b_n}(\w)\,dm\le 1$, and hence
\[
 \wit\W:=\{\w\in\W\,|\;\wit h(\w)<\infty\}
\]
satisfies $m(\wit\W)=1$. Note also that $\w_0\notin\wit\W$. In addition,
$\wit\W$ is $\sigma$-invariant: for every $\w\in\W$, $j\in\N$ and $t\in\R$,
\[
 \sum_{n=1}^j h_{b_n}(\wt)=\sum_{n=1}^j h_{b_n}(\w)+\sum_{n=1}^j \int_0^t b_n(\ws)\,ds
 \le \sum_{n=1}^\infty  h_{b_n}(\w)+|t|\sum_{n=1}^\infty \n{b_n}_\W\,.
\]
\par
Let us define $a:=-\sum_{n=1}^\infty b_n$, which is a continuous function on $\W$.
Then, the function $h_a$ defined by
$h_a(\w):=-\sum_{n=1}^\infty h_{b_n}(\w)=-\wit h(\w)$ for $\w\in\wit\W$ and $h_a(\w):=0$ for $\w\notin\wit\W$
satisfies $h_a(\wit\w{\cdot}t)-h_a(\wit\w)=\int_0^t a(\wit\w{\cdot}s)\,ds$ for all $\w\in\wit\W$.
Observe that $\sup_{t\le 0}\int_0^t a(\wit\w{\cdot}s)\,ds=\sup_{t\le 0}(h_a(\wit\w{\cdot}t)-h_a(\wit\w))=
\sup_{t\le 0}(\wit h(\wit\w)-\wit h_a(\wit\w{\cdot}t))\le \wit h(\wit\w)<\infty$ for all $\wit\w\in\wit\W$.
Let us check that $\inf_{t\le 0}\int_0^t a(\wit\w{\cdot}s)\,ds=-\infty$ for all $\wit\w\in\wit\W$.
We fix $\wit\w\in\wit\W$ and $\w_0\notin\wit\W$ and look for $(t_k)\downarrow-\infty$
such that $\w_0=\lim_{k\to\infty}\wit\w{\cdot}(t_k)$. For any $j\in\N$,
\[
 \int^{t_k}_0 a(\wit\w{\cdot}s)\,ds=\sum_{n=1}^\infty(h_{b_n}(\wit\w)-h_{b_n}(\wit\w{\cdot}t_k))
 \le\sum_{n=1}^\infty h_{b_n}(\wit\w)-\sum_{n=1}^j h_{b_n}(\wit\w{\cdot}t_k)\,,
\]
and hence $\liminf_{k\to\infty}\int^{t_k}_0 a(\wit\w{\cdot}s)\,ds
\le\sum_{n=1}^\infty h_{b_n}(\wit\w)-\sum_{n=1}^j h_{b_n}(\w_0)$ for all $j\in\N$. By letting
$j$ increase, we check the assertion, which in turn precludes the existence of a
continuous primitive for $a$. Altogether, $a$ satisfies all the conditions
of Definition \ref{2.defR}, and hence $a\in\mR_m(\W)$.
\par
Finally, note that the map $a$ is the uniform limit of the sequence $(s_j)$, with
$s_j:=-\sum_{n=1}^j b_n$. Each one of the functions $s_j$ has a continuous primitive,
and hence $\int_{\W} s_j(\w)\,d\wit m=0$ for every $j\in\N$ and
$\wit m\in\mminv(\W,\sigma)$ (see Remark \ref{2.notacontpri}).
%In addition, $\n{s_j}_\W\le 1$ for all $j\in\N$ and $\n{a}_\W\le 1$.
%Lebesgue's dominated convergence theorem ensures that
Therefore, $\int_\W a(\omega)\,d\wit m=0$
for every $\wit m\in\mminv(\W,\sigma)$, and hence Theorem \ref{2.teorespectro} shows that
$\Sigma_a=\{0\}$.
\smallskip\par
(ii) The idea is to repeat the process of (i), but taking functions $(b_n)$ such that
$\int_{\W} h_{b_n}(\w)\,dm\le 1/(2^n)$ for all $m\in\mminv(\W,\R)$. Therefore, we must
change the initial step in the proof of (i) to show that, given $\w_0\in\W$ and $\ep>0$,
there exists a continuous function $b_\ep\colon\W\to\R$ with
$\n{b_\ep}_\W\le\ep$ which admits a
continuous primitive $h_{b_\ep}\colon\W\to[0,1]$ with $h_{b_\ep}(\w_0)=1$ and
$\int_{\W}h_{b_\ep}(\w)\,dm\le\ep$ for any $m\in\mminv(\W,\R)$. Let us call
$\mJ:=\{\w_0{\cdot}t\,|\;t\in[0,T]\}$, and look for a continuous function $c\colon\W\to[0,1]$
such that $c(\w)=1$ for $\w\in\mJ$ and $c(\w)<1$ for $\w\notin\mJ$.
Then, the sequence $(c^n)$ decreases pointwisely to the characteristic function of
$\mJ$, and hence Lebesgue's monotone convergence theorem ensures that
$\lim_{n\to\infty}\int_\W c^n(\w)\,dm=0$ for all $m\in\mminv(\W,\sigma)$.
We consider the maps $i_n\colon\mminv(\W,\sigma)\to\R\,,m\mapsto\int_\W c^n(\w)\,dm$,
which are continuous for the weak$^*$ topology of $\mminv(\W,\R)$. The space
$\mminv(\W,\R)$ is compact and metrizable for this topology
(see e.g.~\cite[Theorems 6.4 and 6.5]{walt}). The sequence $(i_n)$ decreases to the
function 0, and hence Dini's theorem ensures that $0=\lim_{n\to\infty}i_n$
uniformly on $\mminv(\W,\R$). Therefore, given $\ep>0$, there exists
$n_\ep\in\N$ such that $\int_\W c^{n_\ep}(\w)\,dm\le\ep$ for all $m\in\mminv(\W,\sigma)$.
We use $c^{n_\ep}$ to construct $b_\ep$ and $h_{b_\ep}$,
as at the beginnig of the proof of (i), and repeat the
rest of it to check (ii).
\smallskip\par
(iii) Let us call $\W^a_-:=\{\w\in\W\,|\;\sup_{t\le 0}\int_0^t a(\ws)\,ds<\infty\}$,
with $m(\W^a_-)=1$. Let $f$ be the characteristic function of
$\W^a_-$. As recalled in Subsection \ref{2.subsecbasic},
there exists a set $\W_f\subseteq\wit\W_e$ with $m(\W_f)=1$ such that
$m(\W^a_-)=\int_\W f(\w)\,dm=\int_{\W_f}\!\big(\int_\W f(\w)dm_{\w_0}\big)\,dm=
\int_{\W_f}m_{\w_0}(\W_-^a)\,dm$.
This ensures that $m_{\w_0}(\W_-^a)=1$ for $m$-almost every $\w_0\in\W_f$.
Let us take one of these points $\w_0$. Then $\int_{\W}a(\w)\,dm_{\w_0}\ge 0$:
if, on the contrary, $\wit a:=\int_{\W}a(\w)\,dm_{\w_0}<0$, then
Birkhoff's ergodic theorem ensures that $\wit a=\lim_{t\to-\infty}(1/t)
\int_0^t a(\ws)\,ds$ for $m_{\w_0}$-almost every $\w\in\W$, which in turn implies
$m_{\w_0}(\W^a_-)=0$, impossible. Now we look for a subset $\W_a\subseteq\wit\W_e$
with $m(\W_a)=1$ such that
$0=\int_\W a(\w)\,dm=\int_{\W_a}\!\big(\int_\W a(\w)dm_{\w_0}\big)\,dm$, and conclude
that $\int_\W a(\w)\,dm_{\w_0}$ for $m$-almost every point. Therefore,
$a$ satisfies the conditions of Definition \ref{2.defR} for
$m_{\w_0}$ for $m$-almost every $\w_0\in\wit\W_e$, as asserted.
\end{proof}
\par
{From} now on, $m\in\mmerg(\W,\sigma)$ is fixed.
\par
The next result summarizes part of the dynamical consequences
on the solutions of the family of linear scalar equation $x'=a(\wt)\,x$,
which are $x(t,\w,x_0)=x_0\,\exp\big(\int_0^t a(\ws)\,ds\big)$.
\begin{prop}\label{2.propR}
Let $a\colon\W\to\R$ be a continuous function with
$\int_\W a(\w)\,dm=0$. The following assertions are equivalent:
\begin{itemize}
\item[\rm(1)] $a\in\mR_m(\W)$.
\item[\rm(2)] The subset
$\W^a\subseteq\W$ of those points $\w$
such that
$\sup_{t\in\R}\int_0^t a(\ws)\,ds<\infty$,
$\inf_{t\le 0}\int_0^t a(\ws)\,ds=-\infty$, and
$\inf_{t\ge 0}\int_0^t a(\ws)\,ds=-\infty$,
is $\sigma$-invariant and satisfies $m(\W^a)=1$.
\item[\rm(3)] There exist an upper-semicontinuous function
$H_a\colon\W\to[0,1]$
and a $\sigma$-invariant set $\W^a\subseteq\W$ with $m(\W^a)=1$
such that: $\w\in\W^a$ if and only if $H_a(\w)>0$; and,
for all $\w\in\W$,
$H_a(\wt)=H_a(\w)\exp\big(\int_0^t a(\ws)\,ds\big)$ for all $t\in\R$,
$\inf_{t\le 0}H_a(\wt)=0$, and $\inf_{t\ge 0}H_a(\wt)=0$.
\end{itemize}
In addition, the function $H_a$ of point {\rm (3)} vanishes at its continuity points.
\end{prop}
\begin{proof}
The proof of the equivalences
repeats that of \cite[Proposition 6.4]{nuob10}: the map
\[
 H_a(\w)=\inf_{t\in\R}\:\frac{1}{\exp\big(\int_0^ta(\ws)\,ds\big)}
\]
satisfies all the assertions of point (3).
\par
Assume now that the semicontinuous function
$H_a$ satisfies the properties of (3) and,
by contradiction, that $H_a(\w_0)=\rho>0$ at a continuity point $\w_0$.
Then there is a nonempty open ball $\mB:=\mB_\W(\w_0,\delta)$ such that
$H_a(\w)>\rho/2$ for any $\w\in\mB$. The minimality of the flow
provides values of time $t_1<\ldots<t_p$ such that
$\W=\sigma_{t_1}(\mB)\cup\ldots\cup\sigma_{t_p}(\mB)$, from
where it follows easily that $H_a$ is always positive and bounded from
below. But this contradicts the last properties mentioned in (3).
\end{proof}
Observe that the previous result shows that the condition
$\sup_{t\le 0}\int_0^t a(\ws)\,ds<\infty$ for $m$-a.e.~$\w\in\W$
in Definition \ref{2.defR} can be replaced by
$\sup_{t\in\R}\int_0^t a(\ws)\,ds<\infty$ for $m$-a.e.~$\w\in\W$.
\par
Note also that, if $a\in\mR_m(\W)$ and $H_a$ and $\W^a$ are the
function and set of \eqref{2.propR}, then the difference between two solutions
of the equation of $x'=a(\wt)\,x$ for $\w\in\W^a$ is
\[
 x(t,\w,x_2)-x(t,\w,x_1)=(x_2-x_1)\exp\Big(\int_0^t a(\ws)\,ds\Big)=
 (x_2-x_1)\,\frac{H_a(\wt)}{H_a(\w)}\:.
\]
The next results shows that for almost every point $\w\in\W^a$,
the set of positive values of time at which the forward semiorbits
seem to coincide (or are \lq\lq indistinguishable")
has positive lower density; and the same property holds
for the set of positive values of time at which
the semiorbits are \lq\lq distinguishable". This fact
will be of relevance later, in the analysis of the type of
Li-Yorke chaos that we will detect for certain
nonhomogeneous linear dissipative scalar equations.
\par
Given a set $\mC\subset[0,\infty)$, we define its {\em lower density\/} as
\[
 d_l(\mC)=\liminf_{t\to\infty}\frac{1}{t}\:l([0,t]\cap\mC)\,,
\]
where $l$ is the Lebesgue measure on $\R$.
Let us take $a\in\mR_m(\W)$, $\w\in\W^a$, $\ep\in(0,1)$, define
\begin{equation}\label{2.dist}
\begin{split}
 \mI_\ep(\w)&:=\{t\ge 0\,|\;H_a(\wt)/H_a(\w)\le \ep\,\}\,,\\
 \mD_\ep(\w)&:=\{t\ge 0\,|\;H_a(\wt)/H_a(\w)\ge 1-\ep\,\}\,,
\end{split}
\end{equation}
and observe that these ones are the sets of values of time we referred
to before.
\begin{teor}\label{2.teordensidad}
Assume that $a\in\mR_m(\W)$, and let $\W^a$ be the set provided by Proposition
{\rm\ref{2.propR}}. Then, for every
$\ep\in(0,1)$ there exists a subset $\W_\ep\subseteq\W^a$ with $m(\W_\ep)=1$
such that, for any $\w\in\W_\ep$,
\begin{itemize}
\item[\rm(i)] the set $\mI_\ep(\w)$ has positive lower density and is relatively dense in $\R^+$.
\item[\rm(ii)] The set $\mD_\ep(\w)$ has positive lower density.
\end{itemize}
\end{teor}
\begin{proof}
(i) Let us define $\mC_n:=\{\w\in\W^a\,|\;H_a(\w)\ge 1/n\}$.
Since $\mC_n\subseteq\mC_{n+1}$ and $\W^a=\bigcup_{n\in\N}\,\mC_n$
(see point (3) of Proposition \ref{2.propR}), we have
$\lim_{n\to\infty}m(\mC_n)=1$, and hence
$m(\mC_n)>0$ for $n\ge n_0$. We will work with a fixed $n\ge n_0$.
We also fix $\ep\in(0,1)$.
Let us take a continuity point $\w_0$ of $H_a$, so that $H_a(\w_0)=0$ (see
again Proposition \ref{2.propR}),
and look for a nonempty open ball $\mB_\ep:=\mB_\W(\w_0,\delta_\ep)$ such that
$H_a(\w)\le\ep/n$ if $\w\in\mB_\ep$. Note that if $\w\in\mC_n$ and
$\ws\in\mB_\ep$ then
$H_a(\ws)/H_a(\w)\le\ep$; that is,
$s\in\mI_\ep(\w)$. Since $\mB_\ep$ is open and $(\W,\sigma)$ is minimal,
$m(\mB_\ep)>0$. Birkhoff's ergodic theorem ensures that, for $m$-almost every
$\w\in\mC_n$,
\[
\begin{split}
 0&<m(\mB_\ep)=\lim_{t\to\infty}\frac{1}{t}\int_0^t\chi_{\mB_\ep}(\ws)\,ds
 =\lim_{t\to\infty}\frac{1}{t}\:l(\{s\in[0,t]\,|\;\ws\in\mB_\ep\})\\
 &\le\liminf_{t\to\infty}\frac{1}{t}\:l(\{s\in[0,t]\,|\;s\in\mI_\ep\})=
 \liminf_{t\to\infty}\frac{1}{t}\:l([0,t]\cap\mI_\ep(\w))=d_l(\mI_\ep)\,.
\end{split}
\]
This proves the assertion concerning the lower density for $m$-almost all
the elements of $\mC_n$, and hence for $m$-almost all the points of $\W^a$.
\par
To check that
$\mI_\ep(\w)$ is relatively dense in $\R$, we deduce from the minimality of the
base flow and the open character of $\mB_\ep$
that there exist positive values of time $t_1<\cdots<t_p$
such that $\W\subset\sigma_{-t_1}(\mB_\ep)\cup\ldots\cup\sigma_{-t_p}(\mB_\ep)$.
In particular, for any $\w\in\W$ there exists $t\in[0,t_p]$ such that
$\wt\in\mB_\ep$. We take $\w\in\W^a$. Given $s\in\R^+$, we look for
$t\in[0,t_p]$ such that $(\ws){\cdot}t=\w{\cdot}(s+t)\in\mB_\ep$, which
ensures that $\wit s=s+t\in\mI_\ep(\w)$. This ensures that
$\mI_\ep(\w)$ is relatively dense in $\R^+$, and completes the proof of (i).
\smallskip\par
(ii) Let us define $\eta:=\inf\{k\in\R\,|\;
m(\{\w\in\W\,|\;H_a(\w)\ge k\})=0\}\le 1$ and $\W_0:=\{\w\in\W\,|\;H_a(\w)\le\eta\}$,
and note that $m(\W_0)=1$. We fix $\ep\in(0,1)$, define
$\Delta_\ep:=\{\w\in\W\,|\;H_a(\w)>(1-\ep)\,\eta\}$,
and observe that the definition of $\eta$ ensures that $m(\Delta_\ep)>0$.
Now we take $\w\in\W^a\cap\W_0$, and note that the
set $\{t\ge 0\,|\;\wt\in\Delta_\ep\}$ is contained in $\mD_\ep(\w)$,
since $H_a(\wt)/H_a(\w)>(1-\ep)\,\eta/H_a(\w)\ge (1-\ep)$.
\par
Birkhoff's ergodic theorem ensures that, for $m$-almost every
$\w\in\W^a\cap\W_0$ (that is, in a set $\wit\W_\ep$ with $m(\wit\W_\ep)=1$),
\[
\begin{split}
 0&<m(\Delta_\ep)=\lim_{t\to\infty}\frac{1}{t}\int_0^t\chi_{\Delta_\ep}(\ws)\,ds
 =\lim_{t\to\infty}\frac{1}{t}\:l(\{s\in[0,t]\,|\;\ws\in\Delta_\ep\})\\
 &\le\liminf_{t\to\infty}\frac{1}{t}\:l(\{s\in[0,t]\,|\;s\in\mD_\ep(\w)\})=
 \liminf_{t\to\infty}\frac{1}{t}\:l([0,t]\cap\mD_\ep(\w))=d_l(\mD_\ep(\w))\,.
\end{split}
\]
This proves (ii).
\end{proof}
\begin{nota}
The set $\mD_\ep(\w)$ of the previous theorem is never relatively dense. To check this,
we take $\w\in\W^a$ and $(t_n)\uparrow\infty$ such that $\wit\w:=\lim_{n\to\infty}\wt_n=
\wit\w$ is a continuity point for the semicontinuous function $H_a$, so that $H_a(\wit\w)=0$.
Then, $\lim_{n\to\infty}H_a(\w{\cdot}(t_n+t))=
\lim_{n\to\infty}H(\wt_n)\exp\big(\int_{0}^{t} a(\w{\cdot}(t_n+s))\,ds\big)=0$ uniformly for $t$
in any compact interval of $\R$. Therefore, given $\ep\in(0,1/2)$ and $t_*>0$, there exists
$n_0$ such that $H(\w{\cdot}(t_{n_0}+t))\le \ep H(\w)$ for all $t\in[0,t_*]$, so that
$[t_{n_0},t_{n_0}+t_*]\cap \mD_\ep(\w)$ is empty. The assertion follows from the
fact that $t_*$ is arbitraritly chosen.
\end{nota}
%%%%%%%%%%%%%%%%%%%%%%%%%%%%%%%%%%%%%%%%%%%%%%%%%%%%%%%%%%%%
%%%%%%%%%%%%%%%%%%%%%%%%%%%%%%%%%%%%%%%%%%%%%%%%%%%%%%%%%%%%
\subsection{Li-Yorke chaos}\label{2.subsecLY}
As already mentioned, in this paper we will deal with
two types of chaos: Li-Yorke (now defined) and Auslander-Yorke
(defined in Subsection \ref{2.subsecAY}).
The minimality of the flow $(\W,\sigma)$ is not assumed in what follows.
\begin{defi}\label{2.defiLiYorke}
Let $(\W,\sigma)$ be a continuous flow on a compact metric space.
Let $\w_1,\w_2$ be two points of $\W$ whose forward $\sigma$-semiorbits
are globally defined. The points $\w_1,\w_2$ form a
{\em positively distal pair\/} for $\sigma$
if $\liminf_{t\to\infty}\dist_{\W}(\sigma_t(\w_1),\sigma_t(\w_2))>0$, and
a {\em positively asymptotic pair\/} if
$\limsup_{t\to\infty}\dist_{\W}(\sigma_t(\w_1),\sigma_t(\w_2))=0$.
The points $\w_1,\w_2$ form
{\em Li-Yorke pair\/} for the flow if the pair is
neither positively distal nor positively asymptotic.
A set $\mS\subseteq\W$ such that every pair of different points of
$\mS$ form a Li-Yorke pair is called
a {\em scrambled set} for the flow. The flow $(\W,\sigma)$
is {\em Li-Yorke chaotic}\/ if there exists an uncountable
scrambled set.
\end{defi}
After the initial description of this type of chaos for a
certain type of transformations in \cite{liyo}, there
have appear more exigent definitions, like that of Li-Yorke
sensitivity in \cite{akko}. That is also the case of the next
definition, particular for skewproduct flows.
\begin{defi}\label{2.defLiYorkmedida}
Let $(\W\times\R,\tau)$ be a skewproduct flow
over a minimal base $(\W,\sigma)$, and let
$\mK\subseteq\W\times\R$ be a $\tau$-invariant compact set. Then
the restricted flow $(\mK,\tau)$ is {\em Li-Yorke fiber-chaotic in measure
with respect to $m\in\mmerg(\W,m)$\/}
if there exists a set $\W_0\subseteq\W$
with $m(\W_0)=1$ such that $\mK$ contains an uncountable scrambled set
of Li-Yorke pairs with first component $\w$ for each $\w\in\W_0$.
\end{defi}
\begin{nota}\label{2.notaLiYorke}
It is clear that, in the case of skewproduct flow
$(\W\times\R,\tau)$, a pair of
points $(\w,x_1),\,(\w,x_2)$ (with common first component) form: a
positively distal pair if and only if
$\liminf_{t\to\infty}|x(t,\w,x_1)-x(t,\w,x_2)|>0$; a positively asymptotic pair
if and only if $\limsup_{t\to\infty}|x(t,\w,x_1)-x(t,\w,x_2)|=0$;
and a Li-Yorke pair if these two conditions fail.
\par
We point our again that the notion of Li-Yorke fiber-chaos in measure
makes only sense in the setting of
skewproduct flows. The same happens with the notion of residually
Li-Yorke chaotic flow, previously analyzed in \cite{bjjo}
and \cite{huyi}. Li-Yorke chaos for nonautonomous
dynamical systems is also the object of analysis in \cite{calo},
\cite{clos} and \cite{nuob10}.
\end{nota}
%%%%%%%%%%%%%%%%%%%%%%%%%%%%%%%%%%%%%%%%%%%%%%%%%%%%%%%%%%%%
%%%%%%%%%%%%%%%%%%%%%%%%%%%%%%%%%%%%%%%%%%%%%%%%%%%%%%%%%%%%
\subsection{Auslander-Yorke chaos}\label{2.subsecAY}
As in Subsection \ref{2.subsecLY}, the minimality of the flow $(\W,\sigma)$ is
not initially required (although we will assume it later to talk about
skewproduct flows).
\begin{defi}\label{2.defiAY}
Let $(\W,\sigma)$ be a continuous flow on a compact metric space. The flow is
{\em topologically transitive\/} if for any two open subsets
$\mU$ and $\mV$ there exists $t>0$ such that $\sigma_t(\mU)\cap\mV$ is nonempty.
The flow is {\em sensitive\/} or {\em $\ep$-sensitive\/} ({\em with respect to
initial conditions\/}) if there exists $\ep>0$ such that for any $\w_1\in\W$ and
$\delta>0$ there exists $\w_2\in\mB_\W(\w_1,\delta)$  such that
$\sup_{t\ge 0}\dist_\W(\w_1{\cdot}t,\w_2{\cdot}t)>\ep$. The flow is
{\em Auslander-Yorke chaotic\/} if it is topologically transitive and sensitive.
\end{defi}
\begin{notas}\label{2.notasuno}
1.~Observe that this concept of chaos relies deeply on the set we are considering.
More precisely, if the restriction of the flow to a $\sigma$-invariant compact
set $\mK\varsubsetneq\W$ is Li-Yorke chaotic, then so is the global flow. But
this property is not true for Auslander-Yorke chaos, since neither the
transitivity nor the sensitivity on $\mK$ are inherited for the containing set
$\W$.
\par
2.~A point $\w_1\in\W$ is {\em $\ep$-sensitive\/} for an $\ep>0$ if for any
$\delta>0$ there exists $\w_2\in\mB_\W(\w_1,\delta)$ such that
$\sup_{t\ge 0}\dist_\W(\w_1{\cdot}t,\w_2{\cdot}t)>\ep$. The flow
$(\W,\sigma)$ is sensitive if there exists a common $\ep>0$ such
every $\w\in\W$ is $\ep$-sensitive, and this definition
is also valid for a flow on a complete metric space.  A point
is {\em sensitive\/} if it is $\ep$-sensitive for some $\ep>0$.
A non sensitive point is called {\em Lyapunov stable}.
\par
3.~In the case of a compact metric space $\W$, topological
transitivity and point transitivity are equivalent, as proved in
e.g.~\cite[Lemma 3]{auyo}. Point transitivity means the existence
of a dense forward semiorbit, which in general is less restrictive.
A compact set $\W$ with a point transitive and sensitive flow
was called {\em chaotic} by Kaplan and Yorke \cite{kayo}.
An exhaustive analysis of the relation among topological
transitivity, point transitivity, and many other a priori stronger
conditions in more general topological spaces is done in \cite[Theorem 1.4]{akca}.
\end{notas}
The next fundamental result is proved by Glasner and Weiss
in \cite[Theorem 1.3]{glwe4} for the case of
a surjective continuous transformation, which provides a
discrete-time semiflow; but its proof can be easily
adapted to the case of a real flow (see also \cite[Proposition 2.4]{gljk}).
Recall that, if $\W$ is a compact metric space, then
the set $\mminv(\W,\sigma)$ of $\sigma$-invariant measures is nonempty,
and that topological transitivity and point transitivity are equivalent properties
(see Remark \ref{2.notasuno}.3): we will simply say {\em transitivity}.
The definition of equicontinuous (or almost periodic) flow on a compact space
is given in Subsection \ref{2.subsecbasic}.
\begin{teor}\label{2.teorAY}
Let $(\W,\sigma)$ be a continuous flow on a compact metric space.
Assume that the flow is transitive, and that $\W$ is the support of a
measure $m\in\mminv(\W,\sigma)$.
Then,
\begin{itemize}
\item[(1)] either the flow is minimal and equicontinuous,
\item[(2)] or it is Auslander-Yorke chaotic.
\end{itemize}
\end{teor}
That is, in the case of a transitive flow on a compact metric space,
uniform stability and Auslander-Yorke chaos
are indeed opposite terms (see Remark \ref{2.enfibra} and observe that the
sensitivity of a flow precludes its stability).
\begin{coro}\label{2.coroAY}
Let $(\W,\sigma)$ be a continuous flow on a compact metric space.
If it is minimal, then it is either equicontinuous or Auslander-Yorke chaotic.
\end{coro}
Recall that an equicontinuous minimal flow is uniquely ergodic, so that
the ergodic uniqueness is also a property required to avoid the presence of
Auslander-Yorke chaos. On the other hand, point transitivity is equivalent to the
fact that $\W$ is the omega limit set of one of its points. In particular,
an Auslander-Yorke chaotic flow is always chain recurrent:
see Subsection \eqref{2.subsecbasic} and \cite[Section 8]{selg}.
\par
Let us now talk about Auslander-Yorke chaos for $\tau$-invariant
compact subsets of $\W\times\R$, where $(\W,\sigma)$ is a minimal
flow on a compact metric space and $(\W\times\R,\tau)$ is the
skewproduct flow projecting on $(\W,\sigma)$
induced by a family of the type \eqref{2.ecescalar}
given by a continuous function $f\colon\W\times\R$ which is locally
Lipschtiz with respect to the state variable $x$.
\par
The papers \cite{john8} and \cite{orta4} describe examples of
families of linear equations $x'=a(\wt)\,x+b(\wt)$ over an almost periodic
(and hence equicontinuous) base
flow $(\W,\sigma)$ for which there exists just one minimal set $\mM$,
which in addition is not a copy of the base. Let us take $r_1<r_2$
such that $\mM\subset\W\times[r_1,r_2]$, and define
$g(x)$ as $(x-r_1)^2$ for $x<r_1$, $0$ for $x\in[r_1,r_2]$, and
$-(x-r_2)^2$ for $x>r_2$. Then the families
$x'=a(\wt)\,x+b(\wt)+g(\wt,x)$, which satisfy all the conditions
which we will assume on Section \ref{3.sec}, define a flow
$(\W\times\R,\tau)$ for which (obviously) $\mM$ is a $\tau$-minimal set,
and the restricted semiflow $(\mM,\tau)$ is Auslander-Yorke chaotic
(see Remark~\ref{2.notasAY}.2 below). In
Section \ref{3.sec}, Theorem \ref{3.teorSup0-chaosAY},
we will establish conditions under which
Auslader-Yorke chaos appears for {\em infinitely many\/} $\tau$-invariant
compact sets. The proof of the result is strongly based on the next theorem.
It describes a branch of $\tau$-invariant compact sets for which the restricted flows
satisfy the initial hypotheses in Theorem \ref{2.teorAY}, which makes them the suitable sets
to detect the presence of Auslander-Yorke chaos. These sets are previously
known to coincide with the support of a $\tau$-ergodic measure.
\par
Before stating Theorem \ref{2.teorAYsec3}, we explain some basic facts used in its proof.
\begin{notas}\label{2.notasAY}
1.~Theorem \ref{2.teorAY} provides additional information
in the context of such a scalar skewproduct flow $(\W\times\R,\tau)$, associated to
the family \eqref{2.ecescalar}. The target is to analyze
the possible presence of Auslander-Yorke chaos for a given
$\tau$-invariant compact set $\mK\subset\W\times\R$ which is transitive and
the support of a $\tau$-invariant set.
Recall that the base flow $(\W,\sigma)$ is always assumed to be minimal, and that
any $\tau$-invariant minimal set is an almost automorphic extension of the base:
see Subsection \ref{2.subsecescalar}. According to Theorem \ref{2.teorAY},
the options for the restricted flow $(\mK,\tau)$ are two: either it is
an equicontinuous minimal flow, or it is Auslander-Yorke chaotic.
Assume that we are in the first case. Then $\mK$ (an almost automorphic
extension of the base, since it is minimal) is necessarily
an equicontinuous copy of the base: see e.g.~\cite[Theorem A or Part II]{shyi4}.
Therefore, in this case, the base is necessarily equicontinuous.
\par
2.~In particular, the restricted flow to such a set $\mK$ is
Auslander-Yorke chaotic whenever either the base flow is not equicontinuous
or $\mK$ is not a copy of the base.
\end{notas}
The omega limit set $\mO_\tau(\w,x)$ of a point $(\w,x)\in\W\times\R$
and the support $\Supp m$ of a $\tau$-invariant measure $m$ are
defined in Subsection \ref{2.subsecbasic},
and the notion of $\tau$-equilibrium appears in Subsection \ref{2.subsecskew}.
The properties required in the function $\eta$ appearing in the next
statement are satisfied, for instance, for the upper and lower cover of
a $\tau$-invariant compact set $\mK\subset\W\times\R$.
%such that the restricted flow $(\mK,\tau)$ is linear.
\begin{teor}\label{2.teorAYsec3}
Let $(\W,\sigma)$ be a minimal flow on a compact metric space, and
let $(\W\times\R,\tau)$ be the flow induced by the family
of equations \eqref{2.ecescalar}, where $f\colon\W\times\R$ is
jointly continuous and locally Lipschitz with respect to the state variable $x\in\R$.
\begin{itemize}
\item[\rm(i)] Let $\mM\subset\W\times\R$ be a minimal set. Then, either
the base flow $(\W,\sigma)$ is equicontinuous and $\mM$ is a copy of the base,
or the restricted flow $(\mM,\tau)$ is Auslander-Yorke chaotic.
\end{itemize}
Let us fix $m\in\mmerg(\W,\sigma)$ and let $\eta\colon\W\to\R$ be a bounded
Borel function such that
\begin{itemize}
\item[-] there exists a $\sigma$-invariant subset $\W_\eta$ with
$m(\W_\eta)=1$ such that $x(t,\w,\eta(\w))=\eta(\wt)$ for all $t\in\R$ and $\w\in\W_\eta$,
\item[-] there exists a continuity point $\w_\eta$ of $\eta$,
\item[-] the graph of $\eta$ is contained in a $\tau$-invariant compact set
$\mA\subset\W\times\R$.
\end{itemize}
Then,
\begin{itemize}
\item[\rm(ii)] $\displaystyle{\int_{\mA}h(\w,x)\,d\mu_\eta:=
\int_\W h(\w,\eta(\w))\,dm}$ for
$h\in C(\mA,\R)$ defines a regular Borel $\tau$-ergodic measure
$\mu_\eta$ concentrated on $\mA$.
\item[\rm(iii)] Let us define $\mS_\eta:=\Supp\mu_\eta$. Then, there exists
$\W_*\subseteq\W_\eta$ with $m(\W_*)=1$ such that $(\bw,\eta(\bw))\in\mS_\eta$
and
\begin{equation}\label{2.sop}
 \mS_\eta=\mO_\tau(\bw,\eta(\bw))
\end{equation}
for all $\bw\in\W_*$. In particular, $\mS_\eta$ is a
$\tau$-invariant pinched compact set, the flow $(\mS,\tau)$ is transitive,
and the set $\mX_\eta\subseteq\mS_\eta$ of $\tau$-generic points with forward
$\tau$-semiorbit dense in $\mS_\eta$ satisfies $\mu_\eta(\mX_\eta)=1$.
\item[\rm(iv)] Either the base flow $(\W,\sigma)$ is equicontinuous and
$\mS_\eta$ is a copy of the base, or the restricted
flow $(\mS_\eta,\tau)$ is Auslander-Yorke chaotic.
\end{itemize}
\end{teor}
\begin{proof}
(i) According to Corollary~\ref{2.coroAY}, either $(\mM,\tau)$ is Auslander-Yorke
chaotic or it is equicontinuous. As explained in Remark~\ref{2.notasAY}.1, in the
second situation, $\mM$ is a copy of the base and hence the base flow is equicontinuous.
\smallskip\par
(ii) This property is a classical result on measure theory, and an easy and nice exercise
for the interested reader.
\smallskip\par
(iii) The $\tau$-invariance and compactness of $\mS_\eta$ are general properties
which follow from the $\tau$-invariance of $\mu_\eta$ and the compactness of $\mA$:
see Subsection \ref{2.subsecbasic}. Let us check that $\mS_\eta$ is a pinched set
which contains a dense forward $\tau$-semiorbit.
\par
Lusin's theorem and the regularity of $m$ provide a compact set
$\mK\subseteq\W_\eta$
with $m(\mK)>0$ such that the restriction
$\eta\colon\mK\to\R$ is continuous. Let us define the set
\[
 \mK_*:=\{\w\in\mK\,|\;m(\mB_\W(\w,\delta)\cap\mK)>0\;\text{for all $\delta>0$}\}\,,
\]
which is obviously closed and hence compact.
Our first goal is checking that $m(\mK-\mK_*)=0$. Since $m$ is regular, it
is enough to prove that $m(\mC)=0$ for any compact subset $\mC\subseteq\mK-\mK_*$.
For any $\w_0\in\mC$ there exists $\delta_{\w_0}>0$ such that
$m(\mB_\W(\w_0,\delta_{\w_0})\cap\mK)=0$. The compactness of $\mC$ provides
a finite number of points $\w_1,\ldots,\w_m$ such that
$\mC\subseteq\mB_\W(\w_1,\delta_{\w_1})\cup\cdots\cup\mB_\W(\w_m,\delta_{\w_m})$.
Hence, $\mC=\mC\cap\mK\subseteq(\mB_\W(\w_1,\delta_{\w_1})\cap\mK)\cup\cdots\cup
(\mB_\W(\w_m,\delta_{\w_m})\cap\mK)$, which ensures that $m(\mC)=0$, as asserted.
Consequently, $m(\mB_\W(\w,\delta)\cap\mK_*)=m(\mB_\W(\w,\delta)\cap\mK)>0$
for any $\w\in\mK_*$ and $\delta>0$.
\par
The compact set $\mK_*$ is separable, so that we can find a countable and dense
subset $\mD:=\{\w_m\,|\;m\ge 1\}\subseteq\mK_*$. We call $\mK_{m,k}:=\mB_\W(\w_m,1/k)\cap\mK_*$
and observe that $m(\mK_{m,k})>0$ for all $m,k\ge 1$, since $\w_m\in\mK_*$.
Therefore, Birkhoff's ergodic theorem provides a $\sigma$-invariant subset $\W_{m,k}\subseteq\W_\eta$ with $m(\W_{m,k})=1$ such that for any $\w\in\W_{m,k}$ there
exists a sequence $(t_n)\uparrow\infty$
with $\wt_n\in\mK_{m,k}$ for any $n\ge 1$.
\par
The set $\W_*:=\bigcap_{m\ge 1,\,k\ge 1}\W_{m,k}\subseteq\W_\eta$
is $\sigma$-invariant satisfies
$m(\W_*)=1$. We fix $\bw\in\W_*$ and will check that \eqref{2.sop} holds and that
$(\bw,\eta(\bw))\in\mS_\eta$.
Before that, observe that these properties ensure that
\begin{itemize}
\item[-] the restricted flow
$(\mS_\eta,\tau)$ is point transitive (and hence topologically transitive,
see Remark \ref{2.notasuno}.3), since the forward $\tau$-semiorbit of
$(\bw,\eta(\bw))$ is dense in $\mS_\eta$;
\item[-] the set $\mS_\eta$ is pinched,
since its section over the continuity point $\w_\eta$ of the map
$\eta$ reduces to the singleton $\{\eta(\w_\eta)\}$;
\item[-] the set of points $\{(\bw,\eta(\bw))\,
|\;\bw\in\W_*\}$ has full measure $\mu_\eta$. Hence
$\mu_\eta(\mX_\eta)=1$ for the set $\mX_\eta$ of statement (iii),
since the set of generic points for $(\mS_\eta,\tau)$
has complete measure (see Subsection \ref{2.subsecbasic}).
\end{itemize}
That is, the proof of (iii) will be complete once checked these two assertions.
\par
We begin by observing that the definitions of $\mK_*$ and $\W_*$ provide $t>0$ such that
$\bw{\cdot}t\in\mK_*$. Since $\bw\in\W_*\subseteq\W_\eta$, we have $\tau(t,\bw,\eta(\bw))=(\bw{\cdot}t,\eta(\bw{\cdot}t))$.
Therefore, $\mO_\tau(\bw,\eta(\bw))=\mO_\tau(\bw{\cdot}t,\eta(\bw{\cdot}t))$,
and $(\bw,\eta(\bw))\in\mS_\eta$ if and only if
$(\bw{\cdot}t,\eta(\bw{\cdot}t))\in\mS_\eta$.
Consequently, it is enough to prove the two pervious assertions for
$\bw\in\mK_*$, which we assume from now on.
\par
We first prove that
\begin{equation}\label{2.estaen}
 (\ww{\cdot}t,\eta(\ww{\cdot}t))\in\mO_\tau(\bw,\eta(\bw))
\end{equation}
for all $\ww\in\mK_*$ and $t\in\R$. To this end, we take $\w_m\in\mD$ and
$\ep>0$, and look for $k>1/\ep$ such that, if $\w\in\mB_\W(\w_m,1/k)\cap\mK_*$,
then $|\eta(\w_m)-\eta(\w)|<\ep$. We also look for $(t_n)\uparrow\infty$ such
that $\bw{\cdot}t_n\in\mK_{m,k}\subseteq\mB_\W(\w_m,1/k)$. Thus,
$\dist_\W(\w_m,\bw{\cdot}t_n)<1/k<\ep$ and $|\eta(\w_m)-\eta(\bw{\cdot}t_n)|<\ep$,
which proves \eqref{2.estaen} for $\ww=\w_m\in\mD$ and $t=0$. The property
for all $\ww\in\mK_*$ and $t=0$ follows from the density of $\mD$, the continuity
of $\eta\colon\mK_*\to\R$, and the closed character of the right set in
\eqref{2.estaen}. Once this is established, we combine $\mK_*\subset\W_\eta$
with the $\tau$-invariance of the omega limit in order to deduce
\eqref{2.estaen} for all $\ww\in\mK_*$ and $t\in\R$.
\par
We define $\mK_\infty:=
\bigcup_{t\in\R}\sigma_t(\mK_*)$. The definition of $\W_*$ ensures that
$\W_*\subseteq\mK_\infty$, and hence $m(\mK_\infty)\ge m(\W_*)=1$.
Note that \eqref{2.estaen} ensures that
$(\w,\eta(\w))\in\mO_\tau(\bw,\eta(\bw))$ whenever
$\w\in\mK_\infty$. This property and the regularity of $\mu_\eta$ yield
\begin{equation}\label{2.cara}
\begin{split}
 &\mu_\eta(\mO_\tau(\bw,\eta(\bw)))\\
 &\qquad=\inf\left\{\int_\mA f(\w,x)\,d\mu_\eta\,|\;
 f\in C(\mA,[0,1]) \text{ with } f|_{\mO_\tau(\bw,\eta(\bar{\w}))}\equiv 1\right\}\\
 &\qquad=\inf\left\{\int_\mA f(\w,\eta(\w))\,dm\,|\;
 f\in C(\mA,[0,1]) \text{ with } f|_{\mO_\tau(\bw,\eta(\bar{\w}))}\equiv 1\right\}\\
 &\qquad\ge\int_\W\chi|_{_{\mK_\infty}}(\w)\,dm=1\,.
\end{split}
\end{equation}
(As usual,
$\chi|_{_{\mB}}$ is the characteristic function of the set $\mB$.)
Hence $\mu_\eta(\mO_\tau(\bw,\eta(\bar(\w)))=1$,
which ensures that $\mS_\eta\subseteq\mO_\tau(\bw,\eta(\bw))$.
\par
Let us now check that
$\mO_\tau(\bw,\eta(\bw))\subseteq\mS_\eta$.
We take $(\ww,\wit x)\in\mO_\tau(\bw,\eta(\bw))$ and
an open neighborhood $\mU\subset\W\times\R$ of $(\ww,\wit x)$, and
will prove that $\mu_\eta(\mU\cap\mA)>0$. Let us take
$\bar t>0$ such that $(\bw{\cdot}\bar t,\eta(\bw{\cdot}\bar t))\in\mU$.
Then $(\bw,\eta(\bw))\in\mV:=\tau_{-\bar t}(\mU)$, which
combined with the continuity of $\mK_*\to\W\times\R\,,\;\w\mapsto(\w,\eta(\w))$
ensures that there exists $\delta>0$ such that $(\w,\eta(\w))\in\mV$
whenever $\w\in\mB(\bw,\delta)\cap\mK_*$. Since
$m(\mB(\bw,\delta)\cap\mK_*)>0$, we conclude as in \eqref{2.cara} that
$\mu_\eta(\mV\cap\mA)>0$,
which combined with the $\tau$-invariance of the measure
ensures that $\mu_\eta(\mU\cap\mA)=\mu_\eta(\mV\cap\mA)>0$.
An easy contradiction argument shows that
$(\ww,\wit x)\in\mS_\mu$, so that \eqref{2.sop} is proved for
the initially chosen point $\bw\in\mK_*$.
In turn, \eqref{2.sop} combined with \eqref{2.estaen}
for $\ww=\bw$ and $t=0$ ensures that $(\bw,\eta(\bw))\in\mS_\eta$.
This completes the proof of the two assertions, and that of (iii).
\smallskip\par
(iv) According to Theorem~\ref{2.teorAY}, either $(\mS_\eta,\tau)$ is Auslander-Yorke
chaotic or it is equicontinuous and minimal. Remark~\ref{2.notasAY}.1
completes the proof of (iv).
\end{proof}
%%%%%%%%%%%%%%%%%%%%%%%%%%%%%%%%%%%%%%%%%%%%%%%%%%%%%%%%%%%%%%%%%%%%%%%%%%
%%%%%%%%%%%%%%%%%%%%%%%%%%%%%%%%%%%%%%%%%%%%%%%%%%%%%%%%%%%%%%%%%%%%%%%%%%
%%%%%%%%%%%%%%%%%%%%%%%%%%%%%%%%%%%%%%%%%%%%%%%%%%%%%%%%%%%%%%%%%%%%%%%%%%
%%%%%%%%%%%%%%%%%%%%%%%%%%%%%%%%%%%%%%%%%%%%%%%%%%%%%%%%%%%%%%%%%%%%%%%%%%
\section{Dynamics for nonhomogeneous linear dissipative equations}\label{3.sec}
Let $(\W,\sigma)$ be a minimal flow on a compact metric space. (This minimality
is an important requisite throughout the whole section.)
Let $a,b\colon\W\to\R$ be continuous functions.
Let us consider the family of scalar nonautonomous equations
\begin{equation}\label{3.familia}
 x'=a(\wt)\,x+b(\wt)+g(\wt,x)\,,\qquad \w\in\W
\end{equation}
with nonhomogeneous linear part, under the following conditions on
the function $g\colon\W\times\R\to\R$ (although
not all of them will be always in force):
\begin{itemize}
\item[\hypertarget{g1}{{\bf g1}}] There exists the partial derivative $g_x$,
and the functions $g,\,g_x\colon\W\times\R\to\R$ are continuous.
%\item[\hypertarget{g1t}{{\bf\~{g1}}}] There exist the partial
%derivatives $g_x$ and
%$g_{xx}$, and the functions $g,\,g_x,\,g_{xx}
%\colon$ $\W\times\R\to\R$ are continuous.
\item[\hypertarget{g2}{{\bf g2}}] There exist real numbers $r_1\le r_2$
such that: $g(\w,x)=0$ if $r_1\le x\le r_2$,
$g(\w,x)>0$ if $x<r_1$ and $g(\w,x)<0$ if $x>r_2$; and
$g_x(\w,r_1)=0$ for all $\w\in\W$ if $r_1=r_2$.
\item[\hypertarget{g3}{{\bf g3}}] $\lim_{x\to\pm\infty}(g(\w,x)/x)=-\infty$ uniformly on $\W$.
\item[\hypertarget{g4}{{\bf g4}}] $g_x(\w,x)\le 0$ whenever $x\notin[r_1,r_2]$.
\item[\hypertarget{g4t}{{\bf\~{g4}}}] $g_x(\w,x)<0$ whenever $x\notin[r_1,r_2]$.
\end{itemize}
\par
The family \eqref{3.familia} is said to be {\em linear dissipative\/} if
$r_1<r_2$, and {\em purely dissipative\/} if $r_1=r_2$.
Theorem \ref{3.teorexisteA}
will justify the use of the term {\em dissipative\/} in both cases. In this
paper, we are more interested in the linear dissipative case, where
we can detect Li-Yorke chaos and Auslander-Yorke chaos. But it is quite easy
to complete our
analysis in order to include the purely dissipative case, just using at
a certain point (in the proof of Theorem \ref{3.teorSup0-pinched}) one
result of \cite{obsa8}.
\par
As explained in Subsection \ref{2.subsecskew}, the family \eqref{3.familia}
induces a local continuous flow $(\W\times\R,\tau)$, given by
\[
 \tau\colon\mU\subseteq\R\times\W\times\R\to\W\times\R\,,\quad
 (t,\w,x_0)\mapsto (\wt,x(t,\w,x_0))\,,
\]
where $t\mapsto x(t,\w,x_0)$ is the maximal solution of \eqref{3.familia}$_\w$
with $x(0,\w,x_0)=x_0$. In addition, the map $x_0\mapsto x(t,\w,x_0)$ is $C^1$
if \hyperlink{g1}{\bf g1} holds.
%Note also that the flow is monotone, since the equation is scalar.
\par
The associated family of homogeneous linear equations
\begin{equation}\label{3.eclineal}
 x'=a(\wt)\,x\,,
\end{equation}
for $\w\in\W$, will play a fundamental role in the proofs of the results.
Let us denote $x_l(t,\w,x_0):=x_0\exp\big(\int_0^ta(\ws)\,ds\big)$,
and let $(\W\times\R,\tau_l)$ be the associated linear flow, so that
$\tau_l(t,\w,x)=(\wt,x_l(t,\w,x_0))$.
\par
We will begin this section by some general results which require neither
the assumption \hyperlink{g4}{\bf g4} on $g$ nor any
condition on the Sacker and Sell spectrum of the linear family \eqref{3.eclineal}.
More precisely,
we establish the existence of global attractor
$\mA$, in Theorem \ref{3.teorexisteA}, and analyze two minimal sets (which may
coincide) determined by the upper and lower covers of $\mA$, in
Theorem~\ref{3.teortapas}. Then we show, in Theorem~\ref{3.teortodos-},
that if any $\tau$-minimal set is uniformly exponentially stable at $+\infty$,
then there is just one of these sets, which coincides with the global attractor.
\par
The condition \hyperlink{g4}{\bf g4} and the assumptions on $\Sigma_a$ will
hence not be in force
until Subsections \ref{3.subsec-} and \ref{3.subsec0}, where we obtain
a much more accurate description of the global dynamics.
\begin{notas}\label{3.notalower}
We will repeatedly use the next properties.
\par
1.~Let us choose
$\w\in\W$ and assume that two maps $t\mapsto\alpha(\wt)$ and
$t\mapsto\beta(\wt)$ are globally defined solutions of the
equation \eqref{3.familia}$_\w$ with
$\alpha(\wt)\le\beta(\wt)$ for any $t\in\R$.
Assume also that $g$ satisfies \hyperlink{g1}{\bf g1} and
\hyperlink{g2}{\bf g2}, and that
$\alpha(\wt)\le r_2$ and $\beta(\wt)\ge r_1$ for all $t\in\R$. Then, the
map $t\mapsto\beta(\wt)-\alpha(\wt)$ is a nonnegative {\em lower solution}
of the linear equation \eqref{3.eclineal}$_\w$
(that is, its derivative satisfies the differential inequality
$x'\le a(\wt)\,x$).
This assertion follows from property \hyperlink{g2}{\bf g2}, which ensures that
$g(\wt,\beta(\wt))\le 0$ and $g(\wt,\alpha(\wt))\ge 0$. A standard comparison argument
shows that, in this case,
$\beta(\wt)-\alpha(\wt)\ge(\beta(\w)-\alpha(\w))\exp\big(\int_0^t a(\ws)\,ds\big)$ for
$t\le 0$ and $\beta(\wt)-\alpha(\wt)\le(\beta(\w)-\alpha(\w))\exp\big(\int_0^t a(\ws)\,ds\big)$
for $t\ge 0$.
In particular, if any point $\w\in\W$ satisfies the initial assumption, then the map
$\W\to\R\,,\;\w\mapsto\beta(\w)-\alpha(\w)$ is a $\tau_l$-subequilibrium. We
referred to this type of relation between lower (or upper) solutions and
subequilibra (or superequilibria) in Subsection \ref{2.subsecskew}.
\par
2.~Note also that a similar result holds for $t\mapsto c(\beta(\wt)-\alpha(\wt))$ if $c>0$.
\par
3.~If, in addition, $g$ satisfies \hyperlink{g4}{\bf g4} and $c>0$, then
$t\mapsto c\,(\beta(\wt)-\alpha(\wt))$ is a nonnegative lower solution of
$x'=a(\wt)\,x$ independently of the area where their graph is contained,
since $c\,(g(\wt,\beta(\wt))-g(\wt,\alpha(\wt))\le 0$.
\end{notas}
By repeating the arguments leading to \cite[Theorem 16]{calo}
(see also \cite[Section 1.2]{chlr}),
one proves the following fundamental result:
\begin{teor}\label{3.teorexisteA}
Assume that $g$ satisfies \hyperlink{g1}{\bf g1} and \hyperlink{g3}{\bf g3},
and let $(\W\times\R,\tau)$
be the flow induced by the family \eqref{3.familia}. Then,
\begin{itemize}
\item[\rm(i)] the flow $\tau$ is bounded dissipative and admits a global attractor
\[
 \mA=\bigcup_{\w\in\W}\big(\{\w\}\times[\alfaa(\w),\betaa(\w)]\big)\,.
\]
In particular, any forward $\tau$-semiorbit is
globally defined and bounded. In addition,
$\alfaa\colon\W\to\R$ and $\betaa\colon\W\to\R$ are respectively lower and
upper semicontinuous $\tau$-equilibria; and the sets of continuity points for the
functions $\alfaa$ and $\betaa$ are residual and $\sigma$-invariant.
\item[\rm(ii)] In addition, these functions can be obtained as the limits
\[
\begin{split}
%ERRATA CORREGIDA
 \alfaa(\w)&=\lim_{t\to\infty}x(t,\w{\cdot}(-t),-\rho_0)\,,\\
 \betaa(\w)&=\lim_{t\to\infty}x(t,\w{\cdot}(-t),\rho_0)\,,
\end{split}
\]
where the constant $\rho_0$ is large enough to guarantee that
$a(\w)\,x+b(\w)+g(\w,x)>0$ whenever $x\le-\rho_0$ and $a(\w)\,x+b(\w)+g(\w,x)<0$ whenever $x\ge\rho_0$.
\item[\rm(iii)] $\mA$ is the union of the all the $\tau$-orbits which are globally defined and bounded.
\end{itemize}
\end{teor}
In the description of the global dynamics there are two $\tau$-minimal
subsets (which may coincide) easily defined from $\mA$ which play a
fundamental role, and which we describe in the next result.
As recalled in Subsection \ref{2.subsecescalar}, given any $\tau$-minimal
set $\mM$:
there exists a residual $\sigma$-invariant subset $\W_\mM\subseteq\W$ at
whose points the functions $\alfam$ and $\betam$ appearing in the description
\eqref{2.M} of $\mM$ are continuous and take the same value, so
that in particular $\mM$ is an almost automorphic extension of the base; and
$\mM$ is a copy of the base if and only if $\alfam$ and $\betam$
are continuous and coincide everywhere.
The fiber-order relation between two $\tau$-minimal sets, denoted as $\mM\le\mN$ or
$\mM<\mN$, is also described in Subsection \ref{2.subsecescalar}.
\begin{teor}\label{3.teortapas}
Assume that $g$ satisfies \hyperlink{g1}{\bf g1} and
\hyperlink{g3}{\bf g3}, let $(\W\times\R,\tau)$
be the flow induced by the family \eqref{3.familia},
and let $\mA$, $\alpha_{\mA}$ and
$\betaa$ be provided by Theorem~{\rm\ref{3.teorexisteA}}.
Let $\W_c$ be the residual set of common
continuity points of $\alfaa$ and $\betaa$. Let us take $\w_0\in\W_c$ and define
\[
\begin{split}
 \malfa&:=\text{\rm closure}_{\W\times\R}\{(\w_0{\cdot}t,\alfaa(\w_0{\cdot}t))\,|\;t\in\R\}\,,\\
 \mbeta&:=\text{\rm closure}_{\W\times\R}\{(\w_0{\cdot}t,\betaa(\w_0{\cdot}t))\,|\;t\in\R\}\,.
\end{split}
\]
Then,
\begin{itemize}
\item[\rm(i)] $\malfa$ and $\mbeta$ are $\tau$-minimal sets and, for any $\w\in\W_c$,
the sections $(\malfa)_\w$ and $(\mbeta)_\w$ are respectively given by
the singletons $\{\alfaa(\w)\}$ and $\{\betaa(\w)\}$. In addition, any
$\tau$-minimal set $\mM$ satisfies $\malfa\le\mM\le\mbeta$.
\item[\rm(ii)] $\mA$ is a pinched compact set if and only if there exists
$\w_0\in\W_c$ such that $\alfaa(\w_0)=\betaa(\w_0)$. In this case,
$\W_c=\{\w\in\W\,|\;\alfaa(\w)=\betaa(\w)\}$.
\end{itemize}
\end{teor}
\begin{proof}
(i) The $\tau$-invariance $\malfa$ follows from
$(\w_0{\cdot}t,\alfaa(\w_0{\cdot}t))=\tau(t,\w_0,\alfaa(\w_0))$; and $\malfa$
is compact, since $\alfaa$ is a bounded function.
Let us take any $\w\in\W_c$ and $(\w,x)\in\malfa$. Then,
$(\w,x)=\lim_{n\to\infty}(\w_0{\cdot}t_n,\alfaa(\w_0{\cdot}t_n))$
for a sequence $(t_n)$, and the continuity of $\alfaa$ at
$\w$ ensures that $x=\alfaa(\w)$. This is,
$\malfa_{\w}=\{\alfaa(\w)\}$, as asserted. To prove the minimality of $\malfa$,
we take a $\tau$-minimal subset $\mM\subseteq\malfa$, so that $\mM_{\w_0}=\malfa_{\w_0}=
\{\alpha(\w_0)\}$. Hence, the definition of $\malfa$ ensures that $\malfa\subseteq\mM$,
which shows that they coincide. The arguments are analogous for $\mbeta$.
Finally, since any $\tau$-minimal set $\mM$ is contained in $\mA$, we have $\alfaa(\w)\le
x\le \betaa(\w)$ whenever $\w\in\W$ and $(\w,x)\in\mM$.
The last statement in (i) follows easily from here.
\smallskip\par
(ii) Assume that $\mA_{\w_0}$ is a singleton for a certain point $\w_0\in\W$,
so that $\mA_{\w_0}=\{\alfaa(\w_0)\}=\{\betaa(\w_0)\}$, and
take a sequence $(\w_n)$ with limit $\w_0$. Any
subsequence $(\w_k)$ has, in turn, a subsequence $(\w_j)$ such that
there exists $\lim_{j\to\infty}\alfaa(\w_j)=x$. The semicontinuity of
$\alfaa$ and $\betaa$ ensure that $\alfaa(\w_0)\le x\le\betaa(\w_0)$,
and hence $x=\alfaa(\w_0)$. This guarantees that $\alfaa$ is continuous at
$\w_0$. The same argument shows that $\betaa$ is continuous at $\w_0$,
so that $\w_0\in\W_c$. In particular, $\{\w\in\W\,|\;\alfaa(\w)=\betaa(\w)\}
\subseteq\W_c$.
\par
Since $\alfaa$ and $\betaa$ are $\tau$-equilibria, they agree at
$\w_0{\cdot}t$ for all $t\in\R$. Let us now take $\w\in\W_c$ and a sequence
$(t_n)$ with $\lim\w_0{\cdot}t_n=\w$.
Then $\alfaa(\w)=\lim\alfaa(\w_0{\cdot}t_n)=\lim\betaa(\w_0{\cdot}t_n)=\betaa(\w)$.
This shows that $\W_c\subseteq\{\w\in\W\,|\;\alfaa(\w)=\betaa(\w)\}$, and completes
the proof of (ii).
\end{proof}
Corollary \ref{2.coroDE} states that a
$\tau$-minimal set $\mM$ is an exponentially
stable at $+\infty$ copy of the base (the graph of the continuous function
$\alfam=\betam$) if and only if its upper Lyapunov exponent is strictly negative.
We can add some more information for families of equations of the type \eqref{3.familia}:
\begin{teor}\label{3.teortodos-}
Assume that $g$ satisfies \hyperlink{g1}{\bf g1} and
\hyperlink{g3}{\bf g3}, and let $(\W\times\R,\tau)$ be the flow induced by the
family of equations \eqref{3.familia}.
Then, the following assertions are equivalent:
\begin{itemize}
\item[\rm(1)] Any $\tau$-minimal set has strictly negative upper
Lyapunov exponent.
\item[\rm(2)] There exists a unique $\tau$-minimal set whose upper Lyapunov
exponent is strictly negative.
\end{itemize}
Assume that this is the case, let $\mA$, $\alpha_{\mA}$ and $\betaa$ be provided
by Theorem~{\rm \ref{3.teorexisteA}}, and let $\malfa$ and $\mbeta$ be provided by
Theorem~{\rm\ref{3.teortapas}}. Then, the attractor $\mA$ is given for the
unique $\tau$-minimal set $\malfa=\mbeta=\{\alfaa\}=\{\betaa\}$,
and it attracts exponentially any $\tau$-orbit as time increases.
\end{teor}
\begin{proof}
Assume that (1) holds. Recall that the existence of a global attractor
ensures that any solution is defined and bounded on a positive half-line
(see Theorem \ref{3.teorexisteA}(i)), which in turn ensures the
existence of its omega limit set. Recall also that (1) ensures that
any $\tau$-minimal set $\mM$ is a uniformly exponentially stable
at $+\infty$ copy of the base: $\mM=\{\eta\}$
(see Corollary \ref{2.coroDE}).
Given one of these sets, we consider its basin of attraction,
\[
 \mB_\mM:=\{(\w,x_0)\,|\;\lim_{t\to\infty}|x(t,\w,x_0)-\eta(\wt)|=0\}\,.
\]
It is easy to check that $\mB_\mM$ is an open set, and that different $\tau$-minimal sets
give rise to disjoint basins of attraction. It is also easy to check that
every point $(\w,x)$ belongs to the basin of attraction of a $\tau$-minimal set contained
in its omega limit set. Therefore, we can write
\[
 \W\times\R=\bigcup_{\mM\text{ is $\tau$-minimal}}\mB_\mM\,,
\]
which is a disjoint union of open sets. Since $\W\times\R$ is connected, we conclude
that there exists a unique $\tau$-minimal set: (2) holds. The converse is trivial.
\par
Therefore, $\malfa=\mbeta$, and
is a copy of the base. It follows from the definitions of these sets that
that the functions $\alfaa,\betaa\colon\W\to\R$ are continuous and equal,
which obviously ensures that $\malfa=\mbeta=\mA$. The last assertion follows
easily from the hyperbolicity of $\mA$ and the fact that it is contained in
the omega limit set of any $\tau$-orbit. The proof is complete.
\end{proof}
We complete this part of general results with a theorem which
characterizes the set of common continuity points of $\alfaa$ and
$\betaa$ in some cases.
\begin{teor}\label{3.teorWc}
Assume that $g$ satisfies \hyperlink{g1}{\bf g1}, \hyperlink{g2}{\bf g2} and
\hyperlink{g3}{\bf g3},
let $(\W\times\R,\tau)$ be the flow induced by the family \eqref{3.familia},
let $\mA$, $\alfaa$ and $\betaa$ be provided by Theorem~{\rm\ref{3.teorexisteA}},
and let $\W_c$ be the (nonempty) set defined in Theorem~{\rm\ref{3.teortapas}}.
Assume also that
there exists a $\tau$-minimal set $\mM\subseteq\W\times[r_1,r_2]$. Then,
\begin{itemize}
\item[\rm(i)] if there exists $\w_0\in\W$ with $\sup_{t\le 0}\int_0^t a(\w_0{\cdot}s)\,ds=\infty$,
then $\w_0\in\W_c$, $\W_c=\{\w\in\W\,|\;\alfaa(\w)=\betaa(\w)\}$, $\mA$ is pinched,
and $\mM=\malfa=\mbeta$ is the unique $\tau$-minimal set.
\item[\rm(ii)] Let $\alfam$ and $\betam$ be defined by \eqref{2.M}, and
assume that $\mM\subset\W\times[r_1,r_2)$ or
$\mM\subset\W\times[r_1,r_2)$.
If there exists $\w_0\in\W$ with
$\sup_{t\le 0}\int_0^t a(\w_0{\cdot}s)\,ds<\infty$,
then $\alfaa(\w_0)<\betaa(\w_0)$ and $\mM\varsubsetneq\mA$.
\end{itemize}
In particular, if $\mA$ is pinched, and if $\mM:=\malfa=\mbeta$
is contained in either $\W\times[r_1,r_2)$ or in $\W\times(r_1,r_2]$, then
$\W_c=\{\w\in\W\,|\;\alfaa(\w)=\betaa(\w)\}=
\{\w\in\W\,|\; \sup_{t\le 0}\int_0^t a(\ws)\,ds=\infty\}$.
\end{teor}
\begin{proof}
(i) It is enough to prove that $\alfaa(\w_0)=\betaa(\w_0)$: if so, $\mA$ is pinched,
and hence Theorem \ref{3.teortapas}(ii)
proves the remaining assertion. The hypothesis $\mM\subseteq\W\times[r_1,r_2]$
guarantees the that the conditions of  Remark \ref{3.notalower}.1 are fulfilled,
and hence
$\betaa(\w_0{\cdot}t)-\alfaa(\w_0{\cdot}t)\ge(\betaa(\w_0)-\alfaa(\w_0))
\exp\big(\int_0^t a(\w_0{\cdot}s)\,ds\big)$ for
$t\le 0$ (see Remark \ref{3.notalower}.1). Since the left-hand term is bounded,
it is necessarily $\alfaa(\w_0)=\betaa(\w_0)$.
\smallskip\par
(ii) We work in the case $\mM\subset\W\times[r_1,r_2)$,
being the proof analogous in the other case.
Recall that $\exp\int_0^t a(\w_0{\cdot}s)\,ds=x_l(t,\w_0,1)$, solution of
\eqref{3.eclineal}$_{\w_0}$. Let us
look for $\ep>0$ such that $\ep\sup_{t\le 0}x_l(t,\w_0,1)\le r_2-
\sup\{\betam(\w)\,|\;\w\in\W\}$, and define $z(t):=\betam(\w_0{\cdot}t)
+\ep\,x_l(t,\w_0,1)$. Then $z(t)$ takes values in $[r_1,r_2]$ for $t\le 0$
(due to $\mM\subset\W\times[r_1,r_2]$ and to the choice of $\ep$), and hence
it solves \eqref{3.familia}$_{\w_0}$ in $(-\infty,0]$, where, consequently,
it agrees with $x(t,\w_0,z(0))$.
%ANTES PON\'{I}A $x(t,\w_0,y(0))$???
Therefore this last solution of
\eqref{3.familia}$_{\w_0}$
is globally defined and bounded (see Theorem \ref{3.teorexisteA}(i)), which
ensures that $(\w_0,z(0))\in\mA-\mM$
% ANTES PON\'{I}A $(\w_0,y_0)\in\mA-\mM$???
(see Theorem \ref{3.teorexisteA}(iii)). This proves (ii).
\smallskip\par
The final statements of the theorem follow from (i), (ii), and
Theorem \ref{3.teortapas}(ii).
\end{proof}
%%%%%%%%%%%%%%%%%%%%%%%%%%%%%%%%%%%%%%%%%%%%%%%%%%%%%%%%%%%%
%%%%%%%%%%%%%%%%%%%%%%%%%%%%%%%%%%%%%%%%%%%%%%%%%%%%%%%%%%%%
\subsection{The case $\sup\Sigma_a<0\,$}\label{3.subsec-}
In the next two subsections, we describe the $\tau$-minimal sets and the possibility of
occurrence of chaos for the family of equations \eqref{3.familia},
assuming condition \hyperlink{g4}{\bf g4} (or \hyperlink{g4t}{\bf\~g4})
in two cases which depend on the Sacker and Sell spectrum $\Sigma_a$ of
\eqref{3.eclineal} in two cases: $\sup\Sigma_a<0$ and $\sup\Sigma_a=0$.
Remark~\ref{2.notaDE} explains that the first situation is equivalent
to the negative character of the upper of exponential dichotomy of the
family \eqref{3.eclineal} (which therefore has exponential dichotomy
over $\W$), and that the second one is equivalent to the
null character of that upper Lyapunov exponent (so that the
linear family does not have exponential dichotomy).
\par
Let us begin with the case $\sup\Sigma_a<0$.
There is not much to say in this situation, in which the conditions
assumed on $a$ and $g$ provide a very simple global dynamics:
\begin{teor}\label{3.teornegativo}
Assume that $g$ satisfies \hyperlink{g1}{\bf g1}, \hyperlink{g2}{\bf g2},
\hyperlink{g3}{\bf g3} and \hyperlink{g4}{\bf g4}, let $(\W\times\R,\tau)$
be the flow induced by the family \eqref{3.familia}, and let $\mA$
be the global attractor for $\tau$ provided by Theorem~{\em \ref{3.teorexisteA}}.
Assume also that $\sup\Sigma_a<0$.
Then, $\mA$ is a uniformly exponentially stable at $+\infty$
copy of the base which attracts exponentially any $\tau$-orbit as time
increases. In particular, $\mA$ is the unique $\tau$-minimal set.
\end{teor}
\begin{proof}
Recall that, if \hyperlink{g1}{\bf g1} holds,
the upper Lyapunov of a $\tau$-minimal set $\mM$ is
\begin{equation}\label{3.ule}
 \gamma_\mM^s=\int_{\mM}\left(a(\w)+g_x(\w,x)\right)\,d\nu_\mM^s
\end{equation}
for a suitable $\tau$-invariant measure $\nu_\mM^s$ on $\W\times\R$.
Therefore,
\begin{equation}\label{3.ule2}
 \gamma_\mM^s\le\int_{\W}a(\w)\,dm^s_\mM\le\sup\Sigma_a\,,
\end{equation}
where $m_\mM^s\in\mminv(\W,\sigma)$ is the $\sigma$-invariant measure onto
which $\nu_\mM^s$ projects.
The first inequality follows from
\eqref{3.ule}, since conditions \hyperlink{g2}{\bf g2} and \hyperlink{g4}{\bf g4}
ensure that $g_x\le 0$; and the second one from
Theorem~\ref{2.teorespectro}.
Therefore, $\gamma_\mM^s<0$ for any $\tau$-minimal set $\mM$ if
$\sup\Sigma_a<0$, and hence the
assertions follow from Theorem~\ref{3.teortodos-}.
\end{proof}
%%%%%%%%%%%%%%%%%%%%%%%%%%%%%%%%%%%%%%%%%%%%%%%%%%%%%%%%%%%%
%%%%%%%%%%%%%%%%%%%%%%%%%%%%%%%%%%%%%%%%%%%%%%%%%%%%%%%%%%%%
\subsection{The case $\sup\Sigma_a=0\,$}\label{3.subsec0}
This final part is devoted to prove that, as advanced in the Introduction,
under the conditions given by \hyperlink{g1}{\bf g1}, \hyperlink{g2}{\bf g2},
\hyperlink{g3}{\bf g3} and \hyperlink{g4t}{\bf\~g4} on $g$,
there are just two possible global dynamics for the flow $(\W\times\R,\tau)$
induced by \eqref{3.familia} when $\sup\Sigma_a=0$, and in one of them
we are able to detect chaotic behavior.
\par
We begin by describing a particularly simple condition under which $\sup\Sigma_a=0$:
the existence of a continuous primitive for $a$: see Definition \ref{2.defcontpri}.
Observe that condition \hyperlink{g3}{\bf g3} is not assumed, since the stated
properties hold independently of the existence of a global attractor.
\begin{teor}\label{3.teorexconpri}
Assume that $g$ satisfies \hyperlink{g1}{\bf g1} and \hyperlink{g2}{\bf g2} and that
the map $a$ admits a continuous primitive. Let $(\W\times\R,\tau)$
be the flow induced by the family \eqref{3.familia}.
Then, $\Sigma_a=\{0\}$ and, in addition,
\begin{itemize}
\item[\rm(i)] any possible $\tau$-minimal set $\mM$ contained in $\W\times[r_1,r_2]$ is
a copy of the base.
\item[\rm(ii)] If $g$ also satisfies \hyperlink{g4}{\bf g4}, any $\tau$-minimal set
$\mM$ is a copy of the base.
\end{itemize}
\end{teor}
\begin{proof}
The fact that $\Sigma_a=\{0\}$ follows easily from Theorem~\ref{2.teorespectro} and
Birkhoff's ergodic theorem. Let $h_a\colon\W\to\R$ be a continuous primitive of $a$, and
$H_a:=e^{h_a}$. Then, for any $\w\in\W$ and $t\in\R$, $H_a(\wt)=
H_a(\w)\exp\big(\int_0^t a(\ws)\,ds\big)$. In other words,
$H_a(\wt)=x_l(t,\w,H_a(\w))$, solution of $x'=a(\wt)\,x$. Note
also that $H_a$ is positive and bounded from below on $\W$.
\par
Let $\alfam$ and $\betam$ be the maps appearing in the description
\eqref{2.M} of $\mM$. The fundamental points in this proof have been
explained in Remark \ref{3.notalower}: if $\mM$ is contained in
$\W\times[r_1,r_2]$ (as we assume in (i)), or if \hyperlink{g4}{\bf g4}
holds (as in (ii)), then $\betam(\wt)-\alfam(\wt)\le
x_l(t,\w,\betam(\w)-\alfam(\w))$ for any $\w\in\W$ whenever $t\ge 0$.
Let us write $\betam(\w)-\alfam(\w)=k(\w)\,H_a(\w)$. Then,
\[
\begin{split}
 k(\wt)\,H_a(\wt)
 &=\betam(\wt)-\alfam(\wt)\le x_l(t,\w,\betam(\w)-\alfam(\w))\\
 &=x_l(t,\w,k(\w)\,H_a(\w))=k(\w)\,H_a(\wt)
\end{split}
\]
whenever $\w\in\W$ and $t\ge 0$.
It follows easily that the continuous map $t\mapsto k(\wt)$ is decreasing
for any $\w\in\W$. Now we fix any $\w\in\W$ and choose $\w_0$ in the common
set of continuity points of $\alfam$ and $\betam$, so that $\alfam(\w_0)=
\betam(\w_0)$. We look for $(t_n)\downarrow-\infty$ such that $\w_0=
\lim_{n\to\infty}\wt_n$. Then, $\lim_{n\to\infty}(\betam(\wt_n)-\alfam(\wt_n))
=\betam(\w_0)-\alfam(\w_0)=0$, which since $H_a$ is bounded from below
ensures that $\lim_{n\to\infty}k(\wt_n)=0$. Consequently,
$k(\w)=0$, which shows that $\alfam(\w)=\betam(\w)$.
%Since its graph determines $\mM$, the function is continuous.
The proof is complete.
\end{proof}
In the rest of the results we do not assume the existence of a continuous
primitive for $a$. On the contrary, we will see in Theorem \ref{3.teorSup0-lamin}
that this property is not a hypothesis but a consequence
of the first one of the dynamical possibilities for the dynamics
described in the Introduction.
And the existence of continuous primitive will be precluded in the analysis
of the possible occurrence of Li-Yorke chaos and Auslander-Yorke chaos
in the second dynamical possibility (in Theorems \ref{3.teorSup0-chaosLY}
and \ref{3.teorSup0-chaosAY}):
one of our hypotheses there will be precisely the absence of continuous
primitive of $a$.
\par
The next result establishes general properties of the minimal sets. As in the
previous one, condition \hyperlink{g3}{\bf g3} is not assumed, since
the description of the attractor is postponed. In particular,
we check that a $\tau$-minimal set which is not a copy of the base, if it
exists, is contained in $\W\times[r_1,r_2]$ (and hence requires $r_1<r_2$:
such a minimal set cannot
exist in the purely dissipative case if $\sup\Sigma_a=0$).
\begin{teor}\label{3.teorSup0-general}
Assume that $g$ satisfies \hyperlink{g1}{\bf g1}, \hyperlink{g2}{\bf g2}
and \hyperlink{g4t}{\bf\~g4}, and let $(\W\times\R,\tau)$
be the flow induced by the family \eqref{3.familia}.
Assume also that $\sup\Sigma_a=0$, let $\mM$ be a $\tau$-minimal set,
let $\alfam,\,\betam\colon\W\to\R$ be the semicontinuous $\tau$-equilibria
associated to $\mM$ by \eqref{2.M}, and let $\W_\mM$ be the set of their common
continuity points. Then,
\begin{itemize}
\item[\rm(i)] there exists $\w\in\W_\mM$ such that $\alfam(\w)<r_1$ if and only if there
there exists $(\w,x)\in\mM$ with $x<r_1$. In this case, $\mM$ is a uniformly exponentially
stable at $+\infty$ copy of the base: $\mM=\{\alfam\}=\{\betam\}$.
\item[\rm(ii)] There exists $\w\in\W_\mM$ such that $\betam(\w)>r_2$ if and only if there
there exists $(\w,x)\in\mM$ with $x>r_2$. In this case, $\mM$ is a uniformly exponentially
stable at $+\infty$ copy of the base: $\mM=\{\alfam\}=\{\betam\}$.
\end{itemize}
Consequently, if $\mM$ is not a copy of the base, then
$r_1<r_2$ and $\mM\subset\W\times[r_1,r_2]$. In addition,
\begin{itemize}
\item[\rm(iii)] $\mM\subseteq\W\times[r_1,r_2]$ if
its upper Lyapunov exponent is $0$.
\item[\rm(iv)] If $\mM\subseteq\W\times[r_1,r_2]$ and either $r_1<r_2$
or $r_1=r_2$ and $g_x(\w,r_1)=0$ for all $\w\in\W$, then
the upper Lypunov of $\mM$ is $0$.
\end{itemize}
\end{teor}
\begin{proof}
We have seen in the proof of Theorem \ref{3.teornegativo} that conditions
\hyperlink{g1}{\bf g1}, \hyperlink{g2}{\bf g2} and \hyperlink{g4}{\bf g4}
guarantee \eqref{3.ule2}, which in turn ensures that the upper Lyapunov exponent
of any $\tau$-minimal set $\mM$ is $\gamma_\mM^s\le 0$
if $\sup\Sigma_a=0$, as we assume in this subsection. This fact will
be used in what follows.
\smallskip\par
(i) Recall that $\mM=\text{closure}_{\W\times\R}\{(\wt,\alfam(\wt))\,|\;t\in\R\}$,
where $\w$ is any point in $\W_\mM$: see Subsection \ref{2.subsecescalar}.
The first assertion in (i) follows easily from here.
Now we will prove that $\mM$ has negative upper Lyapunov exponent $\gamma_\mM^s$.
Recall that the upper Lyapunov exponent is given by
\eqref{3.ule} for a suitable $\tau$-invariant
measure $\nu_\mM^s$, whose support is, due to minimality, the whole of
$\mM$. Let us take $(\w_0,x_0)\in\malfa$ with $x_0<r_1$.
Property \hyperlink{g4t}{\bf\~g4} ensures that $g_x(\w_0,x_0)=-\rho<0$,
so that \hyperlink{g1}{\bf g1} ensures the existence of
an open set $\mB$ of $\W\times\R$ with $\mB\cap\mM$ non empty and on
which $g_x$ is less that $-\rho/2$.
Since $\mB$ is open and $\Supp\nu^s_\mM=\mM$, we have
$\nu_\mM^s(\mB\cap\malfa)>0$. Using this fact and the
property $g_x\le 0$ everywhere, we obtain
$\int_{\mM}g_x(\w,x)\,d\nu_\mM^s\le\int_{\mB\cap\mM}g_x(\w,x)\,d\nu_\mM^s
\le(-\rho/2)\,\nu_\mM^s(\mB\cap\malfa)<0$, and hence $\gamma_\mM^s<
\int_{\W}a(\w)\,dm_\mM$,
where $m_\mM\in\mminv(\W,\sigma)$ is the measure onto which $\nu_\mM^s$
projects. Definition \ref{2.defiexp} and Theorem~\ref{2.teorespectro}
show that $\sup\Sigma_a=0$ yields
$\int_{\W}a(\w)\,dm_\mM\le 0$, and hence $\gamma_\mM^s<0$.
Corollary \ref{2.coroDE} shows that $\mM$ is a uniformly exponentially
stable at $+\infty$ copy of the base, which in turn ensures that $\alfam$ is
continuous and equal to $\betam$, and that its graph is $\mM$.
\smallskip\par
(ii) The proof of this point is analogous, and the consequence of (i) and (ii)
is clear.
\smallskip\par
(iii)\&(iv) Properties (i) and (ii) prove point (iii). To prove (iv),
we take a $\tau$-minimal set $\mM\subseteq\W\times[r_1,r_2]$.
Theorem \ref{2.teorespectro} ensures the existence of
$m^s\in\mmerg(\W,\sigma)$ such that $\int_\W a(\w)\,dm^s=0$.
Let us define $\nu^s$ from $m^s$ by $\int_{\W\times\R}f(\w,x)\,d\nu^s:=
\int_\W f(\w,\alfam(\w))\,dm^s$ for
$f\colon\W\times\R\to\R$ continuous.
It is easy to check that $\nu^s$ is $\tau$-invariant with $\nu^s(\mM)=1$
(i.e.,~$\nu^s\in\mminv(\mM,\tau)$), and that it projects onto $m^s$.
Since, under the conditions in (iv) (see \hyperlink{g2}{\bf g2}),
$g_x\equiv 0$ on $\W\times[r_1,r_2]$, we have
\[
 \int_{\mM}\left(a(\w)+g_x(\w,x)\right)\,d\nu^s=\int_\W a(\w)\,dm^s=0\,,
\]
and, since $\gamma^s_\mM\le 0$ for any $\tau$-minimal set $\mM$,
we deduce that $\gamma^s_\mM=0$.
\end{proof}
The next results plays a fundamental role in the analysis of the occurrence
of Li-Yorke chaos in the second dynamical situation,
carried-on in Theorem \ref{3.teorSup0-chaosLY}. It establishes conditions under which
the attractor is $m$-almost contained in $\W\times[r_1,r_2]$ (i.e., $\mA_\w\subseteq
[r_1,r_2]$ for $m$-almost every $\w\in\W$), where $m\in\mmerg(\W,\sigma)$
satisfies $\int_\W a(\w)\,dm=0$. Recall once again that Theorem \ref{2.teorespectro}
guarantees the existence of such a measure when $\sup\Sigma_a=0$.
Now, for the sake of generality, we simply assume that $0\in\Sigma_a$ and
that $m$ exists. The results is valid for the linear dissipative and
purely dissipative cases.
\begin{teor}\label{3.teorlineal}
Assume that $g$ satisfies \hyperlink{g1}{\bf g1}, \hyperlink{g2}{\bf g2} and
\hyperlink{g3}{\bf g3},
let $(\W\times\R,\tau)$ be the flow induced by the family \eqref{3.familia},
and let $\mA$ be the global attractor for $\tau$ provided
by Theorem~{\rm\ref{3.teorexisteA}}.
Assume also that $0\in\Sigma_a$ and that
$a\in\mR_m(\W)$, where $m\in\mmerg(\W,\sigma)$ satisfies $\int_\W a(\w)\,dm=0$.
And assume finally that there exists a minimal $\mM\subseteq\W\times[r_1,r_2]$.
Then, the $\sigma$-invariant set
\begin{equation}\label{3.Omegal}
 \W_l:=\{\w\in\W^a\,|\;r_1\le\alfaa(\wt)\le\betaa(\wt)\le r_2 \text{ for all $t\in\R$}\}
\end{equation}
satisfies $m(\W_l)=1$.
\end{teor}
\begin{proof}
The ideas are taken from
\cite[Theorem 35]{calo} and \cite[Theorem 5.8]{clos}.
Note that it is enough to check that the two $\sigma$-invariant sets
\[
\begin{split}
 \W_\alpha &:=\{\w\in\W\,|\;
 \text{there exists $t\in\R$ such that $\alfaa(\wt)<r_1$}\}\,,\\
 \W_\beta &:=\{\w\in\W\,|\;
 \text{there exists $t\in\R$ such that $\betaa(\wt)>r_2$}\}
\end{split}
\]
have null measure. We will reason with $\W_\beta$, being the argument
similar in the case of $\W_\alpha$. Let us assume for contradiction
that $m(\W_\beta)>0$. This provides $s>0$ such that
$\W_{\beta,s}:=\{\w\in\W\,|\;
\text{there exists $t\in\R$ with $\betaa(\wt)>r_2+s$}\}\subseteq\W_\beta$
has positive measure. We call
$\W_{\beta,s}^+:=\{\w\in\W\,|\;
\text{there exists $t>0$ with $\betaa(\wt)>r_2+s$}\}\subseteq\W_\beta$.
\par
We use Lusin's theorem to find a compact set
$\mK\subset\W_{\beta,s}$ with positive measure
such that the restrictions of $\betaa$ and $\alfaa$ to $\mK$ are continuous.
Note that $\alfaa(\w)\ne\betaa(\w)$ whenever $\w\in\mK$, since
the hypothesis $\mM\subseteq\W\times[r_1,r_2]$ and the definition
of $\W_\beta$ provide, for any $\w\in\W_\beta$, a time $t\in\R$ such that
$\alfaa(\wt)\le r_2<\betaa(\wt)$.
We will use this property later.
Birkhoff's ergodic theorem ensures that for $m$-a.e.~$\w\in\W$ there exists
$(t_n)\uparrow\infty$ such that $\wt_n\in\mK$, and the regularity of the measure
provides a new compact set $\mC$ with positive measure with the previous property.
Our next goal is proving that $\mC\subset\W_{\beta,s}^+$. First we check the
existence of $\wit t>0$ such that for any $\w\in\mK$ there exists
$t\in[-\wit t,\wit t]$ with $\betaa(\wt)>r_2+s$. This follows easily
from the equality $\betaa(\wt)=x(t,\w,\betaa(\w))$, the continuity of
$\betaa|_{\mK}$ and the compactness of $\mK$. Now we take $\w\in\mC$,
look for $t_n>\wit t$ such that $\wt_n\in\mK$, and look for $t\in[-\wit t,\wit t]$
such that $\betaa((\wt_n){\cdot}t)>r_2+s$. Since $(\wt_n){\cdot}t=\w{\cdot}(t_n+t)$ and
$t_n+t>0$, we conclude that $\w\in\W_{\beta,s}^+$, as asserted.
\par
Let us fix $\w\in\mC$ and $(t_n)\uparrow\infty$ such that $\wt_n\in\mK$ for all $n\in\N$,
and such that there exists $\ww:=\lim_{n\to\infty}\wt_n$ (so that $\ww\in\mK$).
%We can assume without restriction that $t_n-t_{n-1}\ge\wit t$.
We will check that $\lim_{n\to\infty}\exp\big(\int_0^{t_n}a(\ws)\,ds\big)=\infty$,
or, equivalently, that $\lim_{n\to\infty}x_l(t_n,\w,\betaa(\w)-\alfaa(\w))=\infty$.
Before that, observe that this fact contradicts Proposition \ref{2.propR}(2),
since $m(\mC)>0$, and hence it completes the proof.
\par
As established in Remark~\ref{3.notalower}.2,
the fact that $M\subseteq\W\times[r_1,r_2]$
ensures that any $c>0$ determines the lower solution
$t\mapsto c\,(\betaa(\wt)-\alfaa(\wt))$ for the linear equation $z'=a(\wt)\,z$, and hence
that $x_l(t_n,\w,\betaa(\w)-\alfaa(\w))\ge \betaa(\wt_n)-\alfaa(\wt_n)>
\inf_{\w\in\mK}(\betaa(\w)-\alfaa(\w))>0$. (This is the point in which we use
$\alfaa(\w)<\betaa(\w)$ for $\w\in\mK$.) Let us assume for contradiction
that, for a suitable subsequence $(t_k)$, we have
$\lim_{k\to\infty}x_l(t_k,\w,\betaa(\w)-\alfaa(\w))=
c_0\,(\betaa(\ww)-\alfaa(\ww))$, finite, and hence $c_0>0$.
We take $t_{\ww}$ such that $\betaa(\wwt_{\ww})>
r_2+s$, so that $(d/dt)(\betaa(\wwt_{\ww})-\alfaa(\wwt_{\ww}))
<a(\wwt_{\ww})(\betaa(\wwt_{\ww})-\alfaa(\wwt_{\ww}))$.
This ensures the existence of $\ep>0$ and $t_*>t_{\ww}$ such that
$(c_0+\ep)(\betaa(\wwt_*)-\alfaa(\wwt_*))<
x_l(t_*,\ww,c_0\,(\betaa(\ww)-\alfaa(\ww)))$. In turn,
the last inequality and the definition of $c_0$ provide a point
$t_{k_0}$ of the sequence with
$(c_0+\ep)(\betaa(\w{\cdot}(t_{k_0}+t_*))-\alfaa(\w{\cdot}(t_{k_0}+t_*)))<
x_l(t_*,\wt_{k_0},x_l(t_{k_0},\w,\betaa(\w)-\alfaa(\w)))$ $=x_l(t_*+t_{k_0},\w,
\betaa(\w)-\alfaa(\w))$. Now, we write $t_k=t_*+t_{k_0}+s_k$ with $s_k>0$
for large enough $k$. Then,
\[
\begin{split}
 &x_l(t_k,\w,\betaa(\w)-\alfaa(\w))\\
 &\qquad\qquad=
 x_l(s_k,\w{\cdot}(t_*+t_{k_0}),
 x_l(t_*+t_{k_0},\w,\betaa(\w)-\alfaa(\w)))\\
 &\qquad\qquad>
 x_l(s_k,\w{\cdot}(t_*+t_{k_0}),(c_0+\ep)(\betaa(\w{\cdot}(t_{k_0}+t_*))-
 \alfaa(\w{\cdot}(t_{k_0}+t_*))))\\
 &\qquad\qquad
 \ge (c_0+\ep)(\betaa(\wt_k)-\alfaa(\wt_k))\,.
\end{split}
\]
We have used again Remark~\ref{3.notalower}.2 for the last inequality.
Taking limits as $k\to\infty$, we get $c_0\ge c_0+\ep$.
This is the sought-for contradiction. The proof is complete.
\end{proof}
Let us finally describe the two dynamical possibilities in the case
$\sup\Sigma_a=0$, as well as the cases in which we can ensure the
occurrence of Li-Yorke chaos and Auslander-Yorke chaos.
The first possibility, now analyzed, occurs if and only if the maps
$\alfaa$ and $\betaa$ of Theorem \ref{3.teorexisteA}
coincide at no point of $\W$.
The second one, which occurs when $\alfaa$ and $\betaa$ coincide at
(at least) one point of $\W$, is studied in Theorem \ref{3.teorSup0-pinched}.
And the situations in which we are able to detect Li-Yorke chaos and
Auslander-Yorke chaos are described in Theorems \ref{3.teorSup0-chaosLY}
and \ref{3.teorSup0-chaosAY}, which fit in the second dynamical possibility.
\begin{teor}\label{3.teorSup0-lamin}
Assume that $g$ satisfies \hyperlink{g1}{\bf g1}, \hyperlink{g2}{\bf g2},
\hyperlink{g3}{\bf g3} and \hyperlink{g4t}{\bf\~g4},
let $(\W\times\R,\tau)$
be the flow induced by the family \eqref{3.familia},
let $\mA$, $\alpha_{\mA}$ and $\betaa$ be provided by
Theorem~{\rm\ref{3.teorexisteA}}, and let $\W_c$, $\malfa$ and $\mbeta$ be
defined in Theorem~{\rm\ref{3.teortapas}}. Assume also that $\sup\Sigma_a=0$,
and that there exists $\w_0\in\W_c$ such that $\alfaa(\w_0)<\betaa(\w_0)$.
%\comment{A\~{N}ADIR PROPIEDADES QUE NECESITAMOS EN EL SIGUIENTE???}
Then, $\malfa<\mbeta$. In addition,
\begin{itemize}
\item[\rm(i)] $r_1<r_2$: we are necessarily in the linear dissipative case.
\item[\rm(ii)] The map $a$ has a continuous primitive.
\item[\rm(iii)] $\W=\W_c$, $\malfa=\{\alfaa\}$ and $\mbeta=\{\betaa\}$.
\item[\rm(iv)] Any $\tau$-minimal set is the graph $\mM^c$ of the continuous
map $c\alfaa+(1-c)\betaa$ for a $c\in[0,1]$, and has zero upper Lyapunov exponent.
\item[\rm(v)] $\mA=\bigcup_{c\in[0,1]}\mM^c\subseteq\W\times[r_1,r_2]$, and hence
the restriction of $\tau$ to $\mA$ is linear and uniformly stable at $\pm\infty$.
\end{itemize}
\end{teor}
\begin{proof}
The definitions of $\malfa$ and $\mbeta$ ensure that they are different, so that
they are fiber-ordered (see Subsection \ref{2.subsecescalar}): $\malfa<\mbeta$.
The main step of this proof is showing that both of them are contained in
$\W\times[r_1,r_2]$. Let us assume for the moment being that this is the
case, and let us see how to deduce all the assertions of the theorem.
\par
The existence of two different minimal sets contained in $\W\times[r_1,r_2]$
yields $r_1<r_2$, which is property (i).
Let us take $\w\in\W$, $(\w,x_\alpha)\in\malfa$ and $(\w,x_\beta)\in\mbeta$.
Then the map $t\mapsto x(t,\w,x_\beta)-x(t,\w,x_\alpha)$ solves
$x'=a(\wt)\,x$, and it is bounded and also positively bounded from
below. This implies that all the solutions of
$x'=a(\wt)\,x$ are bounded, for every $\w\in\W$, and hence
$a$ has a continuous primitive: see Remark \ref{2.notacontpri}.
This proves (ii). Theorem \ref{3.teorexconpri}(i) shows that $\malfa$ and
$\mbeta$ are copies of the base: the graphs of $\alfaa$ and $\betaa$, respectively.
Therefore, (iii) holds.
Now, let us take $c\in[0,1]$. It is easy to check
that $t\mapsto c\,\alfaa(\wt)+(1-c)\,\betaa(\wt)=\betaa(\wt)+
c\,(\alfaa(\wt)-\betaa(\wt))$ satisfies $x'=a(\wt)\,x+b(\wt)$.
Since its graph remains in $\W\times[r_1,r_2]$, where $g$ vanishes,
we conclude that
$x(t,\w,c\,\alfaa(\w)+(1-c)\,\betaa(\w))=c\alfaa(\wt))+(1-c)\betaa(\wt)$.
That is, the graph of $c\,\alfaa+(1-c)\,\betaa$ is $\tau$-invariant,
and hence it determines
a copy of the base: a $\tau$-minimal set $\mM^c$. And there
are no more $\tau$-minimal sets, as Theorem \ref{3.teortapas} implies:
any other one should be below $\malfa$ or above $\mbeta$,
impossible. Theorem \ref{3.teorSup0-general}(iii) shows that the
upper Lyapunov exponent of $\mM^c$ is 0, which completes the proof of
(iv). Finally, the decomposition of $\mA$ stated in (v) is an easy consequence
of (iv) and the definition of $\malfa$ and $\mbeta$; the
linearity follows from $\mA\subset\W\times[r_1,r_2]$;
and the uniform stability at $\pm\infty$ of the set $\mA$
for the flow $(\mA,\tau)$ follows from the linearity.
\par
So that the proof will be complete once we show that
$\malfa,\,\mbeta\subset\W\times[r_1,r_2]$. We work with
$\mbeta$, assuming for contradiction that this is not the case.
It follows from Theorem~\ref{3.teorSup0-general}(i)\&(ii) that
$\mbeta$ is a copy of the base (i.e.,~$\mbeta=\{\betaa\}$).
Then, there exists at least a $\tau$-minimal set $\mM$ contained in
$\W\times[r_1,r_2]$: if not, and according to
Theorem \ref{3.teorSup0-general}(i)\&(ii), any $\tau$-minimal set
has strictly negative upper Lyapunov exponent; and hence
Theorem \ref{3.teortodos-} ensures that there exists only one
$\tau$-minimal set, which is not the case.
Theorem \ref{3.teortapas}) ensures that $\malfa\le\mM<\mbeta$, so that
$\betaa\ge r_1$. Therefore, there exits $\bw\in\W$ with $\betaa(\bw)>r_2$.
We will make use of this point a few lines below.
\par
Now we will check that $\sup_{t\ge 0}x_l(t,\w,1)=\infty$  for any $\w\in\W$.
A similar argument has been used in the proof of Theorem \ref{3.teorlineal}.
Let $\alfam$ be the map appearing in the description \eqref{2.M} of the minimal
set $\mM\subset\W\times[r_1,r_2]$.
We take a point $\ww\in\W$ of continuity of $\alfam$ (and, of course, of $\betaa$),
and look for $(t_n)\uparrow\infty$ such that $\lim_{n\to\infty}
\wt_n=\ww$. Remark~\ref{3.notalower}.3 ensures that
any $c>0$ determines the lower solution
$t\mapsto c\,(\betaa(\wwt)-\alfam(\wt))$ of the linear equation $z'=a(\wt)\,z$,
and hence that $x_l(t_n,\w,\betaa(\w)-\alfam(\w))\ge \betaa(\wt_n)-\alfam(\wt_n)>
\inf_{\w\in\W}(\betaa(\w)-\alfam(\w))>0$. Let us assume (for contradiction) that
$\lim_{k\to\infty}x_l(t_k,\w,\betaa(\w)-\alfaa(\w))=
c_0\,(\betaa(\ww)-\alfaa(\ww))<\infty$ for certain subsequence $(t_k)$,
so that $c_0>0$. Since $\betaa$ is continuous, $\betaa(\bw)>r_2$, and
$(\W,\sigma)$ is minimal, we can find $t_0>0$ such that $\betaa(\ww{\cdot}t_0)>r_2$.
This property ensures that $(d/dt)(\betaa(\ww{\cdot}t_0)-\alfam(\ww{\cdot}t_0))
<a(\ww{\cdot}t_0)(\betaa(\ww{\cdot}t_0)-\alfam(\ww{\cdot}t_0))$, which in turn provides $\ep>0$ and $t_*>t_0$ such that
\[
\begin{split}
 &(c_0+\ep)(\betaa(\ww{\cdot}t_*)-\alfam(\ww{\cdot}t_*))\\
 &\qquad\qquad< x_l(t_*-t_0,\ww{\cdot}t_0,c_0\,(\betaa(\ww{\cdot}t_0)-\alfam(\ww{\cdot}t_0)))\\
 &\qquad\qquad\le x_l(t_*-t_0,\ww{\cdot}t_0,x_l(t_0,\ww,c_0\,(\betaa(\ww)-\alfam(\ww))))\\
 &\qquad\qquad=x_l(t_*,\ww,c_0\,(\betaa(\ww)-\alfam(\ww)))\,.
\end{split}
\]
The second inequality follows again from Remark \ref{3.notalower}.3. This strict
inequality combined with $\ww=\lim_{t\to\infty}\wt_n$ and with the
definition of $c_0$ allow us to take a point $t_{k_0}$ of the sequence with
$(c_0+\ep)(\betaa(\w{\cdot}(t_{k_0}+t_*))-\alfam(\w{\cdot}(t_{k_0}+t_*)))<
x_l(t_*,\w{\cdot}t_{k_0},x_l(t_{k_0},\w,\betaa(\w)-\alfam(\w)))=x_l(t_*+t_{k_0},\w,
\betaa(\w)-\alfaa(\w))$. Now, we write $t_k=t_*+t_{k_0}+s_k$ with $s_k>0$
for large enough $k$. Then,
\[
\begin{split}
 &x_l(t_k,\w,\betaa(\w)-\alfam(\w))\\
 &\qquad\qquad =
 x_l(s_k,\w{\cdot}(t_*+t_{k_0}),
 x_l(t_*+t_{k_0},\w,\betaa(\w)-\alfam(\w)))\\
 &\qquad\qquad >
 x_l(s_k,\w{\cdot}(t_*+t_{k_0}),(c_0+\ep)(\betaa(\w{\cdot}(t_{k_0}+t_*))-
 \alfam(\w{\cdot}(t_{k_0}+t_*))))\\
 &\qquad\qquad\ge (c_0+\ep)(\betaa(\wt_k)-\alfam(\wt_k))\,.
\end{split}
\]
We have used once more Remark \ref{3.notalower}.3 in the last inequality.
Recall that $\betaa-\alfam$ is continuous at $\ww$.
Taking limits as $k\to\infty$ we get $c_0\ge c_0+\ep$, impossible.
The conclusion is that $\lim_{n\to\infty}x_l(t_n,\w,\betaa(\w)-\alfam(\w))=\infty$,
and hence that $\lim_{n\to\infty}x_l(t_n,\w,1)=\infty$, as asserted.
\par
The previous contradiction shows that
$\sup_{t\ge 0}x_l(t,\w,1)=\infty$ for all $\w\in\W$.
However, since $0\in\Sigma_a$, there exists
at least a point $\ww\in\W$ such that
$\sup_{t\in\R}x_l(t,\ww,1)<\infty$: see
Remarks~\ref{2.notaDE} and \ref{2.notaED}.2. This new contradiction
shows that our initial assumption cannot hold. That is,
$\mbeta$ is contained in $\W\times[r_1,r_2]$.
An analogous argument shows that also
$\malfa$ is contained in $\W\times[r_1,r_2]$.
\end{proof}
The definitions of set of complete measure and of chain recurrent flow, appearing in
the next statement, are given in Subsection \ref{2.subsecbasic}.
\begin{teor}\label{3.teorSup0-pinched}
Assume that $g$ satisfies \hyperlink{g1}{\bf g1}, \hyperlink{g2}{\bf g2},
\hyperlink{g3}{\bf g3} and \hyperlink{g4t}{\bf\~g4},
let $(\W\times\R,\tau)$ be the flow induced by the
family \eqref{3.familia}, let $\mA$, $\alpha_{\mA}$ and
$\betaa$ be provided by Theorem~{\rm\ref{3.teorexisteA}},
and let $\W_c$, $\malfa$ and $\mbeta$ be defined in Theorem~{\rm \ref{3.teortapas}}.
Assume also that $\sup\Sigma_a=0$, and that $\mA$ is a pinched set.
Then, $\W_c=\{\w\in\W\,|\;\alfaa(\w)=\betaa(\w)\}$, and
$\mM:=\malfa=\mbeta$ is the unique $\tau$-minimal set. In addition,
\begin{itemize}
\item[\rm(i)] $\mM\not\subset\W\times[r_1,r_2]$ if and only if
there exists $\w\in\W_c$ such that $\alfaa(\w)<r_1$ or $\betaa(\w)>r_2$. In
this case,  $\mA=\mM$ is a uniformly exponentially stable at $+\infty$
copy of the base: the graph of the continuous function
$\alfaa=\betaa$.
\item[\rm(ii)] If $\mM\subseteq\W\times[r_1,r_2]$, then
its upper Lyapunov exponent is $0$.
\item[\rm(iii)] If $\mM\subseteq\W\times[r_1,r_2]$ and the map $a$ has a continuous
primitive, then $\mA=\mM=\{\alfaa\}=\{\betaa\}$. In addition, in
this case, $\inf\{\alfaa(\w)\,|\;\w\in\W\}=r_1$
and $\sup\{\alfaa(\w)\,|\;\w\in\W\}=r_2$.
\item[\rm(iv)] If $r_1=r_2=:r$ and $\mM=\{r\}$, then $\W_c$ has complete measure.
%\item[\rm(v)] If $r_1<r_2$, $\mM\subset\W\times[r_1,r_2]$ and either
%$\inf\{\alfaa(\w)\,|\;\w\in\W_c\}>r_1$ or $\sup\{\betaa(\w)\,|\;\w\in\W_c\}<r_2$,
%then the map $a$ does not admit a continuous primitive, $\mM
%\varsubsetneq\mA$, and $\W_c\varsubsetneq\W$.
\item[\rm(v)] If $r_1<r_2$, and either $\mM\subset\W\times[r_1,r_2)$
or $\mM\subset\W\times(r_1,r_2]$,
then the map $a$ does not admit a continuous primitive, $\mM
\varsubsetneq\mA$, and $\W_c\varsubsetneq\W$.
\item[\rm(vi)] The restricted flow $(\mA,\tau)$ is chain recurrent.
\end{itemize}
\end{teor}
\begin{proof}
The equality $\W_c=\{\w\in\W\,|\;\alpha(\w)=\beta(\w)\}$
is proved in Theorem \ref{3.teortapas}(ii).
Hence, clearly, for all $\w\in\W_c$, the points $(\w,\alfaa(\w))=
(\w,\betaa(\w))$ belong to any $\tau$-minimal set, and this fact
combined with the definition of $\malfa$ (see Theorem \ref{3.teortapas})
ensures that $\mM:=\malfa$ contains any $\tau$-minimal set. Hence, it is the
unique one.
\smallskip\par
(i)\&(ii) These assertions follow immediately from Theorem \ref{3.teorSup0-general}.
\smallskip\par
(iii) We repeat step by step the proof of Theorem \ref{3.teorexconpri},
working with the map $\betaa-\alfaa$ instead of $\betam-\alfam$.
The conclusion is that $\alfaa$ and $\betaa$ are equal, so that they are continuous
and determine the copy of the base $\mA=\mM$. The last assertion
in (iii) is trivial if $r_1=r_2$ and follows from (v)
(which will be proved independently) if $r_1<r_2$.
\smallskip\par
(iv) This assertion follows from (iii) if $a$ has a
continuous primitive: in this case, $\alfaa=\betaa\equiv r$ and
$\W_c=\W$. Assume that this
is not the case. Then, the change of variables
$y=x-r$ takes \eqref{3.familia} to the purely dissipative family
$y'=a(\wt)\,y+g(\wt,y+r)$ with linear homogeneous part,
for which $\W\times\{0\}$ is the unique minimal set and the
attractor is $\bigcup_{\w\in\W}\{\w\}\times[\alfaa(\w)-r,\betaa(\w)-r]$.
Following the arguments of \cite[Theorem 5.10]{obsa8}, we prove that
the set of points $\w$ at which $\alfaa(\w)-r=\beta(\w)-r=0$
has complete measure. In fact, \cite{obsa8} is devoted to scalar parabolic
partial differential equations, but ours can be understand as one of that type;
and also a symmetric condition is assumed there on $g$, but this condition
does not imply differences in the arguments we refer to, which can be
repeated for $\alfaa-r$ and for $\betaa-r$. Therefore, coming
back to our initial family, we have that $\alfaa$ and $\betaa$ coincide
(and take the value $r$) on a set of complete measure which, as seen at
the beginning of the proof, is $\W_c$.
\smallskip\par
%(v) The unique minimal set is
%$\mM=\text{\rm closure}_{\W\times\R}\{(\wt,\alfaa(\wt))\,|
%\;t\in\R\}=\text{\rm closure}_{\W\times\R}\{(\wt,\betaa(\wt))\,|
%\;t\in\R\}$ for $\w\in\W_c$: see Theorem \ref{3.teortapas}.
%Therefore, the conditions in (v) imply those of
%Theorem \ref{3.teorWc}, which ensures all the assertions.
%\smallskip\par
(v) Since $0\in\Sigma_a$, there exists $\bw\in\W$ such that
$\sup_{t\in\R}\int_0^ta(\bw{\cdot}s)\,ds<\infty$: see
Remarks~\ref{2.notaDE} and \ref{2.notaED}.2.
Theorem \ref{3.teorWc}(ii) ensures that $\mM\varsubsetneq\mA$ and
that $\alfaa(\bw)<\betaa(\bw)$, which ensures that $\W_c\varsubsetneq\W$.
Theorem \ref{3.teorWc} also shows that
$\W_c=\{\w\in\W\,|\; \sup_{t\le 0}\int_0^t a(\ws)\,ds=\infty\}$,
which precludes the existence of continuous primitive for $a$
(since $\W_c$ is nonempty). This completes the proof of (v).
\smallskip\par
\smallskip\par
(vi) Let us fix two points $(\w,x)$ and $(\ww,\wit x)$ in $\mA$,
$\varepsilon>0$, $t_0>0$, and consider three cases which exhaust the
possibilities.
\par
If $(\ww,\wit x)\in\mM$, then
it belongs to the omega limit set of $(\w_1,x_1):=\tau(t_0,\w,x)$,
and hence there exists $t_1>t_0$ such that
$\dist_{\W\times\R}(\tau(t_1,\w_1,x_1),(\ww,\wit x))<\ep$.
The definition of chain recurrence is fulfilled for the chain
$(\w_0,x_0):=(\w,x),\,(\w_1,x_1)$ and $(\w_2,x_2):=(\ww,\wit x)$
(and the times $t_0$ and $t_1$).
\par
Assume that $(\w,x)\in\mM$. We call $(\w_1,x_1):=\tau(t_0,\w,x)\in\mM$, choose
$t_1>t_0$, and observe that $\tau(t_1,\w_1,x_1)$ belongs to the alpha limit set of
$(\ww,\wit x)$, since it belongs to $\mM$.
We take $t_2>t_0$ such that
$\dist_{\W\times\R}(\tau(-t_2,\wit \w,\wit x),\tau(t_1,\w_1,x_1))<\ep$
and call $(\w_2,x_2):=\tau(-t_2,\ww,\wit x)$, so that
$\tau(t_2,\w_2,x_2)=(\ww,\wit x)$.
The definition of chain recurrence is fulfilled for the chain
$(\w_0,x_0):=(\w,x),\,(\w_1,x_1),\,(\w_2,x_2)$ and $(\w_3,x_3):=(\ww,\wit x)$
(and the times $t_0$, $t_1$, and $t_2$).
\par
Finally, if none of the points belongs to $\mM$, we construct the chain from
$(\w,x)$ to $(\bw,\bar x)$ through any point $(\bw,\bar x)\in\mM$.
This completes the proofs of (vi) and of the theorem.
\end{proof}
\begin{nota}
In the purely dissipative case considered in point (iv),
the set $\W_c$ can be $\W$ (and hence the attractor agrees with
$\{r\}$). This is the simplest situation. But it is
also possible that $\W_c\varsubsetneq\W$, in which case
the dynamics is much more complex
An example of this is given by
the family obtained by the hull procedure (explained in the Introduction)
from the equation $x'=(1/2)(a(t)\,x-x^3)$, where $a(t)=\wit a(-t)$ for
an almost periodic function $\wit a\colon\R\to\R$ with
zero mean value and whose integral $\int_0^t \wit a(s)\,ds$
grows like $t^\mu$ as $t$ increases, for some $0<\mu<1$.
The interested reader is referred to \cite[Example 5.13]{clos}
for the details, as well as for references in which functions
$\wit a$ with the required properties are constructed.
\end{nota}
\begin{nota}\label{3.notachain}
By reviewing the proof of Theorem \ref{3.teorSup0-pinched}(vi), we observe that
the property is general: any flow on a compact metric space
admitting a unique minimal set is chain recurrent.
\end{nota}
The framework of the next theorems, concerning the presence of chaos,
is that of point (v) of the previous one. In particular, it requires the family
\eqref{3.familia} to be in the linear dissipative case (i.e., $r_1<r_2$).
The nonempty set $\mR_m(\W)$ is described in Subsection \ref{2.subsecR}: Theorem
\ref{2.teornovacio} ensures the existence of functions $a\in\mR_{m}(\W)$
with $\Sigma_a=0$ for any $m\in\mminv(\W,\sigma)$.
Recall that when $\mA$ is pinched, there
exists just one $\tau$-minimal set: see Theorem~\ref{3.teorSup0-pinched}.
The scope of the properties
stated in these two results is analyzed after their proofs.
\begin{teor}\label{3.teorSup0-chaosLY}
Assume that $g$ satisfies \hyperlink{g1}{\bf g1}, \hyperlink{g2}{\bf g2},
\hyperlink{g3}{\bf g3} and \hyperlink{g4t}{\bf\~g4},
let $(\W\times\R,\tau)$ be the flow induced by the family \eqref{3.familia},
let $\mA$, $\alpha_{\mA}$ and $\betaa$ be provided by Theorem~{\rm\ref{3.teorexisteA}},
and let $\W_c\subseteq\W$ be the residual set provided by Theorem~{\rm \ref{3.teortapas}}.
Assume also that $r_1<r_2$ and that
\begin{itemize}
%\item[-] $\mA$ is a pinched set, the only
%$\tau$-minimal set $\mM$ is contained in $\W\times[r_1,r_2]$, and either
%$\inf\{\alfaa(\w)\,|\;\w\in\W_c\}>r_1$ or $\sup\{\betaa(\w)\,|\;\w\in\W_c\}<r_2$;
\item[-] $\mA$ is a pinched set, and the only
$\tau$-minimal set $\mM$ is contained either in $\W\times[r_1,r_2)$
or in $\W\times(r_1,r_2]$;
\item[-] $\sup\Sigma_a=0$ and $a\in\mR_{m}(\W)$ for a measure
$m\in\mmerg(\W,\sigma)$.
\end{itemize}
Then,
\begin{itemize}
\item[\rm(i)] $m(\W_c)=0$, and the restricted flow $(\mA,\tau)$ is
chain recurrent and Li-Yorke chaotic in measure with respect to $m$.
More precisely, the $\sigma$-invariant set subset $\W_{LY}\subseteq\W$ of
points $\w$ such that the section $\mA_\w$ is a nondegenerate interval
and the set $\{\w\}\times\mA_\w$ is scrambled,
satisfies $\W_{LY}\subseteq\W-\W_c$ and $m(\W_{LY})=1$.
\item[\rm(ii)] For every
$\ep\in(0,1)$ there exists a subset $\W_\ep\subseteq\W_{LY}$ with $m(\W_\ep)=1$
such that, for any $\w\in\W_\ep$, the set
\[
 \qquad \{t\ge 0\,|\;|x(t,\w,x_2)-x(t,\w,x_1)|\le \ep\,|x_2-x_1|\;\text{ if
 $\;(\w,x_1),\,(\w,x_2)\in\mA$}\}
\]
has positive lower density and is relatively dense in $\R^+$; and the set
\[
 \;\,\quad\qquad \{t\ge 0\,|\;|x(t,\w,x_2)-x(t,\w,x_1)|\ge (1-\ep)\,|x_2-x_1|\;
 \text{\,if $\,(\w,x_1),\,(\w,x_2)\in\mA$}\}
\]
has positive lower density.
\end{itemize}
\end{teor}
\begin{proof}
%Note that Theorem~\ref{3.teorWc} shows that $\W_c\ne\W$.
%\smallskip\par
(i) The chain recurrence of $(\mA,\tau)$ is guaranteed by Theorem
\ref{3.teorSup0-pinched}(vi). Let us take $\w\in\W-\W_c$ and a pair of points
$(\w,x_1),(\w,x_2)\in\mA$ with $x_1\ne x_2$, choose $\w_0\in\W_c$,
recall that $\mA_{\w_0}=\{x_0\}$, and choose a suitable sequence $(t_n)$ such that
$\lim_{n\to\infty}\wt_n=\w_0$ and there exist the two limits
$x(t_n,\w,x_1)$ and $x(t_n,\w,x_2)$. These limits must coincide with $x_0$,
so that any pair of points of $\mA$ sharing the first component form a non positively
distal pair.
\par
To look for non positively asymptotic pairs requires some more work.
Recall that, if $a\in\mR_{m}(\W)$, then
$\sup_{t\in\R}\exp\big(\int_0^ta(\ww{\cdot}s)\,ds\big)<\infty$
whenever $\w$ belongs to a $\sigma$-invariant set
$\W^a\subset\W$ with $m(\W^a)=1$: see Proposition \ref{2.propR}.
Theorem \ref{3.teorWc}(ii) ensures that $\W^a\subseteq\W-\W_c$, and
and hence $m(\W_c)=0$. Theorem \ref{3.teorlineal} shows that also the
$\sigma$-invariant set $\W_l\subseteq\W^a$ defined by \eqref{3.Omegal}
satisfies $m(\W_l)=1$.
Let us take $(\w,x_1),(\w,x_2)\in\mA$ with $\w\in\W_l$.
Then, $t\mapsto x(t,\w,x_1)-x(t,\w,x_2)$ solves
$x'=a(\wt)\,x$, and hence
\begin{equation}\label{3.paravi}
\begin{split}
 &|x(t,\w,x_1)-x(t,\w,x_2)|=|x_l(t,\w,x_1-x_2)|\\
 &\qquad\qquad=
 |x_1-x_2|\exp\int_0^t a(\ws)\,ds=|x_1-x_2|\,\frac{H_a(\wt)}{H_a(\w)}\,,
\end{split}
\end{equation}
where $H_a\colon\W\to[0,1]$ is the bounded function associated to $a$ by
Proposition \ref{2.propR}.
Lusin's theorem and the regularity of $m$ provide a compact set
$\mK\subset\W_l$ with positive measure such that
the restriction of $H_a$ to $\mK$ is continuous, and Birkhoff's ergodic theorem
provides a set $\W^0\subseteq\W$ with $m(\W^0)=1$
such that, if $\w\in\W^0$, then there exists $(t_n)\uparrow\infty$
such that $\wt_n\in\mK$ for all $n\in\N$.
In particular, $\W^0\subseteq\W^l$, since $\W^l$ is $\sigma$-invariant.
Since $H_a$ is globally bounded
and strictly positive at the points of $\W^a\supseteq\W_l$ (see again
Proposition \ref{2.propR}), and
continuous on $\mK\subset\W_l$, we conclude that there exists $\kappa_\w>0$ such that,
whenever $\w\in\W^0$ and $x_1,x_2$ are two different points
of the nondegenerate interval $\mA_{\w}$,
\[
 |x(t_n,\w,x_1)-x(t_n,\w,x_2)|\ge \kappa_\w\,|x_1-x_2|>0
\]
for a sequence $(t_n)\uparrow\infty$.
This shows that $(\w,x_1)$ and $(\w,x_2)$ form a non positively
asymptotic pair.
\par
Altogether, we have proved that the set $\mA_{\w}$ is a nondegenerate
interval with $\{\w\}\times\mA_\w$ scrambled for $m$-almost all $\w\in\W$.
The $\sigma$-invariance of the set $\W_{LY}\supseteq\W^0$
formed by these points is a clear consequence of the definition of scrambled set,
and this completes the proof of (i).
\smallskip\par
(ii) Equality \eqref{3.paravi} and the definitions \eqref{2.dist}
of the sets $\mI_\ep(\w)$ and $\mU_\ep(\w)$ associated to the function $a$
show that, if $\w\in\W_{AY}\cap\W_l$, then
\[
\begin{split}
 &t\in\mI_\ep(\w),\;(\w,x_1),\,(\w,x_2)\in\mA\;\Rightarrow\;
 |x(t,\w,x_1)-x(t,\w,x_2)|\le \ep\,|x_1-x_2|\,,\\
  &t\in\mD_\ep(\w),\;(\w,x_1),\,(\w,x_2)\in\mA\;\Rightarrow\;
 |x(t,\w,x_1)-x(t,\w,x_2)|\ge (1-\ep)\,|x_1-x_2|\,.
\end{split}
\]
Therefore, Theorem~\ref{2.teordensidad} proves (ii).
\end{proof}
\begin{teor}\label{3.teorSup0-chaosAY}
Assume the same hypotheses as in Theorem {\rm\ref{3.teorSup0-chaosLY}},
and let $\W_l$ be the $\sigma$-invariant set with
$m(\W_l)=1$ defined by \eqref{3.Omegal}.
Let us define $\eta_c=c\,\alfaa+(1-c)\,\betaa$ for $c\in[0,1]$. Then,
\begin{itemize}
\item[\rm(i)] $\displaystyle{\int_{\mA}h(\w,x)\,d\mu_c:=\int_\W h(\w,\eta_c(\w))\,dm}$ for
 $h\in C(\mA,\R)$ defines a regular Borel $\tau$-ergodic measure
 $\mu_c$ concentrated on $\mA$.
\end{itemize}
Let us define $\mS_c:=\Supp\mu_c$ for $c\in[0,1]$. Then,
\begin{itemize}
\item[\rm(ii)] there exists a $\sigma$-invariant set
$\W_{AY}\subseteq\W_l$ with $m(\W_{AY})=1$ such that
$(\bw,\eta_c(\bw))\in\mS_c$ and
\begin{equation}\label{3.Sc}
 \mS_c=\mO_\tau(\bw,\eta_c(\bw))
\end{equation}
for all $\bw\in\W_{AY}$ and $c\in[0,1]$. In particular,
$(\mS_c,\tau)$ is a transitive flow on a pinched compact set for
any $c\in[0,1]$.
\item[\rm(iii)] One of the following situations holds:
\begin{itemize}
\item[\rm(1)] $\mM$ is a copy of the base, in which case there exists just
a $c_0\in[0,1]$ such that $\mM=\mS_{c_0}$, and the restricted flow
$(\mS_c,\tau)$ is Auslander-Yorke chaotic for any $c\in[0,1]$, $c\ne c_0$.
\item[\rm(2)] $\mM$ is not a copy of the base, in which case the restricted flow
$(\mS_c,\tau)$ is Auslander-Yorke chaotic for any $c\in[0,1]$.
\end{itemize}
\item[\rm(iv)] $\wit\mS:=\bigcup_{c\in[0,1]}\mS_c$ is a compact $\tau$-invariant
subset of $\mA$, all its points are sensitive,
the restricted flow $(\wit\mS,\tau)$ is chain recurrent,
%\[
% \mS=\bigcup_{\w\in\W}\big(\{\w\}\times[\alpha_{\mS}(\w),\beta_{\mS}(\w)]\big)\,,
%\]
%where $\alpha_{\mS}\colon\W\to\R$ and $\beta_{\mS}\colon\W\to\R$ are
%respectively lower and upper semicontinuous $\tau$-equilibria, with %$\alfaa(\w)=\alpha_{\mS}(\w)=\beta_{\mS}(\w)=\beta_{\mA}(\w)$
%for any $\w\in\W_c$, and $\alfaa(\w)=\alpha_{\mS}(\w)<\beta_{\mS}(\w)=\beta_{\mA}(\w)$
%for any $\w\in\W_{AY}$; that is,
$\mA_\w=\wit\mS_\w$ for every $\w\in\W_{AY}\,\cup\,\W_c$,
$\wit\mS:=\Supp\wit\eta$ for the measure $\wit\eta\in\mminv(\mA,\tau)$
given by ${\displaystyle\int_\mA h(\w,x)\,d\wit\mu:
=\int_0^1\int_\W h(\w,\eta_c(\w))\,dm\,dc}$ for $h\in C(\mA,\R)$,
and there exists a dense $\tau$-invariant subset $\wit\mX\subseteq\wit\mS$ of
$\tau$-generic points.
\end{itemize}
\end{teor}
\begin{proof}
(i) Theorem \ref{3.teorlineal} allows us to assert that
\[
\begin{split}
 x(t,\w,\eta_c(\w))&=x(t,\w,c\,\alfaa(\w)+(1-c)\,\betaa(\w))\\
 &=x_l(t,\w,c\,\alfaa(\w)+(1-c)\,\betaa(\w))\\
 &=c\,x_l(t,\w,\alfaa(\w))+
 (1-c)\,x_l(t,\w,\betaa(\w))\\
 &=c\,\alfaa(\wt)+(1-c)\,\betaa(\wt)=\eta_c(\wt)
\end{split}
\]
for all $t\in\R$ and $\w\in\W_l$. Therefore, $\eta_c$ satisfies the
conditions assumed in Theorem~\ref{2.teorAYsec3}, whose point (ii) proves (i).
\smallskip\par
(ii) By reviewing the proof of Theorem \ref{2.teorAYsec3}(iii),
we first check that we can take as starting point a compact set $\mK\subset\W_l$
such that $\alfaa,\,\betaa\colon\mK\to\R$ are continuous, so that also
$\eta_c\colon\mK\to\R$ is continuous for all $c\in[0,1]$. Second,
we observe that this property combined with the $\sigma$-invariance of
$\W_l$ ensures that the set $\W_{AY}:=\W_*$ constructed
from $\mK$ turns out to be common for all $c\in[0,1]$, and is contained
in $\W_l$. Therefore, the assertions in (ii) follow from
Theorem \ref{2.teorAYsec3}(iii).
\smallskip\par
(iii) Assume that $\mM$ is a copy of the base. We fix $\bw\in\W_{AY}$ and
choose the unique $c_0\in[0,1]$ such that $\mM_{\bw}=\{\eta_{c_0}(\bw)\}$.
Then, \eqref{3.Sc} ensures that $\mS_{c_0}=
\mO_\tau(\bw,\eta_c(\bw))\subset\mM$, so that $\mS_{c_0}=\mM$.
In addition, if $c\in[0,1]-\{c_0\}$, then
$(\bw,\eta_c(\bw))\in\mS_c-\mM$,
and hence $\mS_c\varsupsetneq\mM$:
these sets $\mS_c$ are not copies of the base.
In the case that $\mM$ is not a copy of the base,
$\mS_c\supseteq\mM$ is not a copy of the base for any
$c\in[0,1]$. According to Remark \ref{2.notasAY}.2, the restricted flows
$(\mS_c,\tau)$ are Auslander-Yorke chaotic whenever
$\mS_c$ is not a copy of the base. This proves the assertions in (iii).
(Incidentally, note that, in the first situation, $(\mS_{c_0},\tau)$
is also Auslander-Yorke chaotic unless the base flow $(\W,\sigma)$ is
equicontinuous.)
\smallskip\par
(iv) Let us check that $\wit\mS:=\bigcup_{c\in[0,1]}\mS_c$ is closed.
We fix $\bw\in\W_{AY}$ and
take $(\w_0,x_0):=\lim_{n\to\infty}(\w_nx_n)$ with $(\w_n,x_n)\in\mS_{c_n}=
\mO_{\tau}(\bw,\eta_{c_n}(\bw))$. Let us take a subsequence $(c_j)$ of $(c_n)$
such that there exists $c_0:=\lim_{j\to\infty}c_j$. We will prove that
$(\w_0,x_0)\in\mO_{\tau}(\bw,\eta_{c_0}(\bw))$.
We call $k:=\sup_{\w\in\W}|\alfaa(\w)-\betaa(\w)|$ and note that
$\sup_{\w\in\W}|\eta_{c_j}(\w)-\eta_{c_0}(\w)|=k\,|c_j-c_0|$.
For each $j\in\N$, we look for $t_j>0$ such that
$\dist_{\W\times\R}\big((\w_j,x_j),$ $(\bw{\cdot}t_j,\eta_{c_j}(\bw{\cdot}t_j))\big)
<1/j$. Then, $\dist_{\W}\big(\w_0,\bw{\cdot}t_j\big)\le
\dist_{\W}\big(\w_0,\w_j\big)+\dist_{\W}\big(\w_j,\bw{\cdot}t_j\big)$,
with limit 0; and $|x_0-\eta_{c_0}(\bw{\cdot}t_j)|\le
|x_0-\eta_{c_j}(\bw{\cdot}t_j)|+|\eta_{c_j}(\bw{\cdot}t_j)-\eta_{c_0}(\bw{\cdot}t_j)|
\le |x_0-\eta_{c_j}(\bw{\cdot}t_j)|+k\,|c_j-c_0|$, also with limit 0.
That is, $(\w_0,x_0)=\lim_{j\to\infty}(\bw{\cdot}t_j,\eta_{c_0}(\bw{\cdot}t_j))
\in\mO_{\tau}(\bw,\eta_{c_0}(\bw))=\mS_{c_0}\subseteq\wit\mS$.
\par
Therefore, $\wit\mS$ is closed, and hence, since $\wit\mS\subseteq\mA$,
it is a compact pinched set. Observe that if we are in the situation (1)
of point (iii), then $\mS=\bigcup_{c\in[0,1]-\{c_0\}}\mS_c$, since
$\mS_{c_0}=\mM\subset\mS_c$ for any $c\in[0,1]$. Therefore,
all the points if $\mS$ are sensitive (see Remark \ref{2.notasuno}.2).
Clearly, $\wit\mS$ is $\tau$-invariant, since each set $\mS_c$ is $\tau$-invariant.
Consequently, it is chain recurrent: see
Remark~\ref{3.notachain}. If $\w\in\W_{AY}$ then $\eta_c(\w)\in\wit\mS_{\w}$ for
any $c\in[0,1]$, and $\mA_\w=[\alfaa(\w),\betaa(\w)]=[\eta_0(\w),\eta_1(\w)]
\subseteq\wit\mS_\w\subseteq\mA_\w$, so that the sections coincide.
If $\w\in\W_c$, then $\mA_\w$ is a singleton, and
hence $\mA_\w=\wit\mS_\w$ also in this case.
\par
Note now that ${\displaystyle\int_\mA h(\w,x)\,d\wit\mu
=\int_0^1\int_\mA h(\w,x)\,d\eta_c\,dc}$ for $h\in C(\mA,\R)$.
It is easy to deduce from this property that $\wit\mu$ is a $\tau$-invariant
(regular) measure, and from the regularity that
$\wit\mu(\mK)
\ge\int_0^1\mu_c(\mK)\,dc$ for any compact set $\mK\subset\mA$.
In particular,
$\wit\mu(\wit\mS)=1$, which in turn ensures that $\Supp\wit\mu
\subseteq\wit\mS$. To check that $\Supp\wit\mu\supseteq\wit\mS$,
we assume for contradiction that
$\mU:=\wit\mS-\Supp\wit\mu$ is nonempty, choose $(\w_0,x_0)\in\mU$
look for an open set $\mV\subset\W\times\R$ such that $\mU=\mV\cap\wit\mS$,
and take $\delta>0$ such that $\mB_{\W\times\R}((\w_0,x_0),\delta)\subseteq\mV$.
Now we look for $c_0\in[0,1]$ such that $(\w_0,x_0)\in\mS_{c_0}$,
take $\bw\in\W_{AY}$, and look for $t>0$ such that
$\dist_{\W\times\R}\big((\w_0,x_0),(\bw{\cdot}t,\eta_{c_0}(\bw{\cdot}t))\big)<\delta/2$.
Then, $\dist_{\W\times\R}\big((\w_0,x_0),(\bw{\cdot}t,\eta_c(\bw{\cdot}t))\big)<
\delta/2+|\eta_{c_0}(\bw{\cdot}t)-\eta_c(\bw{\cdot}t)|<\delta$ if
$c$ is close enough to $c_0$, so that $(\w_0,x_0)\in\mU_c:=\mV\cap\mS_c$
for these values of $c$. Therefore, $\mu_c(\mU)\ge
\mu_c(\mU_c)>0$ for a set of values of $c$ with positive Lebesgue measure,
which ensures that $\mu(\mU)>0$, impossible.
\par
It remains to prove that the subset of $\tau$-generic points of $\wit\mS$ is
dense. Let us take an open set $\mU$ of $\wit\mS$, so that $\wit\mu(\mU)>0$.
Since the set $\wit X\subseteq\wit\mS$ of $\tau$-generic points has complete measure,
$\wit\mu(\wit\mX\cap\mU)>0$, and hence there are generic points in $\mU$.
Clearly the subset of generic points of $\wit\S$
is $\tau$-invariant. The proof is complete.
\end{proof}
Note that, unlike the set $\W_{LY}$ of Theorem \ref{3.teorSup0-chaosLY},
the set $\W_{AY}$ of Theorem \ref{3.teorSup0-chaosAY} depends on the measure
$m$ of its statement.
\par
Let us make a short analysis of the previous results.
Regarding Li-Yorke chaos, we point out again that the
set of Li-Yorke pairs that we detect is incomparably larger
than what Definition \ref{2.defiLiYorke} requires.
More precisely, as Theorem \ref{3.teorSup0-chaosLY}(i) proves,
for $m$-almost every point of $\W$ we obtain a scrambled set which
can be identified with a nondegenerate interval, incomparably
larger than a simply uncountable set.
\par
Moreover, Theorem \ref{3.teorSup0-chaosLY}(ii) shows that,
for $m$-almost every point $\w\in\W$, the set
of positive values of time at which the forward
$\tau$-semiorbits of points in $\{\w\}\times\mA_\w$
seem to coincide (or are \lq\lq indistinguishable")
has positive density in $\R^+$; and the same property holds
for the set of positive values of time at which
the semiorbits are \lq\lq distinguishable".
\par
Observe also that, under the hypotheses of Theorem \ref{3.teorSup0-chaosLY},
a function $a\in C(\W,\R)$ may belong to the set $\mR_{\wit m}(\W)$ for a measure
$\wit m\in\mmerg(\W,\sigma)$ different from $m$. This is in fact the case
whenever $\wit m(\{\w\in\W\,|\;\sup_{t\le 0}\int_0^t a(\ws)\,ds<\infty\})=1$,
since this property combined with $\int a(\w)\,d\wit m\le 0$ (in turn
guaranteed by Theorem \ref{2.teorespectro}) and Birkhoff's ergodic
theorem ensures that $\int a(\w)\,d\wit m=0$. Therefore, for each one of these measures,
$\wit m(\W_{LY})=1$, where $\W_{LY}$ is
the set provided by Theorem \ref{3.teorSup0-chaosLY}.
Similartly, if $\wit m\in\mminv(\W,\sigma)$ and
$\wit m(\{\w\in\W\,|\;\sup_{t\le 0}\int_0^t a(\ws)\,ds<\infty\})=1$,
we have $\wit m(\W_{LY})=1$, as deduced from the decomposition of
$\wit m$ in $\sigma$-ergodic measures described in Subsection
\ref{2.subsecbasic}.
In some cases, $\W_{LY}$ is a set of complete measure:
see Theorem \ref{2.teornovacio}(ii).
\par
These properties show the physical observability, both in time
and state, of the type of Li-Yorke chaos that we detect on the global attractor.
\par
Coming now to the Auslander-Yorke chaos detected in almost all
(or all) set $\mS_c=\Supp\mu_c$, Theorem~\ref{2.teorAYsec3}(iii)
shows that $\mS_c$ contains a $\tau$-invariant subset $\mX_c$
with full measure $\mu_c$ composed of $\mu_c$-generic points with
dense forward $\tau$-semiorbits.
%Birkhoff's ergodic theorem
%and the ergodicity of $\mu_c$ allow us to assume that all the
%points of $\mX_c$ are generic (that is,
%$\lim_{t\to\infty}(1/t)\int_0^t f(\tau(s,\w_0,x_0))\,ds=
%\int_\mA f(\w,x)\,d\mu_c$
%for every $f\in C(\mA,\R)$ and $(\w_0,x_0)\in\mX_c$).
Since the orbit of a generic point is composed by generic
points, the orbit of each point of $\mX_c$ provides a
dense subset of $\mS_c$ of generic points.
The natural extension of periodic point
for autonomous or time-periodic systems to non periodic ones
is that of generic point. Hence, as indicated
in \cite{glwe4}, the type of
chaos detected on the sets $\mS_c$ extends the classical one
of \cite{deva} (which requires transitivity, sensitivity,
and the existence of a dense set of periodic points).
\par
Besides this,
as Theorem \ref{3.teorSup0-chaosAY}(iv) shows,
the union $\wit\mS$ of all these possibly Auslander-Chaotic
sets (perhaps one of them is not): is composed by sensitive
(not Lyapunov-stable: see Remark \ref{2.notasuno}.2) points;
although it is not transitive, it is chain recurrent, which
according to \cite[Theorem A]{frse} (see also \cite{conl})
ensures that it is an
{\em abstract omega limit set} (that is, $(\wit\mS,\tau)$ is
topologically conjugate to the restriction of a flow
on a compact space to one of its omega limit sets);
it is the support
of a $\tau$-invariant measure; and it has a dense subset of
$\tau$-generic points. One can
also consider these facts enough to talk about a certain type of
chaos on $\mS$, again opposed to the idea of stability, and
again related to the idea of \cite{deva}.
Finally, $\wit\mS$ fills an
\lq\lq important part"~of $\mA$. More precisely,
$\wit\mS_\w=\mA_\w$ in a $\sigma$-invariant residual set of points with full
measure $m$: the set $\W_c\cup\W_{AY}$. This property shows that also
this chaotic phenomenon has physical relevance. Observe also that
$(\wit\mS,\tau)$ is Li-Yorke chaotic in measure, since for every
$\w$ in the set $\W_{LY}\cap\W_{AY}$ (with full measure $m$),
$\wit\mS_\w=\mA_\w$, and hence $\{\w\}\times\wit\mS_\w$ is
a scrambled set: see Theorems \ref{3.teorSup0-chaosLY}(i)
and \ref{3.teorSup0-chaosAY}(iv).
\par
Let us finally recall that there are functions in $C(\W,\R)$
which satisfy the hypotheses required on $a$ in Theorems \ref{3.teorSup0-chaosLY}
and \ref{3.teorSup0-chaosAY} (namely, $\sup\Sigma_a=0$ and $a\in\mR_m(\W)$
for a measure $m\in\mmerg(\W,\sigma)$), and that
the set of functions $a$ satisfying these two conditions
coincides with that of the functions
$a$ such that $\sup\Sigma_a=0$ and $a\in\mR_m(\W)$ for a measure
$m\in\mminv(\W,\sigma)$. Theorem \ref{2.teornovacio} proves these assertions.
\par
We complete the paper with an easy extension of \cite[Corollary 4.5]{gljk},
which refers to a quasiperiodically forced map
$f\colon\S^1\!\times\![a,b]\to\S^1\!\times\![a,b]$ inducing the discrete semiflow $(\S^1\!\times\![a,b],\phi)$ given by $\phi(n,\w,x):=f^n(\w,x)$.
The authors establish the sensitivity of
$(\S^1\!\times\![a,b],\phi)$ under certain conditions which the next result
adapts to out setting.
\begin{prop}
Assume that $g$ satisfies \hyperlink{g1}{\bf g1}, \hyperlink{g2}{\bf g2}
and \hyperlink{g3}{\bf g3}, let $(\W\times\R,\tau)$ be the flow induced
by the family \eqref{3.familia}, and let $\mA$, $\alfaa$ and $\betaa$
be provided by Theorem~{\rm\ref{3.teorexisteA}}. Assume also that $\mA$ is
pinched, and that the semicontinuous functions $\alfaa,\betaa\colon\W\to\R$ are
not continuous. Then, the flow $(\W\times\R,\tau)$ is sensitive.
\end{prop}
\begin{proof}
The result is trivial if $(\W,\sigma)$ is sensitive. So, there
is not restriction in assuming that this is not the case, which according to
Corollary \ref{2.coroAY} and Definition \ref{2.defiAY}
means that $(\W,\sigma)$ is equicontinuous.
\par
Let $\ww$ be a continuity point for $\betaa$.
Given $(\w_0,x_0)\in\W\times\R$ with $x_0\ge\betaa(\w_0)$,
we look for $(t_n)\uparrow\infty$ with
$\ww=\lim_{n\to\infty}\w_0{\cdot}t_n$ and such that there exists $\wit x:=
\lim_{n\to\infty}x(t_n,\w_0,x_0)\ge\betaa(\ww)$.
Then $(\ww,\wit x)\in\mO_\tau(\w_0,x_0)
\subseteq\mA$, and hence $\wit x=\betaa(\ww)$. This ensures that
$\inf_{t\ge 0}|x(t,\w_0,x_0)-\betaa(\w_0{\cdot}t)|=0$.
The arguments of \cite[Lemma 4.4]{gljk}, which can be adapted
to our setting thanks to the equicontinuity of the base flow,
provide $\ep_{\betaa}>0$ such that all the points $(\w,x)$
above $\mA$, i.e., with $x>\betaa(\w)$, are $\ep_{\betaa}$-sensitive
(see Remark \ref{2.notasuno}.2).
An analogous argument provides $\ep_{\alfaa}>0$ such that all the points
$(\w,x)$ below $x<\alfaa(\w)$, are $\ep_{\alfaa}$-sensitive. Given
$\ep:=\min(\ep_{\alfaa},\ep_{\betaa})$, we define
$\mT_\ep\subseteq\W\times\R$ as the set of points $(\w,x)$ such that
for any $\delta>0$ there exist two points
$(\w_1,x_1),\,(\w_2,x_2)\in\mB_{\W\times\R}((\w,x),\delta)$
%with $\dist_{\W\times\R}((\w_i,x_i),(\w,x))<\delta$ for $i=1,2$
such that
$\sup_{t\ge 0}\dist_{\W\times\R}(\tau(t,\w_1,x_1),\tau(t,\w_2,x_2))>\ep$.
It is easy to check that $\mT_\ep$ is closed and contains all
the $\ep$-sensitive points.
Therefore, $(\W\times\R)-\mA\subset\mT_\ep$, and hence
$(\W\times\R)-\mT_\ep\subset\mA$. But the unique open set
contained in a pinched set is the empty one, so that
$\mT_\ep=\W\times\R$. The proof is completed by checking that any
point in $\mT_\ep$ is $\ep/2\,$-sensitive.
\end{proof}
Observe that if the attractor $\mA$ is a pinched set, then $\alfaa$
(or $\betaa$) is continuous if and only if the unique
$\tau$-minimal set $\mM$ is given by its graph. Consequently,
if the base flow $(\W,\sigma)$ is equicontinuous and if
$\mA$ is pinched, then the flow $(\W\times\R,\tau)$ is sensitive at least
in these two cases:
\begin{itemize}
\item[-] $\mM$ is not a copy of the base;
\item[-] $r_1<r_2$, $\mM\subset\W\times(r_1,r_2)$, and
$\sup_{t\le 0}\int_0^t a(\w_0{\cdot}s)\,ds<\infty$ for a point
$\w_0\in\W$: as seen in the
proof of Theorem \ref{3.teorWc}(ii), in this case the points $(\w,\alfaa(\w))$ and
$(\w,\betaa(\w))$ do not belong to $\mM$ whenever $\w$ belongs to the nonempty
set $\W-\W_c$.
\end{itemize}
\medskip\par\noindent
{\bf Acknowledgement:} Massimo Tarallo died while we were beginning to discuss some
of the results included in this article. His contribution is fundamental in this work.
His friendship, irreplaceable.
%%%%%%%%%%%%%%%%%%%%%%%%%%%%%%%%%%%%%%%%%%%%%%%%%%%%%%%%%%%%%%%%%%%%%%
%%%%%%%%%%%%%%%%%%%%%%%%%%%%%%%%%%%%%%%%%%%%%%%%%%%%%%%%%%%%%%%%%%%%%%
%%%%%%%%%%%%%%%%%%%%%%%%%%%%%%%%%%%%%%%%%%%%%%%%%%%%%%%%%%%%%%%%%%%%%%
%%%%%%%%%%%%%%%%%%%%%%%%%%%%%%%%%%%%%%%%%%%%%%%%%%%%%%%%%%%%%
%%%%%%%%%

\end{document}